\documentclass[reqno]{amsart}

\usepackage{mathrsfs}
\usepackage{amsmath}
\usepackage{amssymb}
\usepackage{cite}
\usepackage{latexsym}
\usepackage{graphicx}

\usepackage{color}
\usepackage[normalem]{ulem}
\usepackage{hyperref}
\usepackage{comment}

\theoremstyle{plain}
\newtheorem{thm}{Theorem}[section]
\newtheorem{lem}[thm]{Lemma}
\newtheorem{cor}[thm]{Corollary}
\newtheorem{prop}[thm]{Proposition}

\theoremstyle{definition}
\newtheorem{dfn}[thm]{Definition}
\newtheorem{rmk}[thm]{Remark}
\newtheorem{exm}[thm]{Example}

\newcommand{\red}[1]{\textcolor{red}{#1}}

\def\D{\mathrm{D}}

\def\R{\mathrm{R}}

\def\d{\mathrm{d}}

\def\Nset{\mathbb{N}}

\def\Qset{\mathbb{Q}}
\def\Rset{\mathbb{R}}
\def\Sset{\mathbb{S}}
\def\Tset{\mathbb{T}}
\def\Zset{\mathbb{Z}}

\def\M{\mathscr{M}}

\def\I{\mathscr{I}}

\def\h{\mathrm{h}}

\DeclareMathOperator{\sech}{sech}

\DeclareMathOperator{\sn}{sn}
\DeclareMathOperator{\cn}{cn}

\DeclareMathOperator{\rank}{rank}

\DeclareRobustCommand{\erase}{\bgroup\markoverwith{\textcolor{red}{\rule[.5ex]{2pt}{0.4pt}}}\ULon}

\def\epsilon{\varepsilon}

\makeatletter
\@addtoreset{equation}{section}
\@namedef{subjclassname@2020}{\textup{2020} Mathematics Subject Classification}
\makeatother

\def\theequation{\arabic{section}.\arabic{equation}}

\begin{document}


\title[Partial Breakdown of Integrability  in Nearly Integrable Systems]%
{Partial Breakdown of Integrability in Nearly Integrable Systems}
\thanks{}

\author[M.Gidea]{Marian Gidea}
\address{Department of Mathematical Sciences, Yeshiva University, New York, 10016, NY, USA}
\email{marian.gidea@yu.edu}

\author[S.Motonaga]{Shoya Motonaga}
\address{Department of Artificial Intelligence, 680-4 Kawazu, Iizuka, Fukuoka 820-8502, Japan}
\email{motonaga.shoya640@mail.kyutech.jp}

\date{\today}
 
\begin{abstract}
We study the persistence of integrable structures in real-analytic nearly integrable systems under perturbations that may destroy some, but not all, of these structures.
Specifically, we establish upper bounds on the numbers of real-analytic first integrals -- quantities conserved by the motion -- and commuting vector fields --  infinitesimal generators of continuous symmetries -- that can persist under the perturbation. 
These bounds are expressed in terms of Melnikov-type functions, which quantify the breakdown of the corresponding integrable structures.
Unlike classical Melnikov theory, our conditions occur even in the absence of zeros of the Melnikov functions, so  they do not necessarily imply the existence of homoclinic tangles.
We give a geometric interpretation of these conditions in the context of the Arnold diffusion problem, showing that they imply micro-diffusion.
We  also obtain a criterion for higher-dimensional nearly integrable systems in which the dimension of the span of the Melnikov vectors restricts the numbers of possible first integrals and commuting vector fields. This shows that the geometry of the Melnikov vectors provides a new mechanism for detecting the breakdown of integrability. We give several applications to integrable Hamiltonian systems subject to forcing and dissipation.
\end{abstract}
\maketitle


\section{Introduction}

The study of nearly integrable systems is a fundamental problem in dynamical systems, dating back to Poincar\'e's pioneering work on the restricted three-body problem \cite{P90,P92}. At its core, it is concerned with understanding how integrability persists or breaks down under perturbations and how the loss of integrability gives rise to the coexistence of regular and chaotic dynamics.

One of the most influential approaches for detecting the onset of chaotic dynamics in nearly integrable systems is \emph{Melnikov theory}, also  referred to as \emph{Poincar\'e--Melnikov theory} \cite{M63,W90,GH83,delshams2000splitting,GL18}. The basic idea is to estimate the first-order  term in the perturbative expansion of the distance between stable and unstable manifolds, which is represented by the Melnikov function (or, in higher-dimensional settings, a Melnikov map).
If this function has a simple zero, then transverse intersections of stable and unstable manifolds occur, resulting in chaotic dynamics. A related perturbative framework is the subharmonic Melnikov method, which is used to detect the persistence and bifurcation of periodic orbits under perturbation. 
 
Poincar\'e's work on non-integrability subsequently evolved into a broader program for determining the Liouville integrability of Hamiltonian systems. This led to modern approaches based on variational equations, which provide powerful obstructions to integrability. Ziglin developed a criterion based on the monodromy group of variational equations along particular solutions \cite{Ziglin82, Ziglin83}.
Later, Morales and Ramis generalized and strengthened Ziglin's approach by replacing the monodromy group with the differential Galois group of the variational equations \cite{MR01a, MR01b, MRS07}. These methods have become standard tools in contemporary integrability analysis.

Another line of development, stemming from Poincar\'e's perturbative approach, relies on Fourier expansions in action-angle coordinates, and was later systematized and extended by Kozlov \cite{K96}. Recently, a connection between Kozlov's integrability criteria and Melnikov functions was established by the second author and his collaborator for one-and-a-half-degree-of-freedom systems \cite{MY23}.

Although these approaches differ substantially in methodology and objectives, they share a common feature: they are primarily qualitative. They determine whether a given structure persists or breaks down, whether splitting occurs, or whether complete integrability is obstructed. By contrast, quantitative questions remain far less understood. If a perturbation destroys integrability only partially, how much of the original integrable structure survives?
More precisely, given a nearly integrable system, how many first integrals and commuting vector fields can persist analytically with respect to the perturbation parameter?

Addressing this question requires going beyond the classical problem of proving non-integrability. Rather than establishing the complete breakdown of integrability, we seek to quantify the extent to which integrable structures persist under perturbations.
This distinction is essential. Traditional obstruction theories typically provide a binary answer: either integrability persists or it is destroyed. However, to understand intermediate dynamical regimes -- where some conservation laws and symmetries survive while others break down --  we need to develop  methods capable not only of proving non-integrability but also of estimating the number of persistent integrals  and  commuting vector fields.
This viewpoint becomes essential in higher-dimensional settings, where partial persistence phenomena are unavoidable and integrability should be regarded as a graded rather than binary property.

In this paper, we establish  upper bounds on the numbers of first integrals and commuting vector fields in nearly integrable systems by means of Melnikov-type integrals.
In our setting, these are given by definite integrals of the first-order perturbation along periodic or quasiperiodic orbits of the unperturbed system, and they depend on both the initial condition and the integration time.

The class of systems considered here consists of perturbations of integrable systems in the sense of Bogoyavlenskij \cite{B98}. This is an extension of  Liouville integrability from Hamiltonian systems to autonomous vector fields of general type by requiring the existence of sufficiently many first integrals and commuting vector fields.
Unlike Liouville integrability, the commuting vector fields are not necessarily generated by the first integrals.  See Examples \ref{exm:exm_0} and \ref{exm:exm_1}.
Remarkably, a version of the Liouville-Arnold theorem holds for Bogoyavlenskij systems, asserting that the phase space is locally foliated by invariant tori on which the dynamics is conjugate to a rigid rotation.

Our approach provides a quantitative method for estimating the number of surviving conservation laws and symmetries rather than merely determining whether integrability persists. Furthermore, the theory applies not only to Hamiltonian systems but also to dissipative systems.

The key new ingredient for obtaining quantitative estimates is the notion of \emph{Melnikov 
functions/vectors with unbounded directions}, which provides the geometric mechanism for bounding the number of persistent first integrals and commuting vector fields.

We give a geometric interpretation of the conditions on Melnikov functions/vectors with unbounded directions in the spirit of the Arnold diffusion problem, showing that these conditions imply micro-diffusion.

To illustrate our approach, we consider rotator-pendulum systems perturbed by quasi-periodic vector fields that are not necessarily Hamiltonian.
We prove that the dimension of the span of  the homoclinic Melnikov vectors imposes restrictions on the number of integrable structures that can persist. Surprisingly, these restrictions remain effective even when the Melnikov vectors/functions have no simple zeros.
Therefore, the obstruction mechanism identified here does not arise from transverse intersections of stable and unstable manifolds themselves, but rather from the geometric properties of the family of Melnikov vectors/functions, in particular from the number of unbounded directions.

Our results reveal a new role for Melnikov vectors/functions in higher-dimensional, nearly integrable systems: beyond detecting chaotic dynamics, they provide quantitative information on the remaining degree of integrability.

The structure of the paper is as follows. 

In Section \ref{sec:setting} we describe the type of systems that we consider, and we recall the concept of Bogoyavlenskij integrability and the corresponding version of the Arnold-Liouville theorem. 

In Section \ref{sec:main} we describe the assumptions on our systems and state the main results of the paper.  In particular, we introduce a new concept of unbounded directions for  multi-dimensional maps or for parametrized families of such maps, which plays a crucial role in our main results. We also recall the notion of a key set (also known as a uniqueness set) for a family of analytic functions. 

In Section \ref{sec:microdiffusion} we provide a geometric interpretation of the conditions that appear in our main results. 

In Section \ref{sec:pf} we give the proofs of the main results. We also give an alternative definition of the concept of $s$ unbounded directions for (parametrized families of)  multi-dimensional maps, as well as a  sufficient condition for having $s$ unbounded directions. 

In Section \ref{sec:connections_melnikov} we establish connections between our main results and the Melnikov method. First, in the case of one-degree-of-freedom Hamiltonian systems, we recall some results that provide non-integrability conditions in terms of subharmonic and homoclinic Melnikov integrals. Second, we extend these results to the higher dimensional case. 
For pendulum-rotator systems subject to general perturbations, we show that the number of unbounded directions of the homoclinic Melnikov vectors determines upper bounds on the number of analytic first integrals and commuting vector fields.
We emphasize that our conditions on the Melnikov vectors are fundamentally different from the classical conditions, which require the existence of non-degenerate zeros.
In particular, our conditions do not imply transversality of the stable and unstable manifolds and the emergence of homoclinic tangles.
The main ingredient of the proof is an approximation of homoclinic Melnikov vectors by subharmonic Melnikov vectors. 

In Section \ref{sec:application_partially_conformal}, we show that, under some general conditions, nearly integrable, partially conformally symplectic systems with $m$ degrees of freedom can possess at most $(m-1)$ analytic first integrals. The notion of a partially conformally symplectic system is recalled in Appendix 
\ref{sec:partially_conformal}.  This is a generalization of conformally symplectic systems (systems with dissipation proportional to velocity), in the sense that it allows for different dissipation coefficients acting on different components. 

In Section \ref{sec:systems_penduli}, we apply the main theorems to a coupled system of two pendulums with dissipation and periodic forcing, and to a system of $N$ pendulums with nearest-neighbor interactions, obtaining  bounds on the number of first integrals and commutative vector fields in both cases. 

\section{Setting}\label{sec:setting}

In this section, we recall the definition of Bogoyavlenskij integrability, the corresponding version of the Arnold-Liouville theorem, and several illustrative examples. 

Let $\M$ be an $n$-dimensional real-analytic manifold
 for $n\ge 2$.
Consider dynamical systems of the form
\begin{align}\label{sys}
 \dot{x}=X_\varepsilon(x), \quad x\in \mathscr{B},
\end{align}
where $\epsilon\in (-\epsilon_0,\epsilon_0)$ is a small parameter with $0<\epsilon_0\ll1$, $\mathscr{B}$ is some domain in $\M$,
 and the vector field $X_\varepsilon$ on $\mathscr{B}$ is analytic with respect to $x$ and $\epsilon$.

Let \begin{equation}
\label{eqn:X_expansion}
X_\varepsilon(x)=X^0(x)+\varepsilon X^1(x)+\epsilon^2 X^2(x,\epsilon)
\end{equation}
 for $\epsilon\ge 0$ sufficiently small.
The system \eqref{sys} becomes
\begin{align}\label{sys0}
\dot{x}=X^0(x)
\end{align}
when $\varepsilon=0$, and \eqref{sys} is regarded as a perturbation of \eqref{sys0}.

In this paper we are interested in the integrability in Bogoyavlenskij sense (defined below) of nearly integrable systems of the form \eqref{sys} for  all $\epsilon>0$ sufficiently small.

\begin{dfn}[Bogoyavlenskij]\label{dfn:integrability}
The $n$-dimensional system \eqref{sys0} is called \emph{$(q,n-q)$-integrable} for some positive integer $q\le n$ or simply \emph{integrable} 
 if there exist $q$ vector fields $Y_1(x)(:= X_0(x)),Y_2(x),\dots,Y_{q}(x)$
 and $n-q$ scalar-valued functions $F_1(x),\dots,F_{n-q}(x)$
 such that the following two conditions hold:
\begin{enumerate}
\setlength{\leftskip}{-1.8em}
\item[\rm(i)]
$Y_1(x),\dots,Y_q(x)$ are linearly independent almost everywhere and commute with each other,
 i.e., $[Y_j,Y_k](x)\equiv 0$ for $j,k=1,\ldots,q$, where $[\cdot,\cdot]$ denotes the Lie bracket;
\item[\rm(ii)]
$F_1(x),\dots, F_{n-q}(x)$ are functionally independent, i.e.,
 $dF_1(x), \dots, dF_{n-q}(x)$ are linearly independent almost everywhere,
 and $F_1(x),\dots,F_{n-q}(x)$ are first integrals of $Y_1, \dots,Y_q$,
 i.e., $dF_k(Y_j)=0$ for $j=1,\ldots,q$ and $k=1,\ldots,n-q$.
\end{enumerate}
If $Y_1,Y_2,\dots,Y_q$ and $F_1,\dots,F_{n-q}$ are analytic, 
 then Eq.~\eqref{sys0} is said to be \emph{analytically integrable}.
\end{dfn}
Since an $m$-degree-of-freedom Liouville-integrable Hamiltonian system with $m\in\Nset$
 has not only $m$ functionally independent first integrals
 but also $m$ linearly independent commutative (Hamiltonian) vector fields
 generated by the first integrals, Definition~\ref{dfn:integrability} is considered as a generalization of 
 Liouville-integrability for Hamiltonian systems \cite{A89}.
 Moreover, it is known that Bogoyavlenskij-integrable systems have the action-angle coordinates in which the system is transformed into a simpler system. Here we state its precise assumptions and statements. Assume the following two conditions on the unperturbed system \eqref{sys0}:
\begin{enumerate}
\setlength{\leftskip}{-1em}
\item[\bf(A1)]
 The system \eqref{sys0} is analytically $(q,n-q)$ Bogoyavlenskij-integrable for some positive integer $q< n$ with pair-wise commutative analytic vector fields $Y_1(x)(:=X^0(x)),\ldots,Y_q(x)$ and their common  $n-q$ analytic first integrals $F_1(x),\ldots,F_{n-q}(x)$, i.e., conditions~(i) and (ii) of Definition~\ref{dfn:integrability} hold for $Y_1(x), \ldots,Y_q(x)$ and $F_1(x),\ldots,F_{n-q}(x)$.
\item[\bf(A2)]
Let $F(x)=(F_1(x),\ldots,F_{n-q}(x))$. 
There exists a regular value $c_0\in\Rset^{n-q}$ of $F$,
 i.e., $\rank dF(x)=n-q$ when $F(x)=c_0$,
 such that $F^{-1}(c_0)$ has a connected and compact component
 and $Y_1(x), \ldots, Y_q(x)$ are linearly independent on $F^{-1}(c_0)$.
\end{enumerate}
Note that the level set $F^{-1}(c)$ is \emph{regular}
 if $c\in\Rset^{n-q}$ is a regular value of $F$.
Henceforth we assume without loss of generality that $F^{-1}(c)$ is connected and compact itself
 in (A2), by reducing the domain of $F$ if necessary.
Now we state an extension of the Liouville-Arnold
 theorem for Hamiltonian systems \cite{A89} 
 (see \cite{B98, Z06,Z18} for the details).
 \begin{prop}[Liouville-Mineur-Arnold-Jost]\label{LAJ}
Suppose that assumptions~(A1) and (A2) hold.
Then we have the following:
\begin{enumerate}
\setlength{\leftskip}{-2em}
\item[\rm(i)]
The level set $F^{-1}(c_0)$ is analytically diffeomorphic to the $q$-dimensional torus
$\Tset^q=\prod_{j=1}^q\Sset^1$ with $\Sset^1=\Rset/2\pi\Zset$;
\item[(ii)]
There exists an analytic diffeomorphism $\varphi:U\times\Tset^q 
 \to\mathcal{U}$, where $U$ and $\mathcal{U}$ are, respectively,
 neighborhoods of $I=I_0$ in $\Rset^{n-q}$ and of $F^{-1}(c_0)$ in $\M$ for some $I_0\in\Rset^{n-q}$,
 such that 
\begin{enumerate}
\setlength{\leftskip}{-1.8em}
\item[(iia)]
$\varphi(\{I_0\}\times\Tset^{q})=F^{-1}(c_0)$;
\item[(iib)]
$F\circ\varphi(I, \theta)$ depends only on $I$,
 where $(I, \theta)\in U\times\Tset^q$;
\item[(iic)]
The vector field $(\varphi^{-1})_*X^0$ on $U\times \Tset^q$
 is of the form
\begin{equation}
\dot{I}=0,\quad\dot{\theta}=\omega(I)
\label{eqn:aasys0}
\end{equation}
 on $U\times \Tset^q$,
where $\omega:U\to\Rset^q$ is analytic.
\end{enumerate}
\end{enumerate}
\end{prop}
The coordinates $(I,\theta)$ are called the \emph{action-angle coordinates},
 and $\omega(I)$ is referred to as the \emph{angular frequency vector}.
 Throughout the paper,  we identify $I\in U$, via the coordinate change $\varphi$, with the corresponding $c\in\mathcal{U}$.
 
 We remark that the solution of Eq.~\eqref{eqn:aasys0} is
 $(I(t),\theta(t))=(c,\omega(c)t+\tau)$ with arbitrary constant $(c,\tau)\in(U,\Tset^{q})$, which provides an $n$ parameter family of quasi-periodic orbits (including periodic ones) $\{\gamma_\tau^{c}\}_{(c,\tau)\in U\times\Tset^q}$ of Eq.~\eqref{sys0} given by \[\gamma_\tau^{c}(t)=\varphi(c,\omega(c)t+\tau).\]

\begin{rmk}\label{rmk:Y} In particular, each orbit $\gamma_\tau^c$ is an integral curve of the vector field 
$\sum_{i=1}^{q} \omega_i(c)Y_i $. By replacing $Y_i$ with another basis of the same abelian Lie algebra, we can assume that $Y_i=\frac{\partial}{\partial\theta_i}$, $i=1,\ldots,q$.
\end{rmk}

\begin{rmk}\label{rmk:complete}
Note that we do not assume that the manifold $\mathscr{M}$ is compact.
While the unperturbed vector field $X^0$ is complete 
on the regular level sets of $F$, the perturbed vector field $X_\varepsilon$
does not need to be complete.
\end{rmk}

\begin{rmk}\label{rmk:lift}
    Let $T^*\mathscr{M}$ be the cotangent bundle of $\mathscr{M}$ and  $\pi:T^*\mathscr{M}\to\mathscr{M}$ be its natural projection.
	 Let $\Omega_0=d\lambda$ be the canonical  symplectic form on  $T^*\mathscr{M}$, where $\lambda:T^*\mathscr{M}\to T^*(T^*\mathscr{M})$ is the Poincar\'{e}-Liouville one-form  given by
	\[
		\lambda_z=p(d\pi_{(x,p)}(\cdot)),\quad (x,p)\in T^*\mathscr{M}.
	\]
	For each vector field $X$ on $\mathscr{M}$, the \emph{cotangent lift} of $X$ is the Hamiltonian vector field $\hat{X}$ with the Hamiltonian
	\[
		h_X(x,p)=p(X(x)),\quad (x,p)\in T^*\mathscr{M}
	\]
	on the symplectic manifold $(T^*\mathscr{M},\Omega_0)$.
    Then it is easy to see that for two vector fields $X,Z$ on $\mathscr{M}$ it holds that $h_{[X,Z]}=\{h_X,h_Z\}$
    (see Lemma~3.6 in \cite{MY21}).
    In particular, if $X$ commutes with $Z$, then $h_Z$ is a first integral of the Hamiltonian $h_X$ on $T^*\mathscr{M}$.
    In other words, commutative vector fields can be regarded as first integrals of the lifted vector field.
    As it turns out, if a system is integerable in the Bogoyavlenskij sense, then its cotangent lift is integrable in the Liouville sense \cite{AZ10}. 
    However, the converse is not necessarily true. 
\end{rmk}

\begin{exm}\label{exm:exm_0}
A simple example of a system that is Bogoyavlenskij-integrable but not Liouville is the following.
Let $H_0:\mathscr{M}\to\mathbb{R}$ be Hamiltonian on a $(2m)$-dimensional symplectic manifold $(\mathscr{M}, \Omega)$,  where $\Omega$  denotes the symplectic form.  Assume that $H_0$ is Liouville-integrable, with first integrals $(F_1,\ldots,F_m)$.
Consider a quasi-periodic perturbations of $H_0$:
\[H_\epsilon(x)=H_0(x)+\epsilon H_1(x, \phi)\]
where $\phi\in \mathbb{T}^d$ is given by
\[\dot\phi=\nu\]
where $\nu\in  \mathbb{R}^d$ is a rationally independent vector.
We consider the extended system
\begin{equation}\label{eqn:exm0} \begin{split}
\dot x=&X_{H_0}(x)+\epsilon X_{H_1}(x,\phi),\\
\dot \phi=&\nu,
\end{split}
\end{equation}
where $(x,\phi)\in\mathscr{M}\times \mathbb{T}^d$ and $H_1:\mathscr{M}\times \mathbb{T}^d\to\mathbb{R}$. Above, $X_{H_0}$ and $X_{H_1}$ denote the Hamiltonian vector fields associated to $H_0,H_1$ respectively. Note that $d$ can be an odd integer.

When $\epsilon=0$, the system is Bogoyavlenskij $(m+d,m)$-integrable, with vector fields $X_{F_1}, \ldots, X_{F_m}, \nu_1\frac{\partial}{\partial \phi_1},\ldots, \nu_d\frac{\partial}{\partial \phi_d}$,  and first integrals $F_1,\ldots, F_m$. 
\end{exm}

\begin{exm}\label{exm:exm_1}
Another example of a system that is Bogoyavlenskij-integrable but not  necessarily Liouville-integrable a priori is given by the Euler’s hydrodynamics equations for an incompressible fluid
\begin{equation*}
  \frac{\partial V}{\partial t}=V \times \textrm{curl} (V) -\nabla \left (p+\frac{\|V\|^2}{2}\right),\quad \textrm{div} (V)=0,
\end{equation*}
where
$V(x,t)$  is the velocity field of the fluid, $p(x,t)$ is the pressure, and $x\in\Rset^3$.
Assume that the vector field $V$ is axially symmetric, so  $V$ expressed in cylindrical coordinates $(z,r,\phi)$ is of the form
\[V(z,r,t)=V_z(z,r,t)\frac{\partial}{\partial z}+V_r(z,r,t)\frac{\partial}{\partial r}+V_\phi (z,r,t) \frac{\partial}{\partial \phi}.\]
The motion of the fluid in extended coordinates $(z,r,\phi,t)$ is given by the equations
\begin{equation} \label{eqn:exm_2}
\frac{dz}{dt}=V_z, \quad  \frac{dr}{dt}=V_r,\quad \frac{d\phi}{dt}= V_\phi,\quad \frac{dt}{dt}=1.\end{equation}
If  $F=r^2 V_\phi(z,r,t)$ is not constant, then the system \eqref{eqn:exm_2} is Bogoyavlenskij $(3,1)$-integrable, with vector fields
$\textrm{curl}(V)$, $\frac{\partial}{\partial \phi}$, $V+\frac{\partial}{\partial t}$ and first integral $F=r^2 V_\phi$.
See \cite{B98} for details.
\end{exm}
 
\section{Main results}\label{sec:main}

In this section, we present the main results concerning the number of first integrals  and commuting vector fields
-- depending real-analytically on both the variables and the parameter -- that the perturbed system \eqref{sys} can admit for all sufficiently small $|\epsilon|$.
We formulate these results for both general, nearly integrable systems (in the sense of Bogoyavlenskij) and Hamiltonian systems.


\subsection{The general case}
We formulate our main results for nearly integrable systems of the form \eqref{sys} under assumptions (A1) and (A2).
By Proposition~\ref{LAJ}, we can rewrite the system \eqref{sys0} in action-angle coordinates \eqref{eqn:aasys0}.
For our main results, we need an additional assumption on \eqref{eqn:aasys0}:
\begin{enumerate}
\setlength{\leftskip}{-1em}
\item[\bf(A3)]
If $r\in\Zset^q$ and $r\cdot \omega(I)=0$ for any $I\in U$, then $r=0$ holds.
\end{enumerate}
\begin{rmk}\label{rmk:(A3)}
If there exists a $q$-tuple $(I^{(1)},\ldots,I^{(q)})\in U^q$ such that $\omega(I^{(1)}),\ldots, \omega(I^{(q)})$ are $\Rset$-linearly independent, then (A3) holds. In particular, if there exists $I^*\in U$ such that ${\rm rank} \ d_{I}\omega(I^*)=q$, then (A3) holds.
\end{rmk}
Note that (A3) is characterized by a condition without the action-angle coordinates (See \cite{MY23} for the proof):
\begin{prop}
	Condition (A3) is equivalent to the following:
	\begin{enumerate}
\setlength{\leftskip}{-0.5em}
\item[\bf(A3')]
There exists a dense subset $D$ of a neighborhood of the regular value $c_0\in\Rset^{n-q}$ such that for each $c\in D$ the level set $F^{-1}(c)$ contains an orbit of \eqref{sys0} which is dense in $F^{-1}(c)$.
\end{enumerate}
\end{prop}
The following assumption is needed when we consider existence of commutative vector fields.
\begin{enumerate}
\setlength{\leftskip}{-1em}
\item[\bf(A4)]
For some $I^*\in U$, ${\rm rank} \ d_{I}\omega(I^*)=n-q$ holds.
\end{enumerate}
We remark that (A4) does not hold when $n-q>q$.

For each $c\in U$, we define the Melnikov-type functions (Melnikov functions, for short) $\mathscr{A}_{F_j}^c: \Tset^q\times\Rset\to\Rset\ (j=1,\ldots,n-q)$ given by
\begin{equation}
	\mathscr{A}_{F_j}^c(\tau,T)=\int_0^T dF_j(X^1)_{\gamma_\tau^c(t)}dt
\label{eqn:melnikov_function}
\end{equation}
and the Melnikov-type vectors (Melnikov vectors, for short)
\begin{equation}
\mathscr{A}_{F}^c(\tau,T):=(\mathscr{A}_{F_1}^c(\tau,T),\ldots,\mathscr{A}_{F_{n-q}}^c(\tau,T)).
\label{eqn:melnikov_vector}
\end{equation}

Note that $\mathscr{A}_{F}^c(\tau,T)$ is analytic in $(c,T,\tau)$ since it is an integration with parameters and the integrant is analytic in $(c,T,\tau)$.

Now we introduce a new concept of unbounded directions for multi-dimensional maps, which plays a crucial role in our main results. 
\begin{dfn}\label{dfn:unbdd_A}
	Let $k\in\Nset$ and $s\in\{1,\ldots,k\}$. We say a map $f=(f_1,\ldots,f_k):\Tset^q\times \Rset\to \Rset^k$ \emph{has $s$ unbounded directions}
	if $s$ is the largest positive integer with the following property:
	there exist  $s$-components $f_{i_1},\ldots, f_{i_s}$ of $f$ such that, for any non-trivial $a=(a_1,\ldots,a_s)\in\Rset^s\setminus\{0\}$, the linear combination $\sum_{j=1}^{s}a_j f_{i_j}(\tau,t)$ is unbounded in $t$ for some $\tau=\tau_a\in\Tset^q$, or, equivalently,
	\begin{align}\label{eq:unbdd}
	a_1f_{i_1}+\cdots+a_l f_{i_s} \text{ is bounded for all }\tau\in\Tset^q
	\Longrightarrow a_1=\ldots=a_s=0.
	\end{align}
\end{dfn}

The following alternative description can be useful.
We define an equivalence relation over functions $g:\mathbb{R} \to \mathbb{R}$:
\[ g_1\sim g_2 \Leftrightarrow t\mapsto (g_1-g_2)(t) \textrm{ is a bounded function}.\]
We denote by $\hat{g}$ the equivalence class of $g$. If $g$ is a bounded function then $\hat g=0$.
Then \eqref{eq:unbdd} is equivalent to the usual linear independence condition:
\[ a_1\hat{f}_{i_1}(\tau, \cdot)+\dots+a_s \hat{f}_{i_s}(\tau, \cdot) =0\  \text{for all }\tau\in\Tset^q
 \Longrightarrow a_1=\ldots=a_s=0 .\]

We define the subsets $\mathscr{Q}_{1}\supset\mathscr{Q}_{2}\supset\ldots \supset\mathscr{Q}_{n-q}$ of $U$ as
\begin{align*}
\mathscr{Q}_s:=\{c\in U; \mathscr{A}_{F}^c\ \text{has $s$ unbounded directions}\},
\end{align*}
for $s=1,\ldots,n-q$.
In particular, 
\begin{align*}
\mathscr{Q}_1=\{c\in U; \mathscr{A}_{F}^c(\tau, T)\ \text{is unbounded in $T$ for some} \ \tau\in\Tset^q\}.
\end{align*}
We will give an equivalent condition and a sufficient condition of having unbounded directions (Definition~\ref{dfn:unbdd_A}) in Sections~\ref{sec:pf} and ~\ref{sec:sufficient_condition} respectively.

\begin{dfn}\label{dfn:key}  A subset $\Delta\subset U$ is said to be a \emph{key set} for the family of analytic functions $C^\omega(U)$ on $U$ if any analytic function vanishing on $\Delta$ also vanishes on $U$. 
\end{dfn}

For example, an open subset of $U$ and a dense set in $U$ are key sets for $C^\omega(U)$.

We now state our main results. 

\begin{thm}\label{thm:main1}
	Assume that (A1), (A2) and (A3) hold.
	Let $s\in\{1,\ldots,n-q\}$. If $\mathscr{Q}_s$ is a key set for $C^\omega(U)$,
	then the perturbed system $X_\varepsilon$ has at most $n-q-s$
	real-analytic first integrals in $(x,\varepsilon)$ in a neighborhood of $F^{-1}(c)$
	such that they are functionally independent in $x$ for $|\varepsilon|\neq0$.
\end{thm}

\begin{rmk}
The key set condition for $\mathscr{Q}_s$ is essential, as the existence of infinitely many elements in $\mathscr{Q}_s$ itself does not prevent integrability of the perturbed systems. See Appendix~\ref{app:example}.
\end{rmk}

We next extend Definition~\ref{dfn:unbdd_A} to a family of maps.
\begin{dfn}\label{dfn:unbdd_A_family}
	Let $k\in\Nset$ and $s\in\{1,\ldots,k\}$. We say a (possibly uncountable) family $\{f_\lambda\}_{\lambda\in\Lambda}$ of maps $f_\lambda=(f_{1,\lambda},\ldots,f_{k,\lambda}):\Tset^q\times \Rset\to \Rset^k$ \emph{has $s$ unbounded directions}
	if $s$ is the largest positive integer with the following property:
	there exists $s$-components $f_{i_1,\lambda},\ldots, f_{i_s,\lambda}$ of $f_{\lambda}$ such that for any non-trivial $a=(a_1,\ldots,a_s)\in\Rset^s\setminus\{0\}$ the linear combination $\sum_{j=1}^{s}a_j f_{i_j,\lambda}(\tau,t)$ is unbounded in $t$ for some $(\lambda,\tau)=(\lambda_a,\tau_a)\in\Lambda\times \Tset^q$, or, equivalently,
	\begin{align*}
	a_1f_{i_1,\lambda}(\tau,t)+\cdots+a_s f_{i_s,\lambda}(\tau,t)&\  \text{is bounded in $t$ for all }(\lambda,\tau)\in\Lambda\times \Tset^q\nonumber\\
	&\Longrightarrow a_1=\ldots=a_s=0.
	\end{align*}
\end{dfn}
Note that Definition~\ref{dfn:unbdd_A_family} reduces to Definition~\ref{dfn:unbdd_A} when $\Lambda$ is a singleton.

We introduce the subsets $\mathscr{Q}'_{1}\supset\mathscr{Q}'_{2}\supset\ldots \supset\mathscr{Q}'_{q-1}$ of $U$ as
\begin{align*}
\mathscr{Q}'_{s'}:=\left\{c\in U; \{d_\tau \mathscr{A}_{F_i}^c\}_{i\in\{1,\ldots,n-q\}}\ \text{has $s'$ unbounded directions}\right\},
\end{align*}
for $s'=1,\ldots,q-1$, where 
\begin{equation}\label{eqn:Melnikov_family}
d_\tau \mathscr{A}_{F_i}^c=\left(\frac{\partial \mathscr{A}_{F_i}^c}{\partial \tau_1}, \ldots, \frac{\partial \mathscr{A}_{F_i}^c}{\partial \tau_q}\right).
\end{equation}
The set of parameters $\lambda\in\Lambda$ in this case is given by $i\in\{1,\ldots,n-q\}$.

\begin{thm}\label{thm:main2}
	Assume that (A1), (A2), (A3) and (A4) hold.
	Let $s'\in\{1,\ldots,q-1\}$. If $\mathscr{Q}'_s$ is a key set for $C^\omega(U)$, then for $|\varepsilon|\neq 0$ the perturbed system $X_\varepsilon$ has at most $q-s'$ real-analytic commutative vector fields in $(x,\varepsilon)$ which are linearly independent near $\varepsilon=0$.
\end{thm}
\begin{rmk}
\label{rmk:Z_1}  
Note that \eqref{sys} has at least one commutative vector field since $X_\varepsilon$ commutes with itself.  If $Z^\varepsilon_1,\ldots,Z^\varepsilon_{q-s'+1}$  denote the commutative vector fields of the perturbed system, we adopt the convention that $Z^\varepsilon_1=X_\varepsilon$.
Therefore, $\{d_\tau \mathscr{A}_{F_i}^c\}_{i\in\{1,\ldots,n-q\}}$ cannot have $q$ unbounded directions, since $\frac{\partial}{\partial_{\tau_1}}\mathscr{A}_{F_i}^c$ is bounded. We will show this in Remark~\ref{rmk:q_unbdd_direc}.
\end{rmk}


\begin{rmk}
	Our results  extend the previous results in \cite{MY23}. 
    In particular, we do not require the assumption (A3) from \cite{MY23}, which amounts to the existence of a family of periodic orbits parametrized by a key set of $I$'s. 
   In contrast, here we can consider the case that the unperturbed system has no periodic orbit.
	Moreover, Theorems~\ref{thm:main1} and \ref{thm:main2} enable us to obtain information  not only on the integrability of Eq.~\eqref{sys} but also on the numbers of first integrals and commutative vector fields of \eqref{sys}.
	In this sense, these quantities measure the dynamical complexity of the perturbed system \eqref{sys}.
\end{rmk}
 
\begin{rmk}[Sharpness of main results]\label{rmk:sharp}
The upper bounds in Theorems~\ref{thm:main1} and \ref{thm:main2} are sharp, i.e.,
there are examples for which $n-q-s$ functionally independent first integrals (resp. $q-s'$ linearly independent commutative vector fields) persist.
See Remarks~\ref{rmk:sharp2} and ~\ref{rmk:penduli_sharp}, and the cases (ii) and (iii) of Proposition~\ref{prop:fi_pendla}.
\end{rmk}

\subsection{The Hamiltonian case}\label{sec:Hamiltonian_systems}

We next consider the case when ~\eqref{sys} is a Hamiltonian system.
Let $m\ge 2$ be an integer and
 let $(\M,\Omega)$ be a $2m$-dimensional analytic symplectic manifold
 with a symplectic form $\Omega$.
 Consider an analytic Hamiltonian $H_\epsilon(x)=H^0(x)+\epsilon H^1(x)+ \epsilon^2 H^2(x,\epsilon)$
 for \eqref{sys}
 depending analytically on $\epsilon$ near $\epsilon=0$.
We assume the following:
\begin {enumerate}
\setlength{\leftskip}{-0.6em}
\item[(A1')]
The unperturbed Hamiltonian system \eqref{sys} with the Hamiltonian $H^0(x)$
 is real-analytically Liouville-integrable, i.e.,
 there exist $m$ functionally independent analytic first integrals $F_1(x)(:=H^0(x )),F_2(x),\ldots,F_m(x)$
 such that they are pair-wise Poisson commutative,
 i.e., $\{F_j,F_k\}=0\ (j,k=1,\ldots,m)$,
 where $\{\cdot,\cdot\}$ denotes the Poisson bracket for the symplectic form $\Omega$.
\end{enumerate}
Note that (A1') corresponds to (A1) for $q=m$ and $Y_j=X_{F_j}$ ($j=1, \ldots, m$),
 where $X_{F_j}$ denotes the Hamiltonian vector field for the Hamiltonian $F_j(x)$.
Then Theorem~\ref{thm:main1} provides an obstruction to Liouville integrability for nearly integrable Hamiltonian systems. 
\begin{cor}\label{cor:Hsys}
Suppose that (A1'), (A2) and (A3) hold.
If $\mathscr{Q}_1$ is a key set for $C^\omega(U)$ with $X_1=X_{H^1}$,
 then for $|\epsilon|\neq0$ sufficiently small, the perturbed Hamiltonian $H_\epsilon$
 is not real-analytically Liouville-integrable near $F^{-1}(c)$
 with $m-1$ additional first integrals depending analytically on $\epsilon$.
\end{cor}
We remark that $\mathscr{A}_{F_1}^c(\tau,T)(=\mathscr{A}_{H^0}^c(\tau,T))$ is always bounded on $T$ for all $(c,\tau)\in U\times \Tset^q$
since
\[
	\mathscr{A}_{H^0}^c(\tau,T)=\int_0^T X_{H^1}(H^0)_{\gamma^c_\tau(t)}dt
	=-\int_0^T X_{H^0}(H^1)_{\gamma^c_\tau(t)}dt=H^1(\gamma^c_\tau(0))-H^1(\gamma^c_\tau(T)).
\]

\subsection{Geometric interpretation}\label{sec:microdiffusion}

In this section, we give a geometric interpretation of Theorem \ref{thm:main1}. We will see that the condition that the Melnikov vectors defined in Section~\ref{sec:main} have an unbounded direction leads to the existence of certain drifting orbits, a phenomenon reminiscent of Arnold diffusion. We recall that the Arnold diffusion problem asserts that small, generic perturbations of integrable systems have orbits along which the  `action variable' changes by $O(1)$ with respect to the perturbation parameter (for a recent survey, see \cite{gidea2020general}). 

More precisely, we will show that,  under the assumption that some component of our Melnikov vectors is unbounded, we obtain either micro-diffusion or Arnold diffusion.
The former occurs in the case that the solutions of the perturbed systems remain confined to compact sets, as is typically the case. In this case, there exist orbits of the perturbed system undergo a change in the first integral $F_j$ by $O(\varepsilon)$ for some $j\in \{1,\ldots,n-q\}$, where the corresponding constant can be chosen arbitrary large.
Since  we only obtain a small change of a first integral (or of its derivative) along orbits, we refer  to this as \emph{micro-diffusion}. (This terminology is borrowed from \cite{bounemoura2016note}).
The latter occurs in the case that the solutions of the perturbed systems are not confined to some compact sets. In this case, there exist arbitrarily small $\varepsilon>0$ and corresponding  orbits of the  perturbed system which undergo a change in $F_j$ by $O(1)$, for some $j\in \{1,\ldots,n-q\}$.

We now state our precise statement.

\begin{prop}\label{cor:microdiffusion}
Let  $c$  be a regular value  of $F=(F_1,\ldots,F_{n-q})$. We have the following: 
\begin{itemize}
\item[(i)] Assume that  there exists a compact neighborhood $K$ of  $F^{-1}(c)$  such that for every  $\varepsilon \in (0,\epsilon_0)$ the solution $x_\varepsilon (t)$  of \eqref{sys} with $x_\varepsilon (0)=\gamma^c_\tau(0)$ stays in $K$ for all $t$. In this case we have:

\begin{itemize}
\item[(i.a)]
If, for some $j\in\{1,\ldots,n-q\}$, there exists $\tau$ such that $\mathscr{A}_{F_j}^c(\tau,T)$ is unbounded in $T$, 
then,  for every $C>0$  there exists (arbitrary large) $T>0$ and (arbitrary small) $\varepsilon^*\in(0,\varepsilon_0)$  (depending on $T$)  such  that  for every $\varepsilon\in(0,\epsilon^*)$
the  solution $x_\varepsilon (t)$ of \eqref{sys} with $x_\varepsilon (0)=\gamma^c_\tau(0)$ satisfies
\[|F_j(x_\varepsilon (T))-F_j(x_\varepsilon (0))|>C \varepsilon.\]

\item[(i.b)]
If, for some $j \in \{1, \ldots, n-q\}$ the vector-valued map $d_\tau A_{F_j}^c=\left(\frac{\partial \mathscr{A}_{F_j}^c}{\partial \tau_1}, \ldots, \frac{\partial \mathscr{A}_{F_j}^c}{\partial \tau_q}\right)$ has $q-1$ unbounded directions, then there exists $\tau$ such that
 for every $C>0$  there exists (arbitrary large) $T>0$ and (arbitrary small) $\varepsilon^*\in(0,\varepsilon_0)$  (depending on $T$) such  that  for every $\varepsilon\in(0,\epsilon^*)$
the  solution $x_\varepsilon (t)$ of \eqref{sys} with $x_\varepsilon (0)=\gamma^c_\tau(0)$ satisfies
\[\left\vert\frac{d}{dt}F_j(x_\varepsilon (T))-\frac {d}{dt} F_j(x_\varepsilon (0))\right\vert>C \varepsilon .\]
\end{itemize}

\item [(ii)] If the assumption (i) is not satisfied, then  for any compact neighborhood $K$ of  $F^{-1}(c)$  and any $\varepsilon^*\in(0,\varepsilon_0)$ there exists $\varepsilon\in (0,\varepsilon^*)$ such that the solution $x_\varepsilon (t)$ of \eqref{sys} with $x_\varepsilon (0)=\gamma^c_\tau(0)$ leaves  $K$. Therefore, there is some $j\in\{1,\ldots,n-q\}$ such that  $F_j(x_\varepsilon (t))$ changes by $O(1)$ over time.
\end{itemize}
\end{prop}

Proposition~\ref{cor:microdiffusion}  shows that micro-diffusion in the perturbed system can be detected through changes in the first integrals $F_1,\ldots,F_{n-q}$ of the unperturbed system. The next results show that the commuting vector fields of the unperturbed system play a similar role. We will use Remark~\ref{rmk:lift} (the fact that the commuting vector field give rise to first integrals in the cotangent lift) and the notation defined there.



\begin{prop}\label{cor:microdiffusion2}
Let  $c$  be a regular value  of $F=(F_1,\ldots,F_{n-q})$,  and let $h_{Y_i}$ denote the Hamiltonian generating the cotangent lift of $Y_i$ to $T^*\mathscr{M}$, for $i=1,\ldots,q$.

We have the following: 
\begin{itemize}
\item[(i)] Assume that, for some $j\in\{1,\ldots,n-q\}$, there exists a compact neighborhood $\hat{K}_j$ of 
$\{(x,dF_j(x))\in T^*\mathscr{M}; x\in F^{-1}(c)\}$ such that for every  $\varepsilon \in(0,\varepsilon _0)$ the solution $\hat{x}_\varepsilon (t)$  of $\dot{\hat{x}}_\epsilon=\hat{X}_\epsilon(\hat{x})$ with $\hat{x}_\varepsilon (0)=(\gamma^c_\tau(0),dF_j(\gamma^c_\tau(0)))$ stays in $\hat{K}_j$ for all $t$. In this case we have:

\begin{itemize}
\item[(i.c)]
If, for some $i\in\{2,\ldots,q\}$, there exists $\tau$ such that $\frac{\partial}{\partial \tau_i}\mathscr{A}_{F_j}^c(\tau,T)$ is unbounded in $T$, 
then,  for every $C>0$ there exists (arbitrary large) $T>0$ and (arbitrary small) $\varepsilon^*\in(0,\varepsilon_0)$ (depending on $T$)  such  that for every $\varepsilon\in(0,\epsilon^*)$
the  solution $\hat{x}_\varepsilon (t)$ of $\dot{\hat{x}}_\epsilon=\hat{X}_\epsilon(\hat{x})$ with $\hat{x}_\varepsilon (0)=(\gamma^c_\tau(0),dF_j(\gamma^c_\tau(0)))$ satisfies 
\[|h_{Y_i}(\hat{x}_\varepsilon (T))-h_{Y_i}(\hat{x}_\varepsilon (0))|>C \varepsilon.\] 
\end{itemize}

\item [(ii)] If the assumption (i) is not satisfied, then  for any compact neighborhood $\hat{K}_j$ of  
$\{(x,dF_j(x))\in T^*\mathscr{M}; x\in F^{-1}(c)\}$
and any $\varepsilon^*\in(0,\varepsilon_0)$ there exists $\varepsilon\in (0,\varepsilon^*)$ such that the solution $\hat{x}_\varepsilon (t)$  of $\dot{\hat{x}}_\epsilon=\hat{X}_\epsilon(\hat{x})$ with $\hat{x}_\varepsilon (0)=(\gamma^c_\tau(0),dF_j(\gamma^c_\tau(0)))$ leaves  $\hat{K}_j$. Therefore, there is some $i\in\{2,\ldots,q\}$ such that  $h_i(\hat{x}_\varepsilon (t))$ changes by $O(1)$ over time.\end{itemize}
\end{prop}

\section{Proofs of the main results}\label{sec:pf}
In this section, we present an equivalent definition  of ``maps with $s$ unbounded directions", 
and then give the proofs of the main theorems.

\subsection{Characterization of maps  with $s$ unbounded directions }
We provide an equivalent definition of ``maps with $s$ unbounded directions", which will be used directly in the proofs.
\begin{dfn}\label{dfn:unbdd_B}
	Let $k\in\Nset$ and $s\in\{1,\ldots,k\}$. We say a map $f=(f_1,\ldots,f_k):\Tset^q\times \Rset\to \Rset^k$ has $s$ unbounded directions  in another sense
	if $s$ is the largest positive integer with the following property:
	for any surjective linear map $L:\Rset^k\to\Rset^{k-s+1}$ there exists $\tau=\tau_L\in\Tset^q$ such that $t \mapsto L\circ f(\tau_L,t)\ (t\in\Rset)$ is unbounded in $t$.
\end{dfn}
It is easy to see that for $s=1$, both of  Definition~\ref{dfn:unbdd_A} and Definition~\ref{dfn:unbdd_B} mean that $f$ is unbounded for some $\tau\in\Tset^q$.
It seems not easy to understand the meaning of Definition~\ref{dfn:unbdd_B}, but this condition plays an essential role in our proofs.
We will prove that Definition~\ref{dfn:unbdd_A} and Definition~\ref{dfn:unbdd_B} are equivalent.
Note that we can verify Definition~\ref{dfn:unbdd_A} for a given map easily in applications as demonstrated in the following concrete example.
\begin{exm}\label
{exm:1}
	Fix $k=4,q=1$ and $s=3$. Consider $f:\Tset\times\Rset\to \Rset^4$ defined as $f(\tau,t)=(t,t+\sin\tau, t\sin\tau, t^2)^\top$ where $(\cdot)^\top$ stands for the transpose matrix.
	
	We first check that $f$ has $3$ unbounded directions in the sense of Definition~\ref{dfn:unbdd_A}.
	Assume that $a_1t+a_3 t\sin\tau +a_4 t^2=t(a_1+a_3\sin\tau+a_4 t)$ is bounded for all $\tau\in\Tset$.
	Then we have $a_1=a_3=a_4=0$.
	On the other hand, when $a_1=-a_2=1$ and $a_3=a_4=0$, we have $a_1t+a_2(t+\sin\tau)+a_3 t\sin\tau +a_4 t^2=-\sin\tau$\ (bounded in $t$ for all $\tau\in\Tset$), which implies $f$ has just $3$ unbounded directions in the sense of Definition~\ref{dfn:unbdd_A}.
	
	Next, we check that $f$ has $3$ unbounded directions in the sense of Definition~\ref{dfn:unbdd_B}.
	Take any (not surjective a priori) linear map $L:\Rset^4\to \Rset^{4-3+1}=\Rset^2$,
	given by $L=\begin{pmatrix} a & b & c& d\\ a' & b' & c' & d'\end{pmatrix}$.
	Then
	$
	L\circ f(\tau,t)=\begin{pmatrix} at+b(t+\sin\tau)+c\sin\tau t+d t^2\\ 
	a't+b'(t+\sin\tau)+c't\sin\tau+d't^2\end{pmatrix}.
	$
	Thus, if $L\circ f(\tau,\cdot)$ is bounded for all $\tau\in\Tset$, then we have $b=-a, c=d=0, b'=-a',c'=d'=0$, i.e., $L=\begin{pmatrix} a & -a & 0&0\\ a' & -a' & 0 &0\end{pmatrix}$ and $\mathrm{rank} (L)\le1$.
	This implies that if $L$ is surjective (equivalently, $\mathrm{rank} (L)=2$) then there exists $\tau\in\Tset$ such that $L\circ f(\tau,\cdot)$ is unbounded in $t$.
	On the other hand, if we take $s=4$, then for  $L=(1\ -1\ 0\ 0)$ we see that $L\circ f(\tau,\cdot)$ is bounded in $t$ for all $\tau\in\Tset$.
	Therefore, $f$ has $3$ unbounded directions in the sense of Definition~\ref{dfn:unbdd_B}.
\end{exm}

We now show the following important proposition:
\begin{prop}\label{prop:equivalence_unbdd}
	Let $s\in\{1,\ldots,k\}$ and $f=(f_1,\ldots,f_k):\Tset^q\times \Rset\to \Rset^k$.
	Then $f$ has $s$ unbounded directions  in the sense of Definition~\ref{dfn:unbdd_A}
	if and only if $f$ has $s$ unbounded directions  in the sense of Definition~\ref{dfn:unbdd_B}.
\end{prop}
\begin{proof}
	 We first prove Definition~\ref{dfn:unbdd_B} $\Rightarrow$ Definition~\ref{dfn:unbdd_A} by contraposition.
	  Assume that for every $s$-components $f_{i_1},\ldots, f_{i_s}$ of $f$,
	  there exists $a=(a_1,\ldots,a_s)\in\Rset^s\setminus\{0\}$ such that
	  $a_1f_{i_1}+\ldots+a_l f_{i_s}$ is bounded in $t$ for all $\tau\in\Tset^q$.
	  We want to prove that there exists a surjective linear map $L:\Rset^k\to \Rset^{k-s+1}$ such that for all $\tau\in\Tset^q$ $L\circ f(\tau,\cdot)$ is bounded in $t$.
	  
	  Take the $s$-components $f_1,\ldots,f_s$ of $f$. Then there exists $a^{(1)}=(a^{(1)}_1,\ldots, a^{(1)}_s)\in\Rset^s\setminus\{0\}$ such that
	  $a^{(1)}_1f_{1}+\ldots+a^{(1)}_s f_{s}$ is bounded in $t$ for all $\tau\in\Tset^q$.
	  By rearranging the components of $f$ among $f_1,\ldots,f_s$ we can assume that $a^{(1)}_1\neq 0$.
	  Take next $s$-components $f_2,\ldots, f_{l+1}$ of $f$.
	  Then there exists $a^{(2)}=(a^{(2)}_2,\ldots, a^{(2)}_{s+1})\in\Rset^s\setminus\{0\}$ such that
	  $a^{(2)}_2 f_{2}+\ldots+a^{(2)}_{s+1} f_{s+1}$ is bounded in $t$ for all $\tau\in\Tset^q$.
	  By rearranging the components of $f$ again among $f_2,\ldots,f_{s+1}$ we can assume that $a^{(2)}_2\neq 0$.
	  We repeat the argument $k-s+1$ times and we then define the matrix
	  \[
	  	L=\begin{pmatrix}
			a^{(1)}_1 &a^{(1)}_2& \ldots & a^{(1)}_{s}& 0 & 0 & \ldots\\
			0 &a^{(2)}_2& \ldots & a^{(2)}_{s}& a^{(2)}_{s+1} & 0 & \ldots\\
			0 &0& \ldots & \ldots& \ldots & \ldots & \ldots\\
			0 &0& \ldots & 0& a^{(k-s+1)}_{k-s+1} & \ldots & a^{(k-s+1)}_k\\
		\end{pmatrix}.
	  \]
	  Since $a^{(j)}_j\neq 0$ for $j=1,\ldots, k-s+1$, the matrix $L$ has rank $k-s+1$ and each component of $L\circ f$ is bounded in $t$ for all $\tau\in\Tset^q$. This proves our claim.

	We next prove Definition~\ref{dfn:unbdd_A}$\Rightarrow$Definition~\ref{dfn:unbdd_B}.
	Suppose that $f$ has $s$ unbounded directions  in the sense of Definition~\ref{dfn:unbdd_A}.
	Without loss of generality, rearranging the components of $f$, we can assume that $i_1=k-s+1, i_2=k-s+2,\ldots, i_s=k$ in Definition~\ref{dfn:unbdd_A}, that is, for any non-trivial $(a_1,\ldots,a_s)\in\Rset^s\setminus\{0\}$, the linear combination $a_1f_{k-s+1}+a_2 f_{k-s+2}+\ldots+a_s f_k$ is unbounded for some $\tau=\tau_a\in\Tset^q$.
	
	Assume that there exists a surjective linear map $L:\Rset^k\to\Rset^{k-s+1}$ given by $L(v)=Av\ (v\in\Rset^k)$
	such that for all $\tau\in\Tset^q$ the functions $L\circ f(\tau,\cdot)$ are bounded,
	 where $A$ is a $(k-s+1)\times k$-matrix with $\mathrm{rank} (A)=k-s+1$.
	 Then, by the elementary row operations, there exist an invertible $(k-s+1)\times (k-s+1)$-matrix $P$ such that $PA$ is in row canonical form.
	 Note that the bottom row $b=(b_1,\ldots, b_k)$ of $PA$ is not the zero vector and its leading entry is on the right of $(k-s)$-th column, i.e.,
	 there exists $m^*\in \{k-s+1,\ldots,k\}$ such that $b_j=0\ (j<m^*)$ and $b_{m^*}=1$.
	Moreover, we see that $PA(f(\tau,\cdot))$ is still bounded.
	 From the computation of the bottom component of $PA(f(\tau,\cdot))$,
	 we deduce that
	 \[
	 	f_{m^*}+\sum_{j>m^*}^{l} b_j f_j
	 \]
	 is bounded for all $\tau\in\Tset^q$, which contradicts to the fact that for any non-trivial $(a_1,\ldots,a_s)\in\Rset^s\setminus\{0\}$, the linear combination $a_1f_{k-s+1}+a_2 f_{k-s+2}+\ldots+a_s f_k$ is unbounded for some $\tau=\tau_a\in\Tset^q$.
	  \end{proof}

	  For the case of families of maps, we can obtain a similar correspondence to the above discussion.
\begin{dfn}\label{dfn:unbdd_B_family}
	Let $k\in\Nset$ and $s\in\{1,\ldots,k\}$. We say a (possibly uncountable) family $\{f_\lambda\}_{\lambda\in\Lambda}$ of maps $f_\lambda:\Tset^q\times\Rset\to\Rset^{k}$ has $s$ unbounded directions  in another sense
	if $s$ is the largest positive integer with the following property:
	for any surjective linear map $L:\Rset^k\to\Rset^{k-s+1}$ there exists$(\lambda,\tau)=(\lambda_L,\tau_L)\in\Lambda\times\Tset^q$ such that $t \mapsto L\circ f_{\lambda_L}(\tau_L,t)\ (t\in\Rset)$ is unbounded in $t$.
\end{dfn}
The proof of Proposition~\ref{prop:equivalence_unbdd} easily extends to the case of families of maps and we obtain the following equivalence of Definitions~\ref{dfn:unbdd_A_family} and ~\ref{dfn:unbdd_B_family}.
\begin{prop}\label{prop:equivalence_unbdd_family}
	Let $k\in\Nset$ and $s\in\{1,\ldots,k\}$.
	Let $\{f_\lambda\}_{\lambda\in\Lambda}$ be a family of maps $f_\lambda:\Tset^q\times\Rset\to\Rset^{k}$.
	Then $\{f_\lambda\}_{\lambda\in\Lambda}$ has $s$ unbounded directions  in the sense of Definition~\ref{dfn:unbdd_A_family}
	if and only if  $\{f_\lambda\}_{\lambda\in\Lambda}$ has $s$ unbounded directions  in the sense of Definition~\ref{dfn:unbdd_B_family}.
\end{prop}

\subsection{Proof of Theorem~\ref{thm:main1}}\label{subsec:pf-main1}

We first prove Theorem~\ref{thm:main1} under (A1), (A2) and (A3).
Fix $s\in\{1,\ldots,n-q\}$ and suppose that $\mathscr{Q}_s$ is a key set for $C^\omega(U)$.

Assume that the perturbed system $X_\varepsilon$ has $n-q-s+1$ analytic first integrals such that
they are functionally independent for $|\varepsilon|\neq0$ sufficiently small but not at $\varepsilon=0$
and depend analytically on $\varepsilon$ near $\varepsilon=0$.
By the analogous argument of Ziglin's Lemma (Proposition 3.6 in \cite{MY23}),
there exist $n-q-s+1$ analytic first integrals $G^\varepsilon_1,\ldots,G^\varepsilon_{n-q-s+1}$ which are functionally independent and depend analytically on $\varepsilon$ near $\varepsilon=0$.
This implies $dG^0_1, \ldots, dG^0_{n-q-s+1}$ are linearly independent almost everywhere.
Moreover, by (A3) and Lemma 3.4 in \cite{MY23}, we see that $G^0:=(G^0_1,\ldots,G^0_{n-q-s+1})$ is of the form
\[
G^0=\psi\circ F
\]
for some analytic map $\psi:F(\mathcal{U})\to\Rset^{n-q-s+1}$.
Note that $\mathrm{rank} \ d\psi=n-q-s+1$ almost everywhere since $dG^0_1, \ldots, dG^0_{n-q-s+1}$ are linearly independent almost everywhere.

We now detect non-persistence of first integrals in the following way.
Since $G_j^{\epsilon}$ is a real-analytic function depending on $\epsilon$ analytically,
it holds that $\frac{d}{d\epsilon}G_j^\epsilon$ is also a real-analytic function depending on
$\epsilon$ analytically.
Taking any $\epsilon\in[-\epsilon_0/2,\epsilon/2]$ and $c\in U$, we see that
$\frac{d}{d\varepsilon}\Big|_{\varepsilon=0}\left(G_j^\varepsilon(\gamma_\tau^{c}(T))-G_j^\varepsilon(\gamma_\tau^{c}(0))\right)$
is bounded in $(\tau,T)$ for $j=1,\ldots,n-q-s+1$.
Here $\gamma_\tau^{c}$ is an integral curve of the unperturbed system and it does not depend on the perturbation parameter $\epsilon$.

	Letting $G^\varepsilon=G^0+\varepsilon \tilde{G}^\varepsilon$, we have
	\begin{align*}
		dG^\varepsilon(X^0)&=-\varepsilon dG^\varepsilon (X^1+\varepsilon  X^2)\\
		&=-\varepsilon dG^0 (X^1)-\varepsilon^2\{ dG^\varepsilon(X^2)+d\tilde{G}^\varepsilon(X^1)\}
	\end{align*}
	since each component of $G^\varepsilon$ is a first integral of $X_\varepsilon$.
	Therefore, we compute
	\begin{align*}
		\frac{d}{d\varepsilon}&\Big|_{\varepsilon=0}\left(G_j^\varepsilon(\gamma_\tau^{c}(t))-G_j^\varepsilon(\gamma_\tau^{c}(0))\right)
		=\frac{d}{d\varepsilon}\Big|_{\varepsilon=0}\left(\int_0^T dG_j^\varepsilon(X^0)_{\gamma_\tau^{c}(t)}dt\right)\\
		&=\frac{d}{d\varepsilon}\Big|_{\varepsilon=0} \left( -\varepsilon\int_0^T dG_j^0(X^1)_{\gamma_\tau^{c}(t)}dt
		-\varepsilon^2\int_0^T \{dG_j^\varepsilon(X^2)+d\tilde{G}_j^\varepsilon(X^1)\}_{\gamma_\tau^{c}(t)}dt\right)\\
		&=-\int_0^T dG_j^0(X^1)_{\gamma_\tau^{c}(t)}dt
	\end{align*}
	Thus we deduce that
	\begin{align*}
	\mathscr{A}_{G_j^0}^{c}(\tau,T)=\int_0^T dG_j^0(X^1)_{\gamma_\tau^{c}(t)}dt
	\end{align*}
	is bounded in $(\tau,T)$ for each $j=1,\ldots,n-q-s+1$ and ${c}\in U$.

	Since $F_1,\ldots,F_{n-q}$ are common first integrals of $Y_1,\ldots,Y_q$,
	it follows that $F(\gamma_\tau^c(t))=F(\gamma_0^c(0))$.
	By using $G=\psi \circ F$, we obtain
	\begin{align*}
		\mathscr{A}_{G^0}^{c}(\tau,T)&=\int_0^T dG^0(X^1)_{\gamma_\tau^{c}(t)}dt
		=\int_0^T d\psi_{F(\gamma_\tau^{c}(t))}dF(X^1)_{\gamma_\tau^{c}(t)}dt\\
		&=d\psi_{F(\gamma_0^{c}(0))}\int_0^T dF(X^1)_{\gamma_\tau^{c}(t)}dt
		=d\psi_{F(\gamma_0^{c}(0))} \mathscr{A}_{F}^c(\tau,T).
	\end{align*}
	Thus $d\psi_{F(\gamma_0^{c}(0))} \mathscr{A}_{F}^{c}(\tau,T)$ is bounded on $(\tau,T)$ for each ${c}\in U$.

On the other hand, for $c\in \mathscr{Q}_s$, we see that the rank of $d\psi_{F(\gamma_0^{c}(0))}$ is less than $n-q-s+1$ since if not $d\psi_{F(\gamma_0^{c}(0))}$ is surjective and the boundedness of $d\psi_{F(\gamma_0^{c}(0))} \mathscr{A}_{F}^{c}(\tau,T)$ contradicts to the property of $\mathscr{Q}_s$.
Here, we use the equivalence of Definitions~\ref{dfn:unbdd_A} and ~\ref{dfn:unbdd_B} by Proposition~\ref{prop:equivalence_unbdd}.
	Writing $G^0=G^0(I)=(G_1^0(I),\ldots,G_{n-q-s+1}^0(I))$ in the action-angle coordinates, we see that all the $(n-q-s+1)$-th order minor determinants of $dG^0(I)$ vanish at each $c\in \mathscr{Q}_s$.
	However, by the assumption of Theorem~\ref{thm:main1}, $\mathscr{Q}_s$ in $U$ is a key set for $C^\omega(U)$,
	which implies that all the $(n-q-s+1)$-th order minor determinants of $dG^0(I)$ vanish not only on $\mathscr{Q}_s$
	but also on the whole $U$.
	This contradicts the fact that $dG_1^0,\ldots,dG_{n-q-s+1}^0$ are linearly independent almost everywhere, which completes the proof.
	
\subsection{Proof of Theorem~\ref{thm:main2}}

	We turn to the proof of Theorem~\ref{thm:main2} under the assumptions (A1), (A2), (A3) and (A4).
	In our proof, by using the cotangent lift trick presented in \cite{AZ10, MY21},
	we reduce the problem on the existence of commutative vector fields to that of first integrals
	and then we discuss non-persistence of commutative vector fields as in a similar way to Section~\ref{subsec:pf-main1}.
	
	Fix $s'\in\{1,\ldots,q-1\}$ and suppose that $\mathscr{Q}'_{s'}$ is a key set for $C^\omega(U)$.
	Assume that the perturbed system $X_\varepsilon$ has $q-s'+1$ analytic commutative vector fields $Z_1^\varepsilon, \ldots, Z_{q-s'+1}^\varepsilon$
	which are linearly independent almost everywhere and depend analytically on $\varepsilon$ near $\varepsilon=0$.
	Moreover, by Lemma 3.11 in \cite{MY23} and the assumptions (A3) and (A4), we can assume that $Z_j^0\ (j=1,\ldots,q-s'+1)$ have the form
	\begin{align*}
		Z_j^0(x)=\sum_{k=1}^q \rho_{j,k}(F(x))Y_k(x), 
	\end{align*}
	where $ \rho_{j,k}:\Rset^{n-q}\to \Rset,\ k=1,\ldots q$ are analytic.
	Note that $\mathrm{rank} \ (\rho_{j,k})=q-s'+1$ almost everywhere since $Z^0_1, \ldots, Z^0_{q-s'+1}$ and $Y_1,\ldots,Y_q$ are linearly independent almost everywhere.
	
	Now we consider non-persistence of commutative vector fields as well as that of first integrals as in Section~\ref{subsec:pf-main1}.
	By using the cotangent lift trick as in Section~\ref{sec:microdiffusion}, we see that the Hamiltonian $h_{Z^\varepsilon_j}\ (j=1,\ldots,q-s'+1)$ of the cotangent lift of the commutative vector field $Z_j^\epsilon$ is a first integral of the cotangent lift $\hat{X}_\varepsilon$ of $X_\varepsilon$.
	Moreover, fixing ${c}\in U$, we see that $\hat{\gamma}_{\tau,F_i}^{c}(t):=(\gamma_\tau^{c}(t),dF_i(\gamma_\tau^{c}(t)))$ is an integral curve of $\hat{X}_\varepsilon$.
	We replace $G_j^\varepsilon, {\gamma}_{\tau}^{c}(t)$ and $X_\varepsilon$ in Section~\ref{subsec:pf-main1} with
	$h_{Z^\varepsilon_j}, \hat{\gamma}_{\tau,F_i}^{c}(t)$ and $\hat{X}_\varepsilon$ respectively,
	and apply the same argument. Then we can deduce that
	\[
	\int_0^T dh_{Z^0_j}(\hat{X}^1)_{\hat{\gamma}_{\tau,F_i}^{c}(t)}dt
	\]
	is bounded in $(\tau,T)$ for each $c\in U$ and $i=1,\ldots,n-q,\ j=1,\ldots,q-s'+1$,
	where $\hat{X}^1$ is the leading perturbation term of $\hat{X}_\varepsilon=\hat{X}+\varepsilon \hat{X}^1+\varepsilon^2 \hat{X}^2$.
	
	Computing $dh_{Z_j}(\hat{X}^1)_{\hat{\gamma}_{\tau,F_i}^{c}(t)}$, we obtain
	\begin{align*}
	dh_{Z^0_j}(\hat{X}^1)_{\hat{\gamma}_{\tau,F_i}^{c}(t)}
	&=X^1(Z_j^0(F_j))_{\gamma_\tau^{c}(t)}-Z_j^0(X^1(F_j))_{\gamma_\tau^{c}(t)}\\
	&=-Z_j^0(X^1(F_j))_{\gamma_\tau^{c}(t)}
	=-\sum_{k=1}^q (\rho_{j,k}\circ F)Y_k(X^1(F_j))_{\gamma_\tau^{c}(t)}
	\end{align*}
	Since the $q$-parameter family of quasi-periodic orbits $\{\gamma^{c}_\tau(t)\}_{\tau\in\Tset^q}$ is generated by the commutative vector fields $Y_1,\ldots,Y_q$ corresponding to \eqref{sys0}, 
    we can assume that $\frac{\partial}{\partial \tau_k}\gamma_\tau^{c}(t)=Y_k(\gamma_\tau^{c}(t))$ (see Remark \ref{rmk:Y}).
	Thus 
	\begin{align*}
		\frac{\partial}{\partial \tau_k}\mathscr{A}_{F_i}^{c}(\tau,T)
		=\int_0^T\frac{\partial}{\partial \tau_k} dF_i(X^1)_{\gamma_\tau^{c}(t)}dt\\
		= \int_0^T Y_k(X^1(F_i))_{\gamma_\tau^{c}(t)}dt.
	\end{align*}
	Therefore, using $\rho_{j,k}\circ F(\gamma_\tau^{c}(t))=\rho_{j,k}\circ F(\gamma_0^{c}(0))$,
	it follows that
	\begin{align*}
		\int_0^T dh_{Z_j^0}(\hat{X}^1)_{\hat{\gamma}_{\tau,F_i}^{c}}dt
		&=-\sum_{k=1}^q  \int_0^T 
		\left((\rho_{j,k}\circ F)Y_k(X^1(F_i))\right)_{\gamma_\tau^{c}(t)}dt\\
		&=-\sum_{k=1}^q (\rho_{j,k}\circ F(\gamma_0^{c}(0)) \left( \int_0^T  Y_k(X^1(F_i))_{\gamma_\tau^{c}(t)}dt\right)\\
		&=-\sum_{k=1}^q (\rho_{j,k}\circ F(\gamma_0^{c}(0)) \frac{\partial \mathscr{A}_{F_i}^{c}}{\partial \tau_k}(\tau,T),
	\end{align*}
	for $i=1,\ldots n-q,\ j=1,\ldots,q-s'+1$.
	This implies that
	\[
	d_\tau \mathscr{A}_{F_i}^c P
	\]
	is bounded in $(\tau,T)$ for each $c\in U$ and $i=1,\ldots, n-q$, where $P$ is the transpose of the matrix $((\rho_{j,k}\circ F(\gamma_0^{c}(0)))$.

We now use the fact that $\{d_{\tau}\mathscr{A}_{F_i}^c\}_{i\in\{1,\ldots,n-q\}}$ has $s'$ unbounded directions for each $c\in \mathscr{Q}'_s$ as in the proof of Theorem~\ref{thm:main1}. Note that we also use Proposition~\ref{prop:equivalence_unbdd_family}.
	For $c\in \mathscr{Q}'_{s'}$, we see that the matrix $P$ cannot have the maximal rank by the definition of $\mathscr{Q}'_{s'}$. Writing $Z^0_j=Z^0_j(I)\ (j=1,\ldots,q-s'+1)$ in the action-angle coordinates, it holds that all the $(q-s'+1)$-th order minor determinants of $P$ vanish not only on $\mathscr{Q}'_{s'}$ but also on the whole $U$. 
	This yields a contradiction since $Z^0_1,\ldots,Z^0_{q-s'+1}$ are linearly independent almost everywhere.

\begin{rmk}\label{rmk:q_unbdd_direc}
The above proof implies that $\{d_\tau\mathscr{A}_{F_j}^c\}_{j=1,...,n-q}$ cannot have $q$ unbounded directions.
Actually, if we set $Z_\varepsilon=X_\varepsilon$ and $Y_1=X_0$, the above calculation shows that
$h_{X_\epsilon} (\hat{\gamma}^c_{\tau, F_j}(T))-h_{X_\epsilon} (\hat{\gamma}^c_{\tau, F_j}(0))$ is bounded in $T$ and its leading term is $-\frac{\partial}{\partial \tau_1} \mathscr{A}_{F_j}^c$.
Thus, the $q$-tuple $(a_1,a_2, ...,a_q)=(1,0,...,0)$ prevents $\{d_\tau\mathscr{A}_{F_j}^c\}_{j=1,...,n-q}$ from having $q$ unbounded directions.
Therefore, Theorem 3.8 makes sense only for the cases where $s'= 1, \ldots, q-1$.
\end{rmk}

\subsection{ Proofs of Propositions~\ref{cor:microdiffusion} and ~\ref{cor:microdiffusion2}}
 We only show Proposition~\ref{cor:microdiffusion} since Proposition~\ref{cor:microdiffusion2} can be shown in a similar way. 
Assume that  condition (i) in  Proposition~\ref{cor:microdiffusion} holds.

First, we prove  (i.a).
We compute
\begin{equation}
\begin{split}
F_j(x_\varepsilon (T))-F_j(x_\varepsilon (0))=&\int_0^T \frac{d}{dt}(F_j(x_\varepsilon (t)))\,dt =\int_0^T dF_j(X_\varepsilon (x_\varepsilon (t)))\, dt\\
=&\int_0^T dF_j\left(X^0(x_\varepsilon (t))+\varepsilon  X^1(x_\varepsilon (t))+\varepsilon ^2  X^2  (x_\varepsilon (t),\varepsilon ) \right)\, dt.
\end{split}
\label{eqn:F_expansion_1}
\end{equation}


Since $X_\varepsilon$ is analytic in $\varepsilon$, after possibly decreasing $\varepsilon_0$, we may assume that it admits an analytic extension to a complex neighborhood of $(0,\varepsilon_0)$, uniformly for $x\in K$. Since $K$ is compact, it follows that   $X^0$ and $X^1$, as well as the remainder $X^2$ 
are uniformly bounded for $x\in K$ and $\varepsilon\in (0,\varepsilon_0)$.

From \eqref{eqn:F_expansion_1}, since  $X^0(F_j)=0$ , it follows that
\begin{equation}
\begin{split}
F_j(x_\varepsilon (T))-F_j(x_\varepsilon (0))=&\varepsilon \int_0^T dF_j(X^1(x_\varepsilon (t)))\, dt+O(\varepsilon ^2) .
\end{split}
\label{eqn:F_expansion_2}
\end{equation}

Now we apply the Gronwall inequality, obtaining
\[\|x_\varepsilon (t)-\gamma^c_\tau(t)\|<\varepsilon  M_0e^{M_1 t},\]
for some $M_0,M_1>0$ that depend on $X^0$ and $X^1$. 
Then 
\begin{equation}
\begin{split} 
\left \vert \int_0^T dF_j(X^1(x_\varepsilon (t)))\,dt 
- \int_0^{T} dF_j(X^1(\gamma_\tau^c(t)))\, dt
\right \vert \le&  \varepsilon  M_2 \int_0^T M_0e^{M_1 t} \,dt\\  <&  \varepsilon  M_2\frac{M_0}{M_1} e^{M_1T},
\end{split}\label{eqn:Gronwall}
\end{equation}
where $M_2>0$ is a constant depending on $dF_jX^1$.
 
From \eqref{eqn:Gronwall} we obtain
\begin{equation}
\begin{split}
|F_j(x_\varepsilon (T))-F_j(x_\varepsilon (0))|\ge &\varepsilon \left[\left\vert \int_0^{T} dF_j(X^1(\gamma_\tau^c(t)))\, dt\right\vert -\varepsilon M_2\frac{M_0}{M_1}e^{M_1T} \right]+O(\varepsilon ^{2})\\
= &\varepsilon \left[\left\vert \mathscr{A}_{F_j}^c(\tau,T)\right\vert -\varepsilon M_2\frac{M_0}{M_1}e^{M_1T}\right] +O(\varepsilon ^2).
\end{split}
\label{eqn:F_expansion_3}
\end{equation}

By assumption,  $\mathscr{A}_{F_j}^c(\tau,T)$ is unbounded in $T$.
This implies that for any $C>0$ (independent of $\varepsilon$), there exists $T$ (arbitrarily large) such that  $|\mathscr{A}_{F_j}^c(\tau,T)|>2C$.
Choose  $0<k<\frac{1}{M_1}$ (for instance, $k=\frac{1}{2M_1}$), and denote $\rho=1-M_1k$.
Then, for arbitrarily large $T$  there exists  arbitrarily small (in absolute value) $\varepsilon \neq 0$   such that 
\begin{itemize} \item $T<k\log(1/\varepsilon )$, and
\item $|-M_2\frac{M_0}{M_1}\varepsilon ^{1+\rho}+O(\varepsilon^2)| <C\varepsilon$,  
\end{itemize}
in which case we obtain
\begin{equation*}
\begin{split}
|F_j(x_\varepsilon (T))-F_j(x_\varepsilon (0))|\ge &\varepsilon \left\vert \mathscr{A}_{F_j}^c(\tau,T)\right\vert -
M_2\frac{M_0}{M_1}\varepsilon ^{1+\rho}  +O(\varepsilon ^2)\\
> & \varepsilon C. 
\end{split}
\label{eqn:F_expansion_4}
\end{equation*}


Now we prove Proposition~\ref{cor:microdiffusion} (i.b).
As noted in  Remark \ref{rmk:Y} we  can assume that $Y_k(\gamma_\tau^{c}(t))=\frac{\partial}{\partial \tau_k}\gamma_\tau^{c}(t)$.
By a direct calculation, we have
\begin{align*}
		\frac{\partial}{\partial \tau_k}\mathscr{A}_{F_j}^{c}(\tau,T)
		=\int_0^T\frac{\partial}{\partial \tau_k} dF_j(X^1)_{\gamma_\tau^{c}(t)}dt\\
		= \int_0^T Y_k(X^1(F_j))_{\gamma_\tau^{c}(t)}dt.
	\end{align*}
Since $Y_k=\frac{\partial}{\partial \theta_k}$ and $\gamma^c_\tau(t)=(c, \tau+\omega(c) t)$, we have $\frac{\partial}{\partial \tau_k}=\frac{\partial}{\partial \theta_k}$.
\[
\begin{aligned}
\frac{d}{dt}
X^1(F_j)(\gamma_\tau^c(t))
&=
\sum_{k=1}^q
\omega_k(c)
\frac{\partial X^1(F_j)}{\partial\theta_k}
(\gamma_\tau^c(t))\\
&=
\sum_{k=1}^q
\omega_k(c)
Y_k\big(X^1(F_j)\big)(\gamma_\tau^c(t)).
\end{aligned}
\]

Therefore, by the Fundamental Theorem of Calculus,

\[
\begin{aligned}
&\sum_{k=1}^q\omega_k(c) \int_0^T
Y_k\big(X^1(F_j)\big)(\gamma_\tau^c(t))\,dt\\
&\qquad= X^1(F_j)(\gamma_\tau^c(T)) -X^1(F_j)(\gamma_\tau^c(0)).
\end{aligned}
\]

Let $x_\varepsilon(t)$ be a solution of \eqref{sys}. 
We have
\[
\frac{d}{dt}F_j(x_\epsilon(t))
=
X_\epsilon(F_j)(x_\epsilon(t)),
\]
therefore
\[
\begin{aligned}
\frac{d}{dt}F_j(x_\epsilon(t))- \frac{d}{dt}F_j(x_\epsilon(0))
&=\epsilon\left[ X^1(F_j)(x_\epsilon(t))-X^1(F_j)(x_\epsilon(0))
\right]\\
&\quad+ \epsilon^2\left[
X^2(F_j)(x_\epsilon(t),\epsilon)-
X^2(F_j)(x_\epsilon(0),\epsilon)\right].
\end{aligned}
\]
Using the Gronwall inequality for $T=k\ln (1/\varepsilon )$ as before we obtain 
\[
\begin{aligned}
\frac{d}{dt}F_j(x_\epsilon(T))- \frac{d}{dt}F_j(x_\epsilon(0))
&=\epsilon\left[\sum_{k=1}^q\omega_k(c) \int_0^T
Y_k\big(X^1(F_j)\big)(\gamma_\tau^c(t))\,dt 
\right] +O(\varepsilon^{1+\rho}) \\
&=\epsilon\left[\sum_{k=1}^q\omega_k(c) \frac{\partial}{\partial \tau_k} \mathscr{A}_{F_j}^c(\tau,T)
\right] +O(\varepsilon^{1+\rho}).
\end{aligned}
\] 
By the assumption that   $d_\tau\mathscr{A}_{F_j}^c(\tau,T)$ has 
$q-1$ unbounded directions,  for  every non-trivial linear combination of  $ \frac{\partial}{\partial \tau_2} \mathscr{A}_{F_j}^c,\ldots, \frac{\partial}{\partial \tau_q} \mathscr{A}_{F_j}^c$ there exists $\tau$ such that the respective linear combination  is unbounded. Since $\frac{\partial}{\partial \tau_1} \mathscr{A}_{F_j}^c$ is bounded (as shown in Remark \ref{rmk:q_unbdd_direc}),
 $\frac{\partial}{\partial \tau_1} \mathscr{A}_{F_j}^c$ together with any other $q-2$ components of $d_\tau\mathscr{A}_{F_j}^c(\tau,T)$ make a bounded linear combination
with $a=(1,0,\ldots, 0)$.
Therefore, it follows that there exists $\tau$ such that 
for any $C>0$, there exists  $T$ sufficiently large such that  
\[\left\vert \sum_{k=2}^q\omega_k(c) \frac{\partial}{\partial \tau_k} \mathscr{A}_{F_j}^c(\tau,T) \right\vert>3C.\]
Choosing $C$ sufficiently large it also follows that
\[\left\vert \sum_{k=1}^q\omega_k(c) \frac{\partial}{\partial \tau_k} \mathscr{A}_{F_j}^c(\tau,T) \right\vert>2C.\]
For any such  $T$ (arbitrarily large) there exists  $\varepsilon  >  0$ arbitrarily small   such that $T<k\log(1/\varepsilon )$ and the error terms are sufficiently small so that
 \begin{equation}
\begin{split}
\left\vert \frac{d}{dt}F_j(x_\epsilon(T))- \frac{d}{dt}F_j(x_\epsilon(0)) \right\vert \ge \varepsilon  \left\vert \sum_{k=1}^q\omega_k(c) \frac{\partial}{\partial \tau_k} \mathscr{A}_{F_j}^c(\tau,T) \right\vert>\varepsilon  C.
\end{split}
\label{eqn:F_expanion_5}
\end{equation}

Now we prove (ii). The statement in (ii) is the negation of the statement in (i).  
Since $K$ is a compact neighborhood of $F^{-1}(c)$, we see that  $d(\gamma^c_\tau(0),\partial K)>0$ is a quantity of $O(1)$.
Then for every $\varepsilon^*>0$  there exists $0<\varepsilon<\varepsilon^*$ such that the solution $x_\varepsilon (t)$ leaves $K$ at some point in time $T>0$. Thus,  there is some $j$ such that  $F_j(x_\varepsilon (t))$ changes by $O(1)$ during $[0,T]$.

\subsection{A sufficient condition for maps with $s$ unbounded directions}\label{sec:sufficient_condition}
In this section we describe a connection between our integrals and the previous one in \cite{MY23}.
We introduce a subset $D_\mathrm{R}\subset U$ given by
\[
	D_\mathrm{R}=\{c\in U; \dim_\Qset  \mathrm{span}_\Qset \{ \omega_1(c), \ldots, \omega_q(c) \} =1\}.
\]
Then for each $c\in D_\mathrm{R}$ there exists a positive constant $\omega_0(c)$ depending on $c\in D_\mathrm{R}$ such that $\omega(c)/\omega_0(c)\in\Zset^q\setminus\{0\}$ and $\omega(c)/(m\omega_0(c))\notin\Zset^q\setminus\{0\}$ for all integer $m\ge 2$. Moreover, for each $c\in D_\mathrm{R}$, the unperturbed system has a $q$-parameter family of periodic orbits
\begin{equation}
(I(t), \theta(t))=(c, \omega(I)t+\tau),\quad
\tau\in\Tset^q,
\label{eqn:po}
\end{equation}
 for $\tau\in\Tset^q$ with the period $T^c=2\pi/\omega_0(c)$.
Note that the periodic orbits given by \eqref{eqn:po} are parameterized
 by $(q-1)$ parameters essentially 
 since two periodic orbits $(c,\omega(c)t+\tau)$ and $(c,\omega(c)t+\tau_0)$
 represent the same orbit if $\tau-\tau_0=\omega(c)t_0$ for some $t_0\in\Rset$.
We also have a $q$-parameter (but essentially $(q-1)$-parameter) family of periodic orbits
\[
\gamma_\tau^c(t)=\varphi(c, \omega(c)t+\tau),\quad
 (c,\tau)\in D_\R\times\Tset^q,
\]
with the period $T^I$ in the unperturbed system.
Define the integrals
\begin{align}\label{eqn:i}
\I_{F_k}^c(\tau):=\int_0^{T^c} dF_j(X^1)_{\gamma_\tau^c (t)} dt,\quad
 k=1,\ldots,n-q,
\end{align}
for $c\in D_\R$ and set $\I_F^c(\tau):=(\I_{F_1}^c(\tau),\ldots,\I_{F_{n-q}}^c(\tau))$.
Note that
\[
\I_{F_k}^c(\tau+\omega(c)t)=\I_{F_k}^c(\tau)
\]
for $\tau\in\Tset^q$ and $t\in\Rset$.

\begin{prop}\label{prop:sufficient_unbounded}
	Fix $c^*\in D_\mathrm{R}$ and let $s_{c^*}=\dim_{\Rset}(\mathrm{span}_\Rset \mathscr{I}_F^{c^*}(\Tset^q))$.
	Then $\mathscr{A}_F^{c^*}$ has at least $s_{c^*}$ unbounded directions.
\end{prop}
\begin{proof}
	From Proposition~\ref{prop:equivalence_unbdd}, it suffices to show that $\mathscr{A}_F^{c^*}$ has  at least  $s_{c^*}$ unbounded directions
	 in the sense of Definition~\ref{dfn:unbdd_B}.
	 
	 It is easy to prove that for any surjective linear map $L:\Rset^{n-q}\to\Rset^{n-q-s_{c^*}+1}$ there exists $w_L\in \mathscr{I}_F^{c^*}(\Tset^q)$ such that $w_L\notin \mathrm{Ker}(L)$.
	 Actually, if not, there exists a surjective linear map $L:\Rset^{n-q}\to\Rset^{n-q-s_{c^*}+1}$ such that for all $w\in \mathscr{I}_F^{c^*}(\Tset^q)$ it holds that $w\in \mathrm{Ker}(L)$. Then we have $ \mathscr{I}_F^{c^*}(\Tset^q)\subset \mathrm{Ker}(L)$ and this implies that $\mathrm{span}_\Rset\mathscr{I}_F^{c^*}(\Tset^q)\subset \mathrm{Ker}(L)$, which contradicts $\dim_{\Rset}(\mathrm{span}_\Rset\mathscr{I}_F^{c^*}(\Tset^q)))=s_{c^*}> s_{c^*}-1=\dim\mathrm{Ker}(L)$.
	 
	Now fix a surjective linear map $L:\Rset^{n-q}\to\Rset^{n-q-s_{c^*}+1}$ and take $w_L\in \mathscr{I}_F^{c^*}(\Tset^q)$ with $w_L\notin \mathrm{Ker}(L)$, i.e., $||L(w_L)||> 0$.
	Then there exists $\tau_L\in \Tset^q$ such that $ \mathscr{I}_F^{c^*}(\tau_L)=w_L$.
	For $k\in\Nset$, we have
	\[
		||L(\mathscr{A}_F^{c^*}(\tau_{L},kT^{c^*}))||=||L(k\mathscr{I}_F^{c^*}(\tau_{L}))||=k||L(\mathscr{I}_F^{c^*}(\tau_{L}))||=k||L(w_L)||\to +\infty
	\]
	as $k\to+\infty$. This implies that $\mathscr{A}_F^{c^*}$ has at least $s_{c^*}$ unbounded directions.
\end{proof}

\begin{cor}
	Fix $c^*\in D_\mathrm{R}$ and let $s'_{c^*}=\max_{\tau'\in\Tset^q}\mathrm{rank}(d_\tau\I_{F}^{c^*}(\tau'))$.
	Then both of $\mathscr{A}_F^{c^*}$ and $\{d_\tau\mathscr{A}_{F_i}^{c^*}\}_{i=1,\ldots, {n-q}}$ have at least $s'_{c^*}$ unbounded directions. 
\end{cor}
\begin{proof}
	Since $\mathrm{rank}(d_\tau\I_{F}^c(\tau'))$ takes an integer values for each $\tau'\in\Tset^q$, there exists $\tau^*\in\Tset^q$ such that $\mathrm{rank}(d_\tau\I_{F}^{c^*}(\tau^*))=s'_{c^*}$.
	Then the rank theorem implies that the image of $\I_{F}^{c^*}$ contains an $s'_{c^*}$- dimensional open ball and thus we deduce that $\mathscr{A}_F^{c^*}$ has $s'_{c^*}$ unbounded directions from Proposition~\ref{prop:sufficient_unbounded}.
	Now we show that $\{d_\tau\mathscr{A}_{F_i}^{c^*}\}_{i=1,\ldots, {n-q}}$ also have $s'_{c^*}$ unbounded directions.
	From $s'_{c^*}=\mathrm{rank}(d_\tau\I_{F}^{c^*}(\tau^*))$,
	for any surjective linear map $L:\Rset^{n-q}\to\Rset^{n-q-s'_{c^*}+1}$,
	 there exists $w^*\in  \{d_\tau\mathscr{A}_{F_i}^{c^*}(\tau^*)\}_{i=1,\ldots,n-q}$ such that $w^*\notin \mathrm{Ker}(L)$, i.e., $||L(w^*)||> 0$ and a similar discussion as in the proof of Proposition~\ref{prop:sufficient_unbounded} provides the desired result.
	 Note that $d_\tau\I_{F_i}^{c^*}(\tau)=d_\tau\mathscr{A}_{F_i}^{c^*}(\tau,T^{c^*})$ and $d_\tau\mathscr{A}_{F_i}^{c^*}(\tau,kT^{c^*})=kd_\tau\mathscr{A}_{F_i}^{c^*}(\tau,T^{c^*})$ for every $\tau\in\Tset^q$ and $k\in\Nset$.
\end{proof}

\section{Connections with classical Melnikov methods}

In this section, we establish connections between the Melnikov functions defined in \eqref{eqn:melnikov_function} and \eqref{eqn:melnikov_vector} and the classical homoclinic and subharmonic Melnikov methods. We then apply the main results to  rotator-pendulum systems perturbed by quasi-periodic vector fields that are not necessarily Hamiltonian.

\label{sec:connections_melnikov}
\subsection{Periodically perturbed one-degree-of-freedom Hamiltonian systems}\label{sec:single_DOF}
In this subsection, we briefly review the standard Melnikov methods for subharmonic and homoclinic orbits in periodically perturbed single-degree-of-freedom Hamiltonian systems, and discuss relationships with our main results.
For this case, the present results recover the obstructions to integrability via subharmonic and homoclinic Melnikov functions already obtained in \cite{MY23},  and they do not provide additional information. However, since we are mainly interested in higher dimensional cases, it is worth describing the connection between our integral $\mathscr{A}^{c}_{F}$ defined in Section~\ref{sec:main} and the subharmonic/homoclinic Melnikov functions for this case.

Consider a periodically perturbed single-degree-of-freedom Hamiltonian system of the form
\begin{align}\label{mel}
\dot{x}=J\D H(x)+\varepsilon u(x,\nu t,\epsilon),\quad
x\in\Rset^2,
\end{align}
where $\epsilon$ is a small parameter as in the preceding sections, $\nu>0$ is a constant,
 $H:\mathbb{R}^2\to\mathbb{R}$ and $u:\Rset^2\times\Sset^1\times (-\epsilon_0,\epsilon_0)\to\Rset^2$ are analytic,
 and $J$ is the $2\times 2$ symplectic matrix,
\[
J=
\begin{pmatrix}
0 & 1\\
-1 & 0
\end{pmatrix}.
\]
Note that when $\epsilon=0$ Eq.~\eqref{mel} becomes the single-degree-of-freedom Hamiltonian system
\begin{align}\label{mel0}
\dot{x}=J\D H(x),
\end{align}
with the Hamiltonian $H(x)$.
We assume that the unperturbed system \eqref{mel0} satisfies the following conditions:
\begin {enumerate}
\setlength{\leftskip}{-0.6em}
\item[\bf(M1)]
There exists a one-parameter family of periodic orbits $x^{\alpha}(t)$
 with period $\hat{T}^{\alpha}>0$, $\alpha\in(0,1)$.
 \item[\bf(M2)]
$x^{\alpha}(t)$ is analytic with respect to $\alpha\in(0,1)$. Moreover, $d\hat{T}^{\alpha}/d\alpha>0$ holds.
\end{enumerate}
Setting $\phi=\nu t\mod 2\pi$ such that $\phi\in\Sset^1$,
we lift the non-autonomous system \eqref{sys} to
an autonomous system
\begin{equation}\label{mel-aut}
\dot{x}=J\D H(x)+\varepsilon u(x,\phi,\epsilon),\quad
\dot{\phi}=\nu.
\end{equation}
Since $H(x)$ is a first integral and $\displaystyle  Y=\frac{\partial}{\partial \phi}$ is a commutative vector field for \eqref{mel-aut} with $\varepsilon=0$,
assumptions~(A1) and (A2) in Section~\ref{sec:main} hold.

Now we fix relatively prime integers $l$ and $n$ and assume that $\alpha\in(0,1)$ satisfies
\begin{equation}
\frac{2\pi}{\hat{T}^\alpha}=\frac{n}{l}\nu.
\label{eqn:res4}
\end{equation}
We denote this $\alpha$ by $\alpha=\alpha^{l/n}$ as in \cite{GH83,W90,Y96}.
Then the \emph{subharmonic Melnikov function} is defined as
\begin{equation}
M^{l/n}(\phi)=\int_0^{2\pi l/\nu}\D H(x^{\alpha^{l/n}}(t))\cdot u(x^{\alpha^{l/n}}(t),\nu t+\phi,0)\d t.
\label{eqn:subM}
\end{equation}

It has been shown that, if $M^{l/n}(\phi)$ has a simple zero at $\phi=\phi_0$
 and $d\hat{T}^\alpha/d\alpha\neq 0$ at $\alpha=\alpha^{l/n}$, then there exists a $2\pi l/\nu$-periodic orbit of \eqref{mel-aut}
 near $(x,\phi)=(x^{\alpha^{l/n}}(t),\nu t+\phi_0)$,  for $|\epsilon|>0$ sufficiently small
  (see \cite{Y96} for the proof and \cite{GH83,W90} for related results).
We also note that, if the $2\pi l/\nu$-periodic orbit $(x(t),\phi(t))=(x^{\alpha^{l/n}}(t),\nu t+\phi_0)$ in \eqref{mel-aut} with $\varepsilon=0$ persists in the perturbed system for $|\epsilon|>0$, then $M^{l/n}(\phi_0)= 0$ holds (see Theorem 2.1 in \cite{MY21}).

Using the action angle coordinates, we can transform the single-degree-of-freedom Hamiltonian system \eqref{mel0} into the form \eqref{eqn:aasys0} with $n-q=q=1$.
Thus the lifted perturbed system \eqref{mel-aut} is transformed into the form \eqref{eqn:aasys0}
 with $q=2$ and $n-q=1$.
Here we take $I=\alpha$ as in \cite{MY23} and have $\omega(I)=(\tilde{\omega}(I),\nu)$ with
\[
\tilde{\omega}(\alpha)=\frac{2\pi}{\hat{T}^\alpha},
\]
in which case the transformed system is not Hamiltonian even when $\epsilon=0$.
Moreover, we see that (A3) and (A4)  in Section~\ref{sec:main} hold for \eqref{mel-aut}. In fact, (A4) is trivial and (A3) follows from the monotonicity of $\hat{T}^\alpha$ in (M2) and Remark~\ref{rmk:(A3)}.

Now we investigate the relationship between our integral defined in Section~\ref{sec:main} and the subharmonic Melnikov function.
For $\alpha=\alpha^{l/n}$ and $\tau=(\tau_1,\tau_2)$, let $\gamma_\tau^{\alpha^{l/n}}(t)=(x^\alpha(t+\tau_1),\nu (t+\tau_1)+\tau_2)$.
Note that $\gamma_\tau^\alpha(t)$ is a $2\pi l/\nu$-periodic orbit
 in \eqref{mel-aut} with $\epsilon=0$.
By the definition of the subharmonic Melnikov function, we easily see that for any $k\in\Nset$
\begin{align*}
	\mathscr{A}_{H}^{\alpha^{\ell/n}}(\tau,k2\pi\ell/\nu)&=\int_0^{k2\pi l/\nu}\D H(x^{\alpha^{\ell/n}}(t+\tau_1))\cdot u(x^{\alpha^{\ell/n}}(t),\nu (t+\tau_1)+\tau_2,0)\d t\\
	&=k\int_{0}^{2\pi l/\nu}\D H(x^{\alpha^{\ell/n}}(t))\cdot u(x^{\alpha^{\ell/n}}(t),\nu t+\tau_2,0)\d t\\
	&=k M^{\ell/n}(\tau_2).
\end{align*}

Thus, $\mathscr{A}_{H}^{\alpha^{\ell/n}}(\tau,T)$ is unbounded in $T$
if $M^{\ell/n}(\tau_2)$ is not identically zero.
Moreover, differentiating both sides of the above identity, we obtain
\begin{align*}
	\partial_{\tau_1} \mathscr{A}_{H}^{\alpha^{\ell/n}}(\tau,k2\pi\ell/\nu)=0,\quad
	\partial_{\tau_2} \mathscr{A}_{H}^{\alpha^{\ell/n}}(\tau,k2\pi\ell/\nu)=k\frac{d M^{\ell/n}}{d\tau_2}(\tau_2),
\end{align*}
which implies that $d_\tau \mathscr{A}_{H}^{\alpha^{\ell/n}}(\tau,T)$ is unbounded in $T$
if $M^{\ell/n}(\tau_2)$ is not constant.
Therefore, applying Theorems~\ref{thm:main1} and \ref{thm:main2}, we obtain  Theorems~5.1 and 5.2 in \cite{MY23}.
\begin{thm}[cf. Theorems 5.1 and 5.2 in \cite{MY23}]\label{thm:sub-mel}\ \ \ \\
Let $D_{\mathrm R}=\left\{\alpha=\alpha^{l/n}\in (0,1); \frac{2\pi}{\hat{T}^\alpha}=\frac{n}{l}\nu, \ l,n:\text{relatively prime}\right\}$.
\begin{itemize}
	\item[(i)]	If $
\mathscr{D}=\{\alpha^{l/n}\in D_{\mathrm R}; M^{l/n} \ \text{is not identically zero}\}$
 is a key set for $C^\omega((0,1))$,
 then for $|\epsilon|\neq 0$ sufficiently small
 the system \eqref{mel-aut} has no non-constant real-analytic first integral
  that  depends analytically on $\epsilon$ near $\epsilon=0$.
	\item[(ii)]	If $\mathscr{D}'=\{\alpha^{l/n}\in D_{\mathrm R}; M^{l/n} \ \text{is not constant}\}$ is a key set for $C^\omega((0,1))$, then for $|\epsilon|\neq 0$ sufficiently small
 the system \eqref{mel-aut} has no real-analytic, commutative vector field
 which depends analytically on $\epsilon$ and is linearly independent with the vector filed of \eqref{mel-aut} almost everywhere near $\epsilon=0$.
\end{itemize}
\end{thm}

We next consider the case of homoclinic orbits.
We additionally assume the following condition on the unperturbed system \eqref{mel0}:
\begin {enumerate}
\setlength{\leftskip}{-0.6em}
\item[\bf(M3)]
There exists a hyperbolic saddle $p$ with a homoclinic orbit $x^\h(t)$ such that
\[
\lim_{\alpha\to1}\sup_{t\in\Rset}d(x^\alpha(t),\Gamma)=0,
\]
where $\Gamma=\{x^\h(t)\mid t\in\Rset\}\cup\{p\}$
 and $d(x,\Gamma)=\inf_{y\in\Gamma}|x-y|$.
\end{enumerate}
Then the \emph{homoclinic Melnikov function} is defined as 
\begin{equation}
M(\phi)=\int_{-\infty}^{\infty}\D H(x^{\h}(t))\cdot u(x^{\h}(t),t+\phi,0)dt.
\label{eqn:homM}
\end{equation}
The standard argument in the Melnikov analysis asserts that
the homoclinic Melnikov function detects chaotic dynamics of \eqref{mel-aut}.
More precisely, for $|\epsilon|>0$ there exist transverse homoclinic orbits to a periodic orbit near $\{p\}\times\Sset^1$
 in \eqref{mel-aut}
provided $M(\phi)$ has a simple zero \cite{GH83,M63,W90}, and 
by the Smale-Birkhoff theorem \cite{GH83,W90} such transverse homoclinic orbits implies that the system \eqref{mel-aut} exhibits chaotic motions.
In addition, in this case, there is no real-analytic (additional) first integral (see, e.g., Chapter III of \cite{M73}).
Note that, even if $M(\phi)$ has a simple zero, there is the possibility of analytical integrability of \eqref{mel-aut}
 since it may have two additional linearly independent commutative vector fields.
 We also remark that the above statements provide no information of the integrability of \eqref{mel-aut} when $M(\phi)$ does not have simple zeros.

From Theorem~4.6.4 of \cite{GH83}, we have
\begin{equation}
\lim_{l\to \infty} M^{l/1}(\phi)=M(\phi)
\label{eqn:limM}
\end{equation}
for each $\phi\in \Sset^1$.
Therefore, it follows from \eqref{eqn:limM} that, if $M(\phi)$ is not identically zero (resp. constant), then
$M^{l/1}(\phi)$ is not identically zero (resp. constant)
for $l>0$ sufficiently large.

Let $\hat{U}\subset\Rset^2$ be a region such that $\partial\hat{U}\supset \Gamma$
and $\hat{U}\supset\{x^\alpha(t)\mid t\in[0,\hat{T}^\alpha)\}$
for some $\alpha\in(0,1)$.
Then, by Eq.~\eqref{eqn:limM} and Theorem~\ref{thm:sub-mel}, we obtain the following,
which is already proved in \cite{MY23}.
\begin{thm}[cf. Theorems 5.3 and 5.4 in \cite{MY23}]\label{thm:hom-mel}\ 
\begin{itemize}
	\item[(i)]	If $M(\phi)$ is not identically zero, then for $|\epsilon|\neq 0$ sufficiently small
 the system \eqref{mel-aut} has no non-constant real-analytic first integral
 in $\hat{U}\times\Sset^1$
 such that it depends analytically in $\epsilon$ near $\epsilon=0$.
	\item[(ii)]	If $M(\phi)$ is not constant, then for $|\epsilon|\neq 0$ sufficiently small
 the system \eqref{mel-aut} has no  real-analytic, commutative vector field in $\hat{U}\times\Sset^1$
 which depends analytically on $\epsilon$ and is linearly independent with the vector filed of \eqref{mel-aut} almost everywhere near $\epsilon=0$.
\end{itemize}
\end{thm}

\subsection{Periodically perturbed higher dimensional Hamiltonian systems}\label{sec:multi_DOF}
Now we turn to the higher dimensional case of Subsection~\ref{sec:single_DOF}.
In this subsection, our phase space is $\mathscr{M}=\Rset^n\times \Rset^n\times\Tset^d\times  \Rset^d$ and denote its coordinate by $x=(q,p,\theta,I)$ with $q=(q_1,\ldots,q_n)\in\Rset^n$, $p=(p_1,\ldots,p_n)\in\Rset^n$, $\theta=(\theta_1,\ldots,\theta_d)\in\Tset^d$ and $I=(I_1,\ldots,I_d)\in\Rset^d$.
Let $\displaystyle \Omega=\sum_{i=1}^n dp_i\wedge dq_i+\sum_{j=1}^d dI_j\wedge d\theta_j$ be the canonical symplectic form on $\Rset^n\times \Rset^n\times\Tset^d\times  \Rset^d$.
 Consider the autonomous rotator-pendulum Hamiltonian $H_0:\Rset^n\times \Rset^n\times\Tset^d\times  \Rset^d \to \Rset; x=(q,p,\theta, I)\mapsto H(x)$ given by
\begin{align}\label{eqn:RP}
	H_0(q,p,\theta,I)=h_0(I)+\sum_{j=1}^n G_j(q_j,p_j),
\end{align}
where $h:\Rset^d\to\Rset$ and $G_j:\Rset^2\to\Rset\ (j=1,\ldots,n)$ are analytic functions. 
The corresponding Hamiltonian system $X_{H_0}$ is given by
\begin{align}\label{eqn:vec_RP}
	\dot{x}=J_{n,d}\D_x H_0(x),
\end{align}
where $J$ is the $(n+d)\times (n+d)$ symplectic matrix
\[
	J_{n,d}=
\begin{pmatrix}
0 & I_{n}& 0 &0\\
-I_{n} & 0& 0 &0\\
 0 &0 & 0& I_{d}\\
 0 &0&-I_{d} & 0\end{pmatrix},
\]
and $I_{n}$ (resp. $I_d$) is the $n\times n$ (resp. $d\times d$) identity matrix.
Note that Eq.~\eqref{eqn:vec_RP} is a decoupled system with
\begin{align}\label{eqn:rotator}
	\dot{\theta}=\frac{\partial h}{\partial I}(I),\qquad \dot{I}=0,
\end{align}
and 
\begin{align}\label{eqn:pendulum}
	\dot{q}_j= \frac{\partial G_j}{\partial p_j}(q_j,p_j),\qquad \dot{p}_j=- \frac{\partial G_j}{\partial q_j}(q_j,p_j)
\end{align}
for $j=1,\ldots,n$.
Eq.~\eqref{eqn:rotator} (or the term $h_0(I)$ in \eqref{eqn:RP}) and \eqref{eqn:pendulum} (or the term $\sum_{i=1}^n G_j(q_j,p_j)$ in \eqref{eqn:RP}) are referred to as the rotator part and  the pendulum part of \eqref{eqn:vec_RP} respectively.
For the pendulum part \eqref{eqn:pendulum}, we assume the same conditions (M1)-(M3) in Subsection~\ref{sec:single_DOF} for $j=1,\ldots,n$. More precisely, we assume the following condition for \eqref{eqn:pendulum} with $j=1,\ldots,n$:
\begin {enumerate}
\setlength{\leftskip}{-0.6em}
\item[\bf(M1)$_j$]
There exists a one-parameter family of periodic orbits $z_j^{\alpha_j}(t)=(q_j^{\alpha_j}(t),p_j^{\alpha_j}(t))$ for \eqref{eqn:pendulum}
 with period $\hat{T}_j^{\alpha_j}>0$, $\alpha_j\in(0,1)$.
 \item[\bf(M2)$_j$]
$z_j^{\alpha_j}(t)$ is analytic with respect to $\alpha_j\in(0,1)$. Moreover, $d\hat{T_j}^{\alpha_j}/d\alpha>0$ holds.
\item[\bf(M3)$_j$]
There exists a hyperbolic saddle $z_j^*=(q_j^*,p_j^*)$ of \eqref{eqn:pendulum} with a homoclinic orbit $z_j^\h(t)=(q_j^{\h}(t),p_j^{\h}(t))$ such that
\[
\lim_{\alpha_j\to1}\sup_{t\in\Rset}d(z_j^{\alpha_j}(t),\Gamma_j)=0,
\]
where $\Gamma_j=\{z_j^\h(t)\mid t\in\Rset\}\cup\{z_j^*\}$
 and $d(z,\Gamma_j)=\inf_{y\in\Gamma_j}||z-y||$.
\end{enumerate}
Note that
\begin{align}\label{eqn:critical}
\frac{\partial G_j}{\partial q_j}(z_j^*)=0,\quad \frac{\partial G_j}{\partial p_j}(z_j^*)=0,\quad i=1,\ldots,n,
\end{align}
since $z_j^*$ is a saddle point of \eqref{eqn:pendulum}.

Moreover, for the rotator part \eqref{eqn:rotator}, we assume a local twist condition:
\begin {enumerate}
\setlength{\leftskip}{-0.6em}
\item[\bf(M4)]  $\frac{\partial^2 h}{\partial I^2}$ is non-degenerate for some $I=I^* \in \Rset^d$. 
\end{enumerate}
 We emphasize that it is not necessary to assume that $\frac{\partial^2 h}{\partial I^2}$ is positive definite. Furthermore, since $h$ is an analytic function, it follows from (M4) that, in fact, it suffices for $h$ not to be affine.
Note that under the conditions (M1)$_j$-(M3)$_j$ for $j=1,\ldots,n$ and (M4), the assumptions  (A1)-(A4) stated in Sections~\ref{sec:setting} and \ref{sec:main} are satisfied with pair-wise Poisson commutative first integrals $I_i\ (i=1,\ldots,d)$ and $G_j$
($j=1,\ldots,n$) of the unperturbed Hamiltonian $H_0$.

We now perturb the Hamiltonian system of $H_0(x)$. Let $\varepsilon_0>0$ be a constant and let $u:\Rset^n\times \Rset^n\times\Tset^d\times  \Rset^d\times \Sset^1\times (-\epsilon_0,\epsilon_0) \to \Rset^{n+d}; (x,\phi,\epsilon)=(q,p,\theta, I, \phi,\epsilon)\mapsto u(x,\phi,\epsilon)$ be analytic (time-dependent) vector field depending on a perturbation parameter $\epsilon$. Consider the following perturbed system:
\begin{align}\label{eqn:vec_RP_perturbed}
	\dot{x}=J_{n,d}\D_x H_0(x)+\epsilon u(x,t,\epsilon).
\end{align}
Note that the perturbed system is, in general, not Hamiltonian, as we also want to consider systems with dissipation.
See Sections 
\ref{sec:application_partially_conformal} and \ref{sec:systems_penduli}.

The corresponding autonomous system in the extended phase space $\Rset^n\times \Rset^n\times\Tset^d\times  \Rset^d\times \Sset^1$ is
\begin{align}\label{eqn:vec_RP_perturbed_aut}
	\dot{x}=J_{n,d}\D_x H_0(x)+\epsilon u(x,\phi,\epsilon),
	\qquad \dot{\phi}=1.
\end{align}



Consider the $n$-dimensional (homoclinic) \emph{Melnikov vector} $M=(M_1,\ldots, M_n)$ of components
\begin{equation}\label{eqn:Melnikov_vector}
\begin{split}
    &M_j(\tau,\theta, I,\phi)\\
	&\quad=-\int_{-\infty}^\infty \D_x G_j\cdot u(z_1^\h(t+\tau_1),\ldots,z_n^\h(t+\tau_n),\omega(I)t+\theta,I,
    t+\phi,0) dt,
\end{split}
\end{equation}
where $z_j^\h(t+\tau_j)=(q_j^\h(t+\tau_j), p_j^\h(t+\tau_j))$ for $j=1,\ldots,n$
and $\cdot$ represents the standard inner product in $\Rset^{n+d}$.
Letting
\[
u=(u_{q_1},\ldots, u_{q_n},u_{p_1},\ldots,u_{p_n}, u_{\theta_1},\ldots, u_{\theta_d},u_{I_1},\ldots, u_{I_d})^\top,
\]
we have
\begin{align}\label{eqn:computation}
\D_x G_j\cdot u=\frac{\partial G_j }{\partial q_j}u_{q_j}+\frac{\partial G_j}{\partial p_j} u_{p_j}.
\end{align}
From  \eqref{eqn:computation} and \eqref{eqn:critical} we have
\[
\D_x G_j\cdot u(z^*_1,\ldots,z^*_n, \omega(I)t+\theta, I, t+\phi, 0)=0 .
\]
Since the integrand in  \eqref{eqn:Melnikov_vector} converges exponentially fast to $0$, the improper integral in is absolutely convergent  and so the $n$-dimensional (homoclinic) Melnikov vector $M$ is well-defined. 



Below, we use the Melnikov vector $M$ to derive upper bounds on the number of persistent first integrals and commuting vector fields.
We will apply our results to general systems, including systems with dissipation. See Sections \ref{sec:application_partially_conformal} and \ref{sec:systems_penduli}.

We now give our main results in this setting. The proofs of the following results are given after Remark~\ref{rmk:Melnikov}.
\begin{thm}\label{thm:Melnikov1}
	Assume that (M1)$_j$-(M3)$_j$ ($j=1,\ldots,n$) and (M4) hold for the unperturbed system \eqref{eqn:vec_RP}.
	Let
	\[
	s=\max_{I'\in \Rset^d} \mathrm{rank}_\Rset \ \left(M(\Rset^n\times \Tset^d\times\{I=I'\}\times\Sset^1)\right).
	\]
	Then the perturbed system \eqref{eqn:vec_RP_perturbed_aut} in the augmented phase space has at most $n+d-s$ real-analytic first integrals in $(x,t,\epsilon)$ such that they are functionally independent in $(x,\phi)$ for $|\epsilon|\neq 0$.
\end{thm}
\begin{thm}\label{thm:Melnikov2}
	Assume that (M1)$_j$-(M3)$_j$ ($j=1,\ldots,n$) and (M4) hold for the unperturbed system \eqref{eqn:vec_RP}.
	Let
	\[
	s'=\max_{\substack{I'\in\Rset^d\\(\tau',\theta',\phi')\in\Rset^n\times\Tset^d\times \Sset^1}} \mathrm{rank}_\Rset(d_\tau M(\tau',\theta', I',\phi')).
	\]
	Then for $|\varepsilon|\neq 0$ the perturbed system \eqref{eqn:vec_RP_perturbed_aut} has at most $(n+d-s')$ real-analytic first integrals and at most $n+d-s'$ real-analytic commutative vector fields in $(x,t,\epsilon)$ such that they are functionally independent in $(x,\phi)$ and linearly independent almost everywhere in the augmented phase near $\varepsilon=0$ respectively.
\end{thm}

\begin{rmk}\label{rmk:Melnikov}
	In the classical Melnikov method, the existence of non-degenerate zeros of the Melnikov vector gives a sufficient condition for transverse intersections between the stable and unstable manifolds of a normally hyperbolic invariant manifold. Such intersections imply chaotic dynamics, which can be used to prove the non-integrability of the perturbed system.
	In the case when the perturbation $u$ includes both external forcing and dissipation, and the dissipation is sufficiently strong relative to the forcing, the Melnikov vector may not have non-degenerate zeros, and so  the classical Melnikov method may fail to establish the existence of chaotic dynamics.
	However, even in such a case, Theorem~\ref{thm:Melnikov1} (resp. Theorem~\ref{thm:Melnikov2}) enables us to determine how many  first integrals (resp., commutative vector fields) of the unperturbed system 
    can persist under the perturbation, which gives a more detailed description of the system's non-integrability.
\end{rmk}

\begin{rmk}\label{rmk:sharp2}
 The upper bounds in Theorems~\ref{thm:Melnikov1} and \ref{thm:Melnikov2} are sharp, as are those in Theorems~\ref{thm:main1} and \ref{thm:main2}.
Consider the $4n+1$-dimensional system \eqref{eqn:vec_RP_perturbed_aut} with $d=n$,
\[
H_0=\sum_{i=1}^n \frac{1}{2}I_i^2+\sum_{j=1}^n \left(\frac{p_j^2}{2}+1-\cos q_j\right),
\]
and
\[
u(x,t,\epsilon)=\sum_{j=1}^n \epsilon \mu_j \cos \theta_j \frac{\partial}{\partial p_j},
\]
where $\mu_j\ (j=1,\ldots,n)$ is a non-zero constant.
Then the perturbed system \eqref{eqn:vec_RP_perturbed_aut} has $n$ functionally independent first integrals $I_1,\ldots, I_n$
and $n+1$ linearly independent commutative vector fields $X_{I_1},\ldots, X_{I_n}$ and $\displaystyle Z=\frac{\partial}{\partial t}$.
 On the other hand, since this system is  decoupled,
the computation of of the Melnikov vector can be done componentwise, yielding
\[
M_j(\tau,\theta,I,\phi)=-2\pi\mu_j \sech\left(\frac{\pi I_j}{2}\right)\cos \left(\frac{\theta_j}{I_j}-\tau_j\right)
\]
for $I_j\neq 0\ (j=1,\ldots,n)$. 
Therefore, from Theorem~\ref{thm:Melnikov2}, we conclude that, for $|\epsilon|\neq 0$, there is no additional real-analytic first integral (resp. commutative vector field)
which is functionally independent with $I_1,\ldots,I_n$ (resp. linearly independent almost everywhere with $X_{I_1},\ldots,X_{I_n},Z$) and depending  analytically on the perturbation parameter $\epsilon$.
The sharpness of Theorems~\ref{thm:Melnikov1}, \ref{thm:Melnikov2}, and \ref{thm:main1} can be shown in a similar way.
\end{rmk}

\begin{rmk}
When the perturbation vector field $u$ is Hamiltonian, i.e., there exists an analytic function $H^1:\Rset^n\times \Rset^n\times\Tset^d\times  \Rset^d\times \Sset^1 \times (-\epsilon_0,\epsilon_0) \to \Rset; (x,\phi)=(q,p,\theta, I,\phi)\mapsto H^1(x,\phi,\epsilon)$ such that
\[
u(x,t,\epsilon)=J_{n,d}\D_x H^1(x,t,\epsilon)
\]
then 
the corresponding perturbed non-autonomous Hamiltonian system is
\begin{align}\label{eqn:vec_RP_perturbed_Hamilton}
	\dot{x}=J_{n,d}\D_x H_0(x)+\epsilon J_{n,d}\D_x H^1(x,t,\epsilon)
\end{align}
and its lifted autonomous system on $\Rset^n\times \Rset^n\times\Tset^d\times  \Rset^d\times \Sset^1$ is given by
\begin{align}\label{eqn:vec_RP_perturbed_aut_Hamilton}
	\dot{x}=J_{n,d}\D_x H_0(x)+\epsilon J_{n,d}\D_x H^1(x,\phi,\epsilon),\qquad \dot{\phi}=1.
\end{align}
Then each component of the Melnikov vector $M$ can be represented as
\begin{align*}
	&M_j(\tau,\theta,I,\phi)\\
	&\quad=-\int_{-\infty}^\infty \{G_j, H^1\}(z_1^\h(t+\tau_1),\ldots,z_n^\h(t+\tau_n),\omega(I)t+\theta,I,t+\phi,0) dt,
\end{align*}
where $\{,\}$ is the Poisson bracket associated with the canonical symplectic form $\Omega=\sum_{i=1}^n dp_i\wedge dq_i+\sum_{j=1}^d dI_j\wedge d\theta_j$.
In this case, we can introduce the \emph{Melnikov potential} $\mathcal{M}$ given by
\begin{align*}
\mathcal{M}(\tau,\theta,I,\phi)=-\int_{-\infty}^\infty\Big( H^1(z_1^\h(t+\tau_1),\ldots,z_n^\h(t+\tau_n),\omega(I)t+\theta, I,t+\phi,0)\\
\qquad-H^1(z_1^*,\ldots,z_n^*,\omega(I)t+\theta,I,t+\phi,0)\Big) dt
\end{align*}
and obtain
\begin{align}\label{eqn:gradient}
M_j(\tau,\theta, I,\phi)=\frac{\partial \mathcal{M}}{\partial \tau_j}(\tau,\theta, I,\phi),\quad j=1,\ldots,n.
\end{align}
Again, note that the Melnikov potential $\mathcal{M}$ is well-defined by the hyperbolicity of $z^*_j\ (j=1,\ldots,n)$.
From \eqref{eqn:gradient}, it holds that the non-degenerate critical points of $\mathcal{M}(\tau,\theta, I,\phi)$ in the variable $\tau$ correspond to the non-degenerate zeros of $M(\tau,\theta, I,\phi)$ in $\tau$.
\end{rmk}

We now want to prove Theorems~\ref{thm:Melnikov1}  and \ref{thm:Melnikov2}, which are higher dimensional versions of Theorem~\ref{thm:hom-mel}.
The main difficulty is to obtain an analog of \eqref{eqn:limM}. When the pendulum part has one degree of freedom, the integral along the periodic orbit is approximated by the integral along the homoclinic orbit. On the other hand, when the pendulum part has two or more degrees of freedom, considering the integral along the periodic orbits requires examining  periodic orbits of various frequency ratios; therefore, the challenge lies in determining how to relate the integral along the homoclinic orbits according to these frequency ratios. 
Our strategy is the following:
\begin{itemize}
\item[(1)] We show that, for a family of periodic orbits near the homoclinic orbit in each plane,
there exists a critical parameter value $\alpha^*_j$ such that all of periodic orbits $z_j^\alpha$ are trapped in a small neighborhood of the saddle point for a very long time for $\alpha_j\in (\alpha^*_j,1)$. (Trapping estimate)
\item[(2)] Approximate  some linear combinations  of the homoclinic Melnikov integrals {projected to some space} by subharmonic Melnikov integrals for periodic orbits with  suitable frequency ratios  near the homoclinic manifold in the $2n$-dimensional space.
\item[(3)] Extract unbounded directions from the subharmonic Melnikov integral by using the periodicity of the integrands.
\end{itemize}
\begin{lem}[Trapping estimate]\label{lem:trapping_estimate}
Let $\mathscr{U}$ a bounded domain in $\Rset^n$.
Then, for every small $\mu>0$, there exists $(T,\alpha^*)=(T_{\mathscr{U},\mu}, \alpha^*_{\mathscr{U},\mu})\in \Rset_{>0}\times (0,1)^n$ such that,
for every $\tau\in\mathscr{U}$, $\alpha_j\in(\alpha_j^*,1)$ with
$T_j^{\alpha_j}>T$, for $j=1,\ldots,n$, we have
\[
d_{\Rset^2}(z^\h_j(t+\tau_j),z^{\alpha_j}_j(t+\tau_j))<\mu\quad \text{for all} \ t\in  [-T_j^{\alpha_j}/2,T_j^{\alpha_j}/2]
\]
and
\[
	d_{\Rset^2}(z_j^{\alpha_j}(t+\tau_j),z_j^*)<\mu
	\quad \text{for all} \ t\in  [-T_j^{\alpha_j}/2,-T/2]\cup[T/2, T_j^{\alpha_j}/2].
\]
\end{lem}
\begin{proof}
Fix $\tau=(\tau_1,\ldots,\tau_n)\in\mathscr{U}(\subset \Rset^n)$.
Take arbitrary small $\mu>0$.
Let $U_j:=B_\mu(z_j^*)=\{z\in\Rset^2\mid d_{\Rset^2}(z,z_j^*)<\mu\}$ for each $j=1,\ldots, n$.
Take $T>0$ so that
\begin{align}\label{eqn:homo_saddle}
d_{\Rset^2}(z_j^\h(t),z_j^*)<\mu/3\quad \text{for all}\ |t|\ge T/2
\end{align}
and
$B_{\mu/3}(z^\h_j(\pm T/2+\tau_j))\subset U_j$ for all $j=1,\ldots, n$. Note that here we use
\[
\lim_{T\to+\infty} d_{\Rset^2}(z^\h_j(\pm T/2)),z^*_j)=0
\]
and that $T$ does not depend on the choice of $\tau\in\mathscr{U}$ since $\mathscr{U}$ is bounded.
By (M2)$_j$ and (M3)$_j$ for $j=1,\ldots,n$, there exists $\alpha^*=(\alpha_1^*,\ldots,\alpha_n^*)\in (0,1)^n$ such that
\begin{itemize}
\item[(1)] $T_j^{\alpha_j^*}>T$,
\item[(2)] $T_j^{\alpha_j}$ is monotonically increasing for $\alpha_j\in(\alpha_j^*,1)$,
\item[(3)] For all $t\in  [-T_j^{\alpha_j}/2,T_j^{\alpha_j}/2],\ \alpha\in(\alpha_j^*,1)$ and $j=1,\ldots,n$, it holds that
\begin{align}\label{eqn:homo_per}
d_{\Rset^2}(z^\h_j(t+\tau_j),z^\alpha_j(t+\tau_j))<\mu/3.
\end{align}
\end{itemize}
This implies that for each $\alpha_j \in (\alpha_j^*,1)$, it holds that
\begin{align*}
	z_j^{\alpha_j}(t+\tau_j)\in U_j
	\quad \text{for all} \ t\in  [-T_j^{\alpha_j}/2,-T/2]\cup[T/2, T_j^{\alpha_j}/2]
\end{align*}
since 
\begin{align*}
	d_{\Rset^2}(z_j^{\alpha_j}(t+\tau_j),z_j^*)&\le d_{\Rset^2}(z_j^{\alpha_j}(t+\tau_j),z_j^{\h}(t+\tau_j))+d_{\Rset^2}(z_j^{\h}(t+\tau_j),z_j^*)\\
	&\le  \mu/3+\mu/3=2\mu/3
\end{align*}
for $t\in  [-T_j^{\alpha_j}/2,-T/2]\cup[T/2, T_j^{\alpha_j}/2]$ by \eqref{eqn:homo_saddle} and \eqref{eqn:homo_per}.
\end{proof}

Let $r=(r_1,\ldots, r_n)$ be an $n$-tuple of positive integers with $r_1\le r_2\le\cdots\le r_n$.
 By $d\hat{T_j}^{\alpha_j}/d\alpha>0$ and $\displaystyle \lim_{\alpha_j\to 0} \hat{T}_j^{\alpha_j}=+\infty$ for $j=1,\ldots,n$, we see that for sufficiently large $l\in\Nset$,
there exists a unique $n$-tuple $\alpha(l)=\alpha=(\alpha_1,\ldots,\alpha_n)\in (0,1)^n$ such that
\[
(\hat{T}_1^{\alpha_1},\ldots, \hat{T}_n^{\alpha_n})=2\pi l(r_1,\ldots, r_n).
\]
Below, we will frequently use the notations
\[
z^{\alpha(l)}=(z_1^{\alpha_1(l)},\ldots, z_n^{\alpha_n(l)}),\quad 
z^{\h}=(z_1^{\h},\ldots, z_n^{\h}).
\]
For $\tau=(\tau_1,\ldots,\tau_n)\in\Rset^n$, denote
\begin{align*}
z(t,\tau):=(z_1(t+\tau_1),\ldots,z_n(t+\tau_n))
\end{align*}
and
\begin{align*}
z_{[m]}(t,\tau):=(z_1(t+\tau_1),\ldots,z_{m}(t+\tau_m), z_{m+1}^*,\ldots,z_{n}^*),\qquad m=1,\ldots,n
\end{align*}
where we put $z_{[n]}(t,\tau)=z(t,\tau)$.
Note that
\[
\int_{-\infty}^\infty f_{j,n}^{\h} (t;\tau,\theta, I,\phi)dt=M_j(\tau,\theta, I,\phi),\qquad 
j=1,\ldots, n.
\]

For each $(\tau,\theta,I,\phi)\in\Rset^n\times \Tset^d\times \Rset^d\times \Sset^1$, $j=1,\ldots, n$ and $t\in\Rset$,
set
	\begin{align*}
	f_j^{\alpha(l)}(t;\tau,\theta, I,\phi)&:=\D G_j\cdot u(z^{\alpha(l)}(t,\tau),\omega(I)t+\theta,I,t+\phi,0)\\
	 f_j^\h(t;\tau,\theta, I,\phi)&:=\D G_j\cdot u(z^{\h}(t,\tau),\omega(I)t+\theta,I,t+\phi,0).
	\end{align*}
    and
    \begin{align*}
    f_{j,m}^{\alpha(l)}&(t;\tau,\theta, I,\phi):=
    \D G_j\cdot u(z_{[m]}^{\alpha(l)}(t,\tau),\omega(I)t+\theta,I,t+\phi,0),\\
    f_{j,m}^\h&(t;\tau,\theta, I,\phi):=
    \D G_j\cdot u(z_{[m]}^{\h}(t,\tau),\omega(I)t+\theta,I,t+\phi,0)
    \end{align*}
    where $m=1,\ldots,n$.

The following result is a higher dimensional analogue of Theorem 4.6.4 in \cite{GH83}]
\begin{lem}[Approximation of  the homoclinic Melnikov integrals by subharmonic Melnikov integrals]\label{lem: periodic_to_homoclinic}
Let $U$ and $\mathscr{U}$ be a bounded domain in $\Rset^d$ and $\Rset^n$ respectively. Then it holds that
\[
\lim_{l\to+\infty} \int_{-\pi l r_1}^{\pi l r_1} f_{j,m}^{\alpha(l)}(t;\tau,\theta, I,\phi) dt=\int_{-\infty}^{\infty} f_{j,m}^{\h}(t;\tau,\theta, I,\phi) dt
\]
uniformly in $(\tau,\theta, I,\phi)\in  \mathscr{U}\times \Tset^d\times\bar{U}\times\Sset^1$
for $j,m=1,\ldots, n$.
\end{lem}
    Note that we cannot replace $r_1$ in Lemma~\ref{lem: periodic_to_homoclinic} with other $r_2,\ldots, r_n$ since Lemma~\ref{lem: periodic_to_homoclinic} relies on the fact that the motion of homoclinic orbit $z^{\h}(t)$ is approximated by $z^{\alpha(l)}(t)$ on $t\in[-T/2,T/2]$, for sufficiently large $T>0$ and $l\in\Nset$.
\begin{proof}[Proof of Lemma~\ref{lem: periodic_to_homoclinic}]
    First, for each $l\in\Nset$  and $t\in\Rset$, we compute
	\begin{align*}
		& \int_{-\infty}^{\infty} f_{j,m}^{\h}(t;\tau,\theta, I,\phi) dt-\int_{-\pi l r_1}^{\pi l r_1} f_{ j,m}^{\alpha(l)}(t;\tau,\theta, I,\phi)dt\\
		&=\Big(\int_{T/2}^\infty f_{j,m}^\h(t;\tau,\theta, I,\phi) dt+\int_{-T/2}^{T/2} f_{j,m}^\h(t;\tau,\theta, I,\phi) dt+\int_{-\infty}^{-T/2} f_{j,m}^\h(t;\tau,\theta, I,\phi) dt\Big)\\
		&\quad-\Big(\int_{T/2}^{\pi l r_1} f_{j,m}^{\alpha(l)}(t;\tau,\theta, I,\phi) dt
		+\int_{-{T/2}}^{{T/2}} f_{j,m}^{\alpha(l)}(t;\tau,\theta, I,\phi) dt
		+\int_{-\pi l r_1}^{-{T/2}} f_{j,m}^{\alpha(l)}(t;\tau,\theta, I,\phi) dt\Big)\\
		&=\int_{-{T/2}}^{T/2} \big(f_{j,m}^\h(t;\tau,\theta, I,\phi)- f_{j,m}^{\alpha(l)}(t;\tau,\theta, I,\phi)\big) dt+\Big(\int_{T/2}^\infty f_{j,m}^\h(t;\tau,\theta, I,\phi) dt\\
		&\qquad\qquad -\int_{T/2}^{\pi l r_1} f_{j,m}^{\alpha(l)}(t;\tau,\theta, I,\phi) dt
		+\int_{-\infty}^{-T/2} f_{j,m}^\h(t;\tau,\theta, I,\phi) dt-\int_{-\pi l r_1}^{-T/2} f_{j,m}^{\alpha(l)}(t;\tau,\theta, I,\phi) dt\Big).
	\end{align*}
	Fix arbitrary $\mu>0$.
	For each $j=1,\ldots,n$, let $U_j(\mu)$ be an open neighborhood of $z^*_j\in\Rset^2$ such that the arc lengths of $U_j(\mu)\cap \Gamma_j$ 
    is less than $\mu$.
	For sufficiently large $T>0$ 
    we see that 
    $z_j^\h(\pm T/2+\tau_j)$ 
    belongs to $U_j(\mu)$
    for all $\tau\in \mathscr{U}$.
	
	By the smoothness of integrant for each variable and continuity of solutions with respect to initial conditions, we deduce that there is $l^*=l^*(T,\mathscr{U}) \in\Nset$ such that for all $l>l^*$ we have $2\pi l r_1>T$ and
	\[
	\Big|\int_{-T/2}^{T/2} \big(f_{j,m}^\h(t;\tau,\theta, I,\phi)- f_{j,m}^{\alpha(l)}(t;\tau,\theta, I,\phi)\big) dt\Big|<\mu,\qquad j=1,\ldots, d.
	\]
   Now, using Lemma~\ref{lem:trapping_estimate}, we take $l^{**}\in\Nset$ such that if $l>l^{**}$
    the arc lengths of $U_j(\mu)\cap \Gamma_j^{\alpha_j(l)}$ 
    is less than $\mu$ and
    $z_j^{\alpha_j(l)}(t+\tau_j)$ 
    belongs to $U_j(\mu)$
    for all $j=1,\ldots,n$, $\tau\in \mathscr{U}$ and $t\in [-\pi l r_1, -T/2]\cup [T/2,\pi l r_1]$.
    Set $l'=\max\{l^*,l^{**}\}$.
	Then, introducing the arc length parameters, we compute
	\begin{align*}
		\Big|\int_{T/2}^\infty f_{j,m}^\h(t;\tau,\theta, I,\phi) dt\Big|
		&=\Big|\int_0^{L(z^\h)} \D G_j\cdot u \frac{1}{||X_{H_0}||}ds\Big|\\
		&\le \int_0^{L(z^\h)} \Big|\D G_j\cdot u  \frac{1}{||X_{H_0}||}\Big|ds\\
		&\le \int_0^{L(z^\h)} || u || ds\le CL(z^\h)\le C\mu,
	\end{align*}
	where $C=\max_{(x,\phi)\in W\times \Sset^1} ||u(x,I',\theta',\phi;0)||$ with $W=\Gamma_1\times\ldots\times\Gamma_{n}\times \bar{U}\times \Tset^d$.
	Note that here we use the relation $||\D G_j||\le ||X_{H_0}||$ on $\Gamma_1\times\ldots\times\Gamma_{n}\setminus\{0\}$.
	We can obtain similar estimates
	for $\int_{T/2}^{\pi l r_1} f_{j,m}^{\alpha(l)}(t;\tau,\theta, I,\phi) dt$, $\int_{-\infty}^{-{T/2}} f_{j,m}^\h(t;\tau,\theta, I,\phi) dt$, $\int_{-\pi l r_1}^{-{T/2}} f_{j,m}^{\alpha(l)}(t;\tau,\theta, I,\phi) dt$ for all $l>l'$,
	and thus we conclude that
	\begin{align*}
		\Big|\int_{T/2}^\infty f_{j,m}^\h(t;\tau,\theta, I,\phi) dt&-\int_{T/2}^{\pi l r_1} f_{j,m}^{\alpha(l)}(t;\tau,\theta, I,\phi) dt
		+\int_{-\infty}^{-{T/2}} f_{j,m}^\h(t;\tau,\theta, I,\phi) dt\\
		& -\int_{-\pi l r_1}^{-{T/2}} f_{j,m}^{\alpha(l)}(t;\tau,\theta, I,\phi) dt\Big|\le 4C\mu
	\end{align*}
	provided $l> l'$.
	Hence
	\[
		\left|\int_{-\infty}^{\infty} f_{j,m}^{\h}(t;\tau,\theta, I,\phi) dt-\int_{-\pi l r_1}^{\pi l r_1} f_{j,m}^{\alpha(l)}(t;\tau,\theta, I,\phi)dt\right|< (4C+1)\mu,
	\]
	which implies that $\displaystyle \left|\int_{-\infty}^{\infty} f_{j,m}^{\h}(t;\tau,\theta, I,\phi) dt-\int_{-\pi l r_1}^{\pi l r_1} f_j^{\alpha(l)}(t;\tau,\theta, I,\phi)dt\right|\to 0$ as $l\to+\infty$.
	\end{proof}
    
Combining Lemma~\ref{lem:trapping_estimate} and Lemma~\ref{lem: periodic_to_homoclinic}, we obtain the following approximation formula.
\begin{prop}\label{prop:higher_approximation}
Suppose that $I\in U$ satisfies $\omega(I)\in \frac{1}{R}\Zset^d$ for some $R\in\Nset$.
Let $\mathscr{U}$ be a bounded domain in $\Rset^n$. 
For arbitrary positive  odd  integers $k_1,k_2,\ldots,k_n$ and  $l\in\Nset$, consider the case $T_{m}^{\alpha_{m}(l)}=k_{m}k_{m-1}\ldots k_1 2\pi lR\ (m=1,\ldots, n)$, i.e., $r_m=k_mk_{m-1} \ldots k_1 R$. Then
\begin{align*}
\lim_{l\to+\infty} &\int_{ -T_n^{\alpha_n(l)}/2}^{ T_n^{\alpha_n(l)}/2} f_j^{\alpha(l)}(t;\tau,\theta, I,\phi) dt\\
&=\int_{-\infty}^{\infty} f_j^\h(t;\tau,\theta, I,\phi) dt+\sum_{i=2}^{n}\frac{k_1k_2\cdots k_n}{k_1 k_2\cdots k_i}(k_i-1) 
\int_{-\infty}^{\infty} f_{j,i-1}^{\h}(t;\tau,\theta, I,\phi) dt
\end{align*}
uniformly in $(\theta,\tau,\phi)\in \Tset^d\times \mathscr{U}\times \Sset^1$. 
\end{prop}
Note that
\[
\frac{k_1k_2\cdots k_n}{k_1 k_2\cdots k_i}(k_i-1) =\begin{cases} k_2-1 & (n=2)\\ (k_2-1)k_3\ldots k_n & (n\ge 3).
\end{cases}
\]
\begin{proof}
Fix $j=1,\ldots, n$ and arbitrarily small $\mu>0$.
We take $l\in\Nset$ sufficiently large so that $\alpha_j(l)>\alpha_j^*$ with $\alpha^*$ in Lemma~\ref{lem:trapping_estimate} and 
\[
\left|\int_{-\pi lr_1}^{\pi l r_1} f_{j,m}^{\alpha(l)}(t;\tau,\theta, I,\phi)dt-
\int_{-\infty}^{\infty} f_{j,m}^{\h}(t;\tau,\theta, I,\phi)dt\right|<\mu,\quad m=1,\ldots,n.
\]
Let $k'_m=(k_m-1)/2$ and
\[
a_{j,m}^i(l;\tau,\theta, I,\phi)=a_{j,m}^i:=\int_{-T_i^{\alpha_i(l)}/2}^{T_i^{\alpha_i(l)}/2} f_{j,m}^{\alpha(l)}(t;\tau,\theta, I,\phi) dt
\]
for $i,m=1,\ldots,n$.
Consider intervals
\[
\Delta(m;\sigma)=\left[\left(\sigma-\frac{1}{2}\right)T_{m}^{\alpha_{m}(l)}, \left(\sigma+\frac{1}{2}\right)T_{m}^{\alpha_{m}(l)}\right],\quad \sigma=0,\pm1,\ldots,\pm k_{m+1}'
\]
for $m=1,\ldots,n-1$.
Note that
\[
\Delta(m;0)=\bigcup_{\sigma=0,\pm 1,\ldots,\pm k_m'} \Delta(m-1;\sigma).
\]

Then, for $m=2,\ldots,n$, it holds that
\begin{align*}
a_{j,m}^m
&=\int_{\Delta(m;0)} f_{j,m}^{\alpha(l)}(t;\tau,\theta, I,\phi) dt\\
&=
\int_{\Delta(m-1;0)} f_{j,m}^{\alpha(l)}(t;\tau,\theta, I,\phi) dt +\sum_{j=\pm 1,\ldots,\pm k'_m} \int_{\Delta(m-1; \sigma)} f_{j,m}^{\alpha(l)}(t;\tau,\theta, I,\phi) dt\\
&=a_{j,m}^{m-1}+\sum_{j=\pm 1,\ldots,\pm k'_m} \int_{\Delta(m-1; \sigma)} f_{j,m}^{\alpha(l)}(t;\tau,\theta, I,\phi) dt
\end{align*}

For the second term, by shifting the time variable $t=t'+\sigma T_{m-1}^{\alpha_{m-1}}$ and the periodicity of the perturbation $u$ in $(\theta,t)$, we have $T_{m-1}^{\alpha_{m-1}}\omega(I)\in2\pi \Zset^d$ and
\begin{align*}
& \int_{\Delta(m-1; \sigma)} f_{j,m}^{\alpha(l)}(t;\tau,\theta, I,\phi) dt\\
& = \int_{\Delta(m-1; 0)}  \D G_j\cdot u(z_{[m]}^{\alpha(l)}(t'+\sigma T_{m-1}^{\alpha_{m-1}},\tau),\omega(I)t'+\theta,I,t'+\phi,0)dt'
\end{align*}
Moreover, since $z_1^{\alpha_1}(t),\ldots,z_{m-1}^{\alpha_{m-1}}(t)$ are $T_{m-1}^{\alpha_{m-1}}$-periodic,
\begin{align*}
&z_{[m]}^{\alpha(l)}(t'+\sigma T_{m-1}^{\alpha_{m-1}},\tau)\\
&=(z_1^{\alpha_1(l)}(t'+\tau_1), \ldots, z_{m-1}^{\alpha_{m-1}(l)}(t'+\tau_{m-1}), z_m^{\alpha_m(l)}(t'+\sigma T_{m-1}^{\alpha_{m-1}}+\tau_m), z_{m+1}^*,\ldots,z_{n}^*)
\end{align*}
By  Lemma~\ref{lem:trapping_estimate} and $t'\in \Delta(m-1; 0)$,
we obtain
\begin{align*}
& \int_{\Delta(m-1; 0)}\D G_j\cdot u(z_{[m]}^{\alpha(l)}(t'+\sigma T_{m-1}^{\alpha_{m-1}},\tau),\omega(I)t'+\theta,I, t'+\phi,0)dt'\\
&= \int_{\Delta(m-1; 0)} \D G_j\cdot u(z_{[m-1]}^{\alpha(l)}(t',\tau),\omega(I)t'+\theta, I, t'+\phi,0)dt'+O(\mu)\\
&=a_{j,{m-1}}^{m-1}+O(\mu),
\end{align*}
here we employ the estimate using the arc length parameter as in Lemma~\ref{lem: periodic_to_homoclinic}.
Therefore, we deduce that
\[
a_{j,m}^m=a_{j,m}^{m-1}+(k_m-1)a_{j,m-1}^{m-1}+O(\mu).
\]

Repeating a similar computation for the first term of the right-hand side, we obtain
\begin{align*}
a_{j,m}^m=a_{j,m}^{1}+\sum_{i=1}^{m-1} (k_{i+1}-1)a_{j,i}^i+O(\mu),\qquad m=2,\ldots,n.
\end{align*}
Let
\[
S_{j,m'}:=\sum_{i=1}^{m'} (k_{i+1}-1)a_{j,i}^i
\]
for $m'=1,\ldots,n-1$.
It holds that $S_{j,1}=(k_2-1)a_{j,1}^1$.
Then we compute
\begin{align*}
S_{j,m}=(k_{m+1}-1)a_{j,m}^1+k_{m+1}S_{j,m-1}+O(\mu),\qquad m=2,\ldots,n
\end{align*}
since
\begin{align*}
S_{j,m}&=(k_{m+1}-1)a_{j,m}^{m}+\sum_{i=1}^{m-1} (k_{i+1}-1)a_{j,i}^i\\
&=(k_{m+1}-1)\left(a_{j,m}^1+\sum_{i=1}^{m-1} (k_{i+1}-1)a_{j,i}^i+O(\mu).\right)
+\sum_{i=1}^{m-1} (k_{i+1}-1)a_{j,i}^i\\
&=(k_{m+1}-1)a_{j,m}^1+\sum_{i=1}^{m-1} \{(k_{m+1}-1)(k_{i+1}-1)+(k_{i+1}-1)\}a_{j,i}^i+O(\mu)\\
&=(k_{m+1}-1)a_{j,m}^1+k_{m+1}S_{j,m-1}+O(\mu).
\end{align*}
Now, by the induction, we show that 
\[
S_{j,m}=\sum_{i=2}^{m+1}\frac{k_1k_2\cdots k_{m+1}}{k_1 k_2\cdots k_i}(k_{i}-1)a_{j,i-1}^1+O(\mu).
\]
for all $m=1,\ldots, n-1$. This already holds for $m=1$ and assume that the case $m=m'\in\{1,\ldots,n-2\}$ also holds.
Then
\begin{align*}
S_{j,m'+1}
&=(k_{m'+2}-1)a_{j,m'+1}^1+k_{m'+2}S_{j,m'}+O(\mu)\\
&=(k_{m'+2}-1)a_{j,m'+1}^1+k_{m'+2}\sum_{i=2}^{m'+1}\frac{k_1k_2\cdots k_{m'+1}}{k_1 k_2\cdots k_i}(k_i-1)a_{j,i-1}^1+O(\mu)\\
&=\frac{k_1k_2\cdots k_{m'+1}k_{m'+2}}{k_1k_2\cdots k_{m'+1}k_{m'+2}}(k_{m'+2}-1)a_{j,m'+1}^1\\
&\qquad+\sum_{i=2}^{m'+1}\frac{k_1k_2\cdots k_{m'+1} k_{m'+2}}{k_1 k_2\cdots k_i}(k_i-1)a_{j,i-1}^1+O(\mu)\\
&=\sum_{i=2}^{m'+2}\frac{k_1k_2\cdots k_{m'+2}}{k_1 k_2\cdots k_i}(k_i-1)a_{j,i-1}^1+O(\mu).
\end{align*}
Therefore, we have completed the induction step, and thus we conclude that
\begin{align*}
a_{j,m}^m&=a_{j,m}^{1}+S_{j,m-1}+O(\mu)\\
&=a_{j,m}^{1}+\sum_{i=2}^{m}\frac{k_1k_2\cdots k_{m}}{k_1 k_2\cdots k_i}(k_{i}-1)a_{j,i-1}^1+O(\mu).
\end{align*}
From Lemma~\ref{lem: periodic_to_homoclinic}, we obtain
\begin{align*}
\lim_{l\to+\infty}&a_{j,n}^n(l;\tau,\theta, I,\phi)=\int_{-\infty}^{\infty} f_j^{\h}(t;\tau,\theta, I,\phi) dt\\
&+\sum_{i=2}^{n}\frac{k_1k_2\cdots k_n}{k_1 k_2\cdots k_i}(k_i-1) 
\int_{-\infty}^{\infty} f_{j,i-1}(t;\tau,\theta, I,\phi) dt,
\end{align*}
which yields the desired result.
\end{proof}

\begin{lem}\label{lem:key_set}
Let $D$ be a key set for the set of analytic functions on an open subset $U$ in $\Rset^d$.
Suppose that for each $I\in D$ there exists a key set $K_I$ for the set of analytic functions on $(0,1)^n$.
Then $\bigcup_{I\in D} K_I\times\{I\}$ is a key set for the set of analytic functions on $(0,1)^n\times\Rset^n$.
\end{lem}
\begin{proof}
Take any analytic function $\Phi,\Psi$ on $(0,1)^n\times\Rset^d$.
Assume that $\Phi=\Psi$ on the set $\bigcup_{I\in D} K_I\times\{I\}$.
Then for each $I\in D$, we have $\Phi(\alpha,I)=\Psi(\alpha,I)$ for all $\alpha\in K_I$.
Since $K_I$ is a key set for the set of analytic function on $(0,1)^n$,
it holds that $\Phi(\alpha,I)=\Psi(\alpha,I)$ on $(0,1)^n\times\{I\}$ for each $I\in D$.
Thus
\[
\Phi|_{(0,1)^n\times D}=\Psi|_{(0,1)^n\times D}
\]
holds. For fixed $\alpha\in(0,1)^n$, we obtain $\Phi(\alpha,I)=\Psi(\alpha,I)$ for all $I\in D$ and this implies that $\Phi|_{\{\alpha\}\times U}=\Psi|_{\{\alpha\}\times U}$ for each $\alpha\in (0,1)^n$ since $D$ is a key set of analytic functions on $U$.
Therefore, $\Phi$ coincides with $\Psi$ on the open subset $(0,1)^n\times U$ in $(0,1)^n\times \Rset^d$ and the identity theorem implies $\Phi=\Psi$ on $(0,1)^n\times \Rset^d$.
\end{proof}

\begin{proof}[Proof of Theorem~\ref{thm:Melnikov1}]
Pick $I'\in\Rset^d$ such that
\[
\dim_{\Rset}\mathrm{span}_{\Rset} (M(\Rset^n\times \Tset^d\times \{I=I'\}\times\Sset^1))=s.
\]
Then there exists $s$ points $(\tau_i,\theta_i,\phi_i)\in \Rset^n\times \Tset^d\times\Sset^1\ (i=1,\ldots, s)$ such that $v_i(I'):=M(\tau_i,\theta_i, I',\phi_i)\ (i=1,\ldots, s)$ are linearly independent.
Since $M$ is continuous, there exists an open neighborhood $U_{I'}$ of $I=I'$ such that
 $v_i(I)=M(\tau_i,\theta_i, I',\phi_i)\ (i=1,\ldots, s)$ are linearly independent for all $I\in U_{I'}$.
By the analyticity of $h$ and (M4), we can take an open subset $\widetilde{U}\subset U_{I'}$ such that $v_1(I),\ldots, v_{s}(I)$ are linearly independent and $\frac{\partial^2 h}{\partial I^2}(I)$ is non-degenerate for all $I\in \widetilde{U}$.

 By the density of rational points in $\Rset^d$ and the fact that the non-degeneracy of $\frac{\partial^2 h}{\partial I^2}(I)$ implies that $\nabla h$ is a local diffeomorphism on $\widetilde{U}$, we see that $\widetilde{U}$ contains a dense subset $D$ such that $\nabla{h}(I)$ is a rational point for all $I\in D$.
 Then, for each $I\in D$, there exist a positive integer $R_I$ such that $R_I\nabla{h}(I)\in\Zset^n$. We take the smallest one for $R_{I}$.

Fix $I\in D$. We now construct a key set $K_I$ for the set of analytic functions on $(0,1)^n$.
Consider $a_{j, n}^n(l;\tau,\theta, I,\phi)$ defined in the proof of Proposition~\ref{prop:higher_approximation} with arbitrary positive odd integers $k_1,\ldots,k_n,l$ and $l\in\Nset$ with $R=R_I$.
Let $w_i(l;I)=a_{i,n}^n(l;\tau_i,\theta_i,I,\phi_i)$ for $i=1,\ldots,s$.
Assume that $w_1(l;I),\ldots,w_s(l;I)$ are linearly dependent, i.e., 
\[
\xi_1 w_1(l;I)+\ldots+\xi_s w_s(l;I)=0
\]
for some constant $(\xi_1,\ldots,\xi_s)\neq (0,\ldots,0)$.
Take arbitrary small $\mu>0$. Then, for sufficiently large $l\in\Nset$, we have
\begin{align}\label{eq:linearly_dependent}
\sum_{j=1}^s\xi_j &\int_{-\infty}^{\infty} f^\h(t;\tau_j,\theta_j,I,\phi_j) dt\nonumber\\
&+\sum_{i=2}^{n}\frac{k_1k_2\cdots k_n}{k_1 k_2\cdots k_i}(k_i-1) 
\sum_{j=1}^s\xi_j\int_{-\infty}^{\infty} f_{[i-1]}^{\h}(t;\tau_j,\theta_j,I,\phi_j) dt=O(\mu)
\end{align}
from Proposition~\ref{prop:higher_approximation},
where we use $f^{\h}=(f_1^{\h},\ldots,f_n^{\h})$ and $f_{[i]}^{\h}=(f_{1,i}^{\h},\ldots,f_{n,i}^{\h})$ for $i=1,\ldots,n$.
Dividing both sides by $k_2$ and letting $k_2\to+\infty$, we have
\begin{align}\label{eq:linearly_dependent_1}
\sum_{j=1}^s\xi_j\int_{-\infty}^{\infty} f_{[1]}^{\h}(t;\tau_j,\theta_j,I,\phi_j) dt=0.
\end{align}
Substituting \eqref{eq:linearly_dependent_1} to \eqref{eq:linearly_dependent} and repeating similar procedures, we obtain
\[
\sum_{j=1}^s\xi_j\int_{-\infty}^{\infty} f_{[i]}^{\h}(t;\tau_j,\theta_j,I,\phi_j) dt=0,\quad j=1,\ldots,n.
\]
However, the case $i=n$ contradicts to the fact that $v_1(I)=M(\tau_1,\theta_1, I,\phi_1),\ldots, v_s(I)=M(\tau_s,\theta_s,I,\phi_s)$ are linearly independent.
As a result, for sufficiently large $k_1>k_2>k_3>\ldots>k_n$ and $l\in\Nset$,
we see that $w_1(l;I),\ldots,w_s(l;I)$ are linearly independent.
This implies that for each $I\in D$ there exists a key set $K_I$ in $(0,1)^n$ such that $w_1(l;I),\ldots,w_s(l;I)$ are linearly independent.

By Lemma~\ref{lem:key_set}, $\bigcup_{I\in D} K_I\times\{I\}$ is a key set in $(0,1)^n\times \Rset^d$ and, from the same discussion of Proposition~\ref{prop:sufficient_unbounded}, we deduce that for $(\alpha,I)\in \bigcup_{I\in D} K_I\times\{I\}$ the vector-valued map $\mathscr{A}_{(G,I)}^{(\alpha,I)}$ has at least $s$ unbounded directions.
Therefore, the set
\[
\left\{(\alpha,I)\in (0,1)^n\times \Rset^d; \mathscr{A}_{(G,I)}^{(\alpha,I)} \ \text{has}\  s\text{ unbounded directions}\right\}
\]
is a key set for the set of analytic functions on $(0,1)^n\times \Rset^d$,
which completes the proof by Theorem~\ref{thm:main1}.
\end{proof}

\begin{proof}[Proof of Theorem~\ref{thm:Melnikov2}]
The proof of Theorem~ is almost the same as that of Theorem~\ref{thm:Melnikov1}, thus we omit it.
\end{proof}

\section{Application to partially conformally symplectic Hamiltonian systems}
\label{sec:application_partially_conformal}

In this subsection, we consider nearly integrable, partially conformally symplectic Hamiltonian systems. 
For definition and properties, see Appendix \ref{sec:partially_conformal}.

Let $D\subset \Rset^m$ be a domain. Consider an analytic Hamiltonian $H_0(q,p):D\times\Rset^m\to \Rset$ of the form
	\begin{align}\label{eqn:mechanical_Hamiltonian}
		H_0(q,p)=\frac{1}{2}||p||^2+V(q),
	\end{align}
	where $V:D\to \Rset$ is a non-constant analytic function.
	We consider a perturbed vector field given by
	\begin{align}\label{eqn:dissipativity}
		X_\varepsilon=X_{H_0}+\varepsilon (X_{H^1}+Y),
	\end{align}
	where $H^1:D\times\Rset^m\to \Rset$ is analytic and $Y$ is of the form
	\begin{equation}
	Y(q,p)=\sum_{i=1}^m -\lambda_i p_i\frac{\partial}{\partial p_i}
	\end{equation}
	where  $\lambda_i \ge 0\ (i=1,\ldots,m)$ and $\lambda_{i_0}>0$ for some $i_0$.

	We now show the following result, which implies that
	dissipativity prevents the existence of a maximal number  of functionally independent first integrals for \eqref{eqn:dissipativity}
	(see \cite{M24} for the planar case).
	\begin{thm}\label{thm:dissipative_nonint}
		Under the condition (A1)', (A2) and (A3) with the unperturbed Hamiltonian $H_0$ given by Eq.~\eqref{eqn:mechanical_Hamiltonian}, for $|\epsilon|\neq0$, Eq.~\eqref{eqn:dissipativity} has at most
		$m-1$ functionally independent analytic first integrals which depend on $\varepsilon$ analytically.
	\end{thm}
	\begin{proof}
        Assume that $\lambda_{i_0}>0$ for some $i_0\in\{1,\ldots,m\}$. 
		We can apply Theorem~\ref{thm:main1} to this case and
		we need to check that $\mathscr{Q}_1$ is a key set for $C^\omega(U)$.

		Take any  regular value $c\in U$ of $H_0$ such that $\omega(c)$ is a rationally independent vector. 

  Let $\mathcal{J}_c$ be the invariant torus of the unperturbed integrable system corresponding to the level $c\in  D_{\mathrm{N}}$.
  Since  $\omega(c)$ is a rationally independent vector, each orbit  $\{\gamma_\tau^c\}_{\tau\in\Tset^m}$ is dense in $\mathcal{J}_c$.
		
		Fix $(\tau,c)\in \Tset^m\times  D_{\mathrm{N}}$. Computing $\mathscr{A}_{H_0}^c$, we see that
		\begin{align*}
			\mathscr{A}_{H_0}^c(\tau,T)&=\int_0^T \big(X_{H^1}(H_0)+Y(H_0)\big)_{\gamma_\tau^c(t)}dt\\
			&=H^1(\gamma_\tau^c(0))-H^1(\gamma_\tau^c(T))-\sum_{i=1}^m \lambda_i \int_0^T p_i^2 (\gamma_\tau^c(t))dt\\
			&\le H^1(\gamma_\tau^c(0))-\max_{z\in \mathcal{J}_c} H^1(z)
			-\lambda_{i_0}\int_0^T  p^2_{i_0}(\gamma_\tau^c(t)) dt.
		\end{align*}

If $p_{i_0}(\gamma_\tau^c(t))$ is identically zero  in $t$, from the density of $\gamma_\tau^c$ it follows that $p_{i_0}$ is identically zero on $\mathcal{J}_c$.
Since $\dot q_{i_0}(\gamma_\tau^c(t))=p_{i_0}(\gamma_\tau^c(t))=0$, it follows that $q_{i_0} =\textrm{const}$ on $\mathcal{J}_c$.
This implies that $\dfrac{\partial p_{i_0}}{\partial \theta}=\dfrac{\partial q_{i_0}}{\partial \theta}=0$ on $\mathcal{J}_c$, which contradicts the fact that
$(I,\theta)\mapsto \phi(I,\theta)=(p,q)$ is a diffeomorphism. 

Therefore, it holds that $ p_{i_0} ^2(\gamma_\tau^c(t))$ is not identically zero as a function on $t$,
		 i.e., there exists $T^*\in\Rset$ such that
		 	$\delta:=p_{i_0}^2(\gamma_\tau^c(T^*))>0$.
		 From the quasi-periodicity of $\gamma_\tau^c$ and continuity of the function $p_{i_0}^2$, there is a monotonically increasing sequence $\{T_j\}_{j\in\Nset}\subset\Rset$ such that $p^2_{i_0}(\gamma_\tau^c(T_j))>\delta/2$ and $\lim_{j\to\infty} T_j=+\infty$. 
Taking a subsequence we can assume that $T_{j}-T_{j-1}>2$ for all $j\in\Nset$ with $T_0=0$.
		 Since
		 \begin{align*}
		 	\frac{d}{dt}p^2_{i_0}(\gamma_\tau^c(t))= -2p_{i_0}(\gamma_\tau^c(t))\frac{\partial V}{\partial q_{i_0}}(\gamma_\tau^c(t)),
		 \end{align*}
		 we have
		 \[
		 \left|	p_{i_0}^2(\gamma_\tau^c(T_j+t))-p_{i_0}^2(\gamma_\tau^c(T_j))\right|\le M|t|,
		 \]
		 where $M=\max\{|2 p_{i_0}\frac{\partial V}{\partial q_{i_0}}|; (q,p)\in \mathcal{J}_c\}$.
		 Hence for $|t|\le\frac{\delta}{4M}$ we obtain
		 \[
		 p^2_{i_0}(\gamma_\tau^c(T_j+t))\ge p^2_{i_0}(\gamma_\tau^c(T_j))-M|t|\ge \delta/4.
		 \]
		 Letting $\delta'=\min\{\frac{\delta}{4M}, 1\}$, we see that $T_{j-1}+1, T_{j+1}-1\notin[T_j-\delta',T_j+\delta' ]$ for all $j$.
		 Then
		 \[
		 \lim_{T\to +\infty}\int_0^T p^2_{i_0}(\gamma_\tau^c(t)) dt
		 \ge \sum_{j=1}^\infty \int_{T_j-\delta'}^{T_j+\delta'} p ^2_{i_0}(\gamma_\tau^c(t)) dt\ge \sum_{j=1}^\infty 2\delta'\cdot \frac{\delta}{4}=+\infty.
		 \]
		 Therefore, $\mathscr{A}_{H_0}^c(\tau,T)$ is unbounded in $T$ for each $(\tau,c)\in \Tset^m\times U^d$, which implies the dense set  $U^d$ is included in $\mathscr{Q}_1$ and thus $\mathscr{Q}_1$ is a key set for $C^\omega(U)$.
	\end{proof}

\begin{rmk}
Note that the Hamiltonian part of the perturbation $H^1$ plays no role in Theorem \ref{thm:dissipative_nonint}, so the result remains valid for $H^1=0$. 
\end{rmk}

\begin{rmk}
The result in Theorem \ref{thm:dissipative_nonint} remains true for dissipative vector fields of the form
 \begin{equation}
	Y(q,p)=\sum_{i=1}^m -\lambda_i p_i\|p\|^{n-1}\frac{\partial}{\partial p_i}, \quad n\ge 1,
	\end{equation}
which include, as a particular case for $n=2$, the quadratic drag force (as in the case of falling objects through air or high-speed motion in fluids in the turbulent regime).
\end{rmk}

\section{Application  to systems of pendulums}

\label{sec:systems_penduli}
\subsection{Coupled pendulums with a quasi-periodic external force}

Consider the system 
\begin{equation}\label{eqn:pendulums}
\begin{alignedat}{2}
\dot{q}_1 &= p_1, \qquad &
\dot{p}_1 &= -\sin q_1-\varepsilon\alpha(q_1-q_2)
+\varepsilon\delta_1 p_1+\varepsilon\mu_1\cos \phi,\\
\dot{q}_2 &= p_2, \qquad &
\dot{p}_2 &= -\sin q_2+\varepsilon\alpha(q_1-q_2)
+\varepsilon\delta_2 p_2+\varepsilon\mu_2\cos \theta,\\
\dot{\phi} &=  \nu, \qquad & &\\
\dot{\theta} &= I, \qquad &
\dot{I} &= 0.
\end{alignedat}
\end{equation}
with variables $(q_i,p_i)\in\Rset^2\ (i=1,2)$, $\phi\in\Sset^1=\Rset/(2\pi\Zset)$, and
$(\theta,I)\in \Sset^1\times (0,\infty)$,
where $\alpha,\delta_i,\mu_i, \nu \in\Rset\  (i=1,2)$ are constants with $ \nu >0$ and $\varepsilon\in(-\epsilon_0,\epsilon_0)$ is a perturbation parameter with some $\epsilon_0>0$.
Note that in Eq.~\eqref{eqn:pendulums} the terms $\delta_ip_i\ (i=1,2)$ and $\mu_1\cos\phi, \mu_2\cos\theta$ correspond to dissipation and  external forcing respectively, and $\alpha(q_1-q_2)$ represents the interaction of two pendulums.
We also note that \eqref{eqn:pendulums} has a trivial first integral $I$.

When $\varepsilon=0$, Eq.~\eqref{eqn:pendulums} becomes a system given by
\begin{equation}\label{eqn:pendulums_linear}
\begin{alignedat}{2}
\dot{q}_1 &= p_1, \qquad &
\dot{p}_1 &= -\sin q_1,\\
\dot{q}_2 &= p_2, \qquad &
\dot{p}_2 &= -\sin q_2,\\
\dot{\phi} &=  \nu, \qquad & &\\
\dot{\theta} &= I, \qquad &
\dot{I} &= 0.
\end{alignedat}
\end{equation}
The unperturbed system \eqref{eqn:pendulums_linear} has two first integrals
\begin{align*}
F_i=\frac{p_i^2}{2}+(1-\cos q_i),\quad i=1,2,
\end{align*}
and a family of quasi-periodic orbits $\gamma_\tau^c(t)=(q_1(t+\tau_1;c_1),p_1(t+\tau_1;c_1),q_2(t+\tau_2;c_2),p_2(t+\tau_2;c_2), \nu t+\tau_3,c_3 t+\tau_4, c_3)$ with 
\begin{align*}
	(q_i(t;c_i), p_i(t;c_i))=\left(2\arcsin\left(\sqrt{\frac{c_i}{2}} \sn \left(t;\sqrt{\frac{c_i}{2}}\right)\right), 2\sqrt{\frac{c_i}{2}}\cn\left(t;\sqrt{\frac{c_i}{2}}\right)\right),
	\end{align*}
	with $i=1,2$ for $ c=(c_1,c_2,c_3)\in(0,2)^2\times(0,\infty)$ and $ \tau=(\tau_1,\tau_2,\tau_3,\tau_4)\in\Tset^4=(\Rset/2\pi\Zset)^4$,
	where $\sn(\cdot;k)$ and $\cn(\cdot;k)$ represent the Jacobi elliptic functions with the elliptic modulus $k$.
	Note that $(q_i(t;c_i),p_i(t;c_i))$ is periodic with period $T_i(c_i)=4K(\sqrt{\frac{c_i}{2}})$ for $i=1,2$ and $c_i\in(0,2)$, where $K$ is the complete elliptic integral of the first kind.
	Then the corresponding frequency of the periodic orbit $(q_i(t;c_i),p_i(t;c_i))$ for $i=1,2$ is given by
	\[
	\omega_i(c_i)=\frac{2\pi}{T_i(c_i)}=\frac{\pi}{2K(\sqrt{\frac{c_i}{2}})}.
	\]

	Now we  compute
\[
\mathscr{A}_{F_i}^c(\tau,T)=\int_0^T dF_i(X^1)_{\gamma_\tau^c(t)}dt
\]
for $i=1,2$ with
\begin{align*}
X^1&=\{\alpha(q_2-q_1)+\delta_1 p_1+\mu_1 \cos  \phi \}\frac{\partial}{\partial p_1}\\
&\qquad+
\{\alpha(q_1-q_2)+\delta_2 p_2+\mu_2 \cos \theta\}\frac{\partial}{\partial p_2}.
\end{align*}

We describe how to compute $\mathscr{A}_{F_i}^c(\tau,T)$ only for $i=1$.
The case $i=2$ can be computed in a similar way.
 Note that $\mathscr{A}_{I}^c(\tau,T)$ is identically zero on $(\tau,T)$. 

Since $dF_1=\sin q_1dq_1+p_1dp_1$, we have
\begin{align*}
	&\mathscr{A}_{F_1}^c(\tau,T)=\int_0^T p_1\{\alpha(q_2-q_1)+\delta_1p_1+\mu_1 \cos\phi\}_{\gamma_\tau^c(t)}dt\\
	&=\alpha\int_0^T p_1(t+\tau_1;c_1)q_2(t+\tau_2;c_2)dt-\alpha\int_0^T p_1(t+\tau_1;c_1)q_1(t+\tau_1;c_1)dt\\
	&\quad+\delta_1\int_0^T \{p_1(t+\tau_1;c_1)\}^2 dt 
	+\mu_1 \int_0^T p_1(t+\tau_1;c_1)\cos ( \nu_1t+\tau_{3})dt.
\end{align*}
We now consider each term of $\mathscr{A}_{F_1}^c(\tau,T)$.
We first compute 
\begin{align*}
\int_0^T p_1(t+\tau_1;c_1)q_1(t+\tau_1;c_1)dt
=\int_0^T \frac{d q_1(t+\tau_1;c_1)}{dt} q_1(t+\tau_1;c_1)dt\\
=\frac{1}{2}[q_1(t+\tau_1;c_1)^2]_0^T= \frac{1}{2}\left(q_1(T+\tau_1;c_1)^2-q_1(\tau_1;c_1)^2\right),
\end{align*}
which is always bounded in $T$ for every $ (c,\tau)\in(0,2)^2\times(0,\infty)\times \Tset^4$ by the periodicity of $q_1(t;c_1)$.

Next, we consider the third term $\delta_1\int_0^T \{p_1(t+\tau_1;c_1)\}^2 dt$.
Let $D_1(c_1)$ be a domain in the $(q_1,p_1)$-plane
whose boundary is $z_1(\cdot;c_1)=(q_1(\cdot;c_1), p_1(\cdot;c_1))$ for $c_1\in (0,2)$. Then
\begin{align*}
\int_0^{T_1(c_1)} \{p_1(t+\tau_1;c_1)\}^2 dt
&=\int_0^{T_1(c_1)} p_1(t+\tau_1;c_1) \frac{d q_1(t+\tau_1;c_1)}{dt} dt\\
&=\int_{z_1(\cdot;c_1)} p_1 dq_1=\int_{D_1(c_1)} dp_1\wedge dq_1=\mathrm{Area}(D_1(c_1)),
\end{align*}
where $\mathrm{Area}(\cdot)$ stands for an area in $\Rset^2$.
Hence, for $T=l T_1(c_1)+T'$ with $l\in\Nset$ and $T'\in [0,T_1(c_1))$, we obtain
\[
	\int_0^{T} \{p_1(t+\tau_1;c_1)\}^2 dt=l\mathrm{Area}(D_1(c_1))+\int_0^{T'} \{p_1(t+\tau_1;c_1)\}^2 dt.
\]
From the monotonicity of $\int_0^{T} \{p_1(t+\tau_1;c_1)\}^2 dt$ with respect to $T$, 
we see that
\[
\left|\int_0^{T} \{p_1(t+\tau_1;c_1)\}^2 dt- \frac{\mathrm{Area}(D_1(c_1))}{T_1(c_1)}T\right|\le \mathrm{Area}(D_1(c_1))
\quad \text{for all}\ T>0,
\]
which implies that
\[
\int_0^{T} \{p_1(t+\tau_1;c_1)\}^2 dt\quad\sim\quad \frac{\mathrm{Area}(D_1(c_1))}{T_1(c_1)}T,
\]
where $\sim$ represents an equivalence relation which we introduced after Definition~\ref{dfn:unbdd_A}.

We compute the fourth term $\int_0^T p_1(t+\tau_1;c_1)\cos ( \nu_1t+\tau_{3})dt$.
Letting $u=t+\tau_1$, we have
\begin{align*}
	\int_0^T &p_1(t+\tau_1;c_1)\cos (\nu t+\tau_{3})dt=\int_{\tau_1}^{T+\tau_1} p_1(u;c_1)\cos ( \nu u+\tau_{3}-\nu \tau_1)du\\
	&= \int_{\tau_1}^{T+\tau_1} p_1(u;c_1)\cos (\nu u)du\cdot \cos(\tau_{3}-\nu \tau_1)
	\\&\qquad-\int_{\tau_1}^{T+\tau_1} p_1(u;c_1)\sin (\nu u)du\cdot \sin(\tau_{3}- \nu \tau_1)
\end{align*}
The Fourier expansion of $p_1(u;c_1)$ is given by
\begin{align*}
p_1(u;c_1)=2\sqrt{\frac{c_1}{2}}\cn\left(t;\sqrt{\frac{c_1}{2}}\right)
=\frac{4\pi}{K(\sqrt{\frac{c_1}{2}})}\sum_{l=0}^\infty \frac{ {\rho_1}^{l+1}}{1+ {\rho_1}^{2l+1}}
\cos\frac{(2l+1)\pi u}{2K(\sqrt{\frac{c_1}{2}})},
\end{align*}
where $ {\rho_i=\mathrm{exp}\left(\frac{-\pi K(\sqrt{1-c_i/2})}{K(\sqrt{c_i/2}))}\right)}$ is the elliptic nome.
Hence we obtain
\begin{align*}
\int_0^T &p_1(t+\tau_1;c_1)\cos (\nu_1 t+\tau_{3})dt
\\
&\sim 
\begin{cases}
\frac{2\pi  {\rho_1}^{l+1}T}{(1+ {\rho_1}^{2l+1})K(\sqrt{\frac{c_1}{2}})} \cos(\tau_{3}- \nu \tau_1)
& \text{if  {$\frac{\nu}{\omega(c_1)}=(2l+1)$} for some $l\in\Nset_0$}, \\
    0                 & \text{otherwise}. 
\end{cases}
\end{align*}

We now turn to estimate the first term $\alpha\int_0^T p_1(t+\tau_1;c_1)q_2(t+\tau_2;c_2)dt$.
It is difficult to obtain its explicit formula, so we borrow the idea presented in \cite{VH85}.
Let
\[
M^c(t_0;T)=\int_0^T p_1(t;c_1)q_2(t+t_0;c_2)dt.
\]
Then $M^c(t_0,T)$ is $T_2(c_2)$-periodic in $t_0$ for every $T>0$.
By the smoothness of $q_2$, we obtain
\[
\frac{\partial M^c}{\partial t_0}(t_0;T)=\int_0^T p_1(t;c_1)p_2(t+t_0;c_2)dt.
\]
Using the Fourier expansions of $p_i(t;c_i)$ for $i=1,2$, we obtain
\[
\frac{\partial M^c}{\partial t_0}(0;T)=\frac{8\pi^2 {\rho_1}^{m+1} {\rho_2}^{j+1}}{K(\sqrt{\frac{c_1}{2}})K(\sqrt{\frac{c_2}{2}})(1+ {\rho_1}^{2m+1})(1+ {\rho_2}^{2j+1})}T+\text{(bounded)}
\]
if
 $ {\omega(c_2)/\omega(c_1)}=(2m+1)/(2j+1)$
 for some $m,j\in\Nset_0$.
Moreover,
\[
M^c(0;4K(\sqrt{c_1/2})K(\sqrt{c_2/2}))=0
\]
 if $ {\omega(c_2)/\omega(c_1)}=(2m+1)/(2j+1)$ for some $m,j\in\Nset_0$ since $p_1$ is an even function and $q_2$ is an odd one.
Therefore, $M^c(t_0;T)$ is unbounded in $T$ and neither constant in $t_0$ nor of the form $a T+\text{(bounded)}$ with some constant $a$.

We introduce
\begin{align*}
V_0&=\left\{c=(c_1,c_2,c_3)\in(0,2)^2\times (0,\infty) ;\   {\frac{\omega(c_2)}{\omega(c_1)}}=\frac{2m+1}{2j+1}\ \text{for some} \ m,j\in\Nset_0\right\},\\
V_1&=\left\{\ c=(c_1,c_2,c_3)\in(0,2)^2\times (0,\infty) ; \  {\frac{\ \nu}{\omega(c_1)}=2l_1+1}\ \text{for some} \ l_1\in\Nset_0\right\},\\
V_2&=\left\{c=(c_1,c_2,c_3)\in(0,2)^2\times (0,\infty) ; \  {\frac{c_3}{\omega(c_2)}=2l_2+1}\ \text{for some} \ l_2\in\Nset_0\right\}.
\end{align*}
Note that $V_0,V_1,V_2$ are key sets.
Then, we see that
\begin{align*}
\mathscr{A}_{F_1}^c(\tau,T)\sim \alpha M^c(\tau_2-\tau_1;T){\bf{1}}_{V_0}(c)+\delta_1 \frac{\mathrm{Area}(D_1(c_1))}{T_1(c_1)}T+\mu_1\frac{2\pi  {\rho_1}^{l+1}T\cos(\tau_{3}-\nu\tau_1)}{(1+ {\rho_1}^{2l+1})K(\sqrt{\frac{c_1}{2}})} {\bf{1}}_{V_1}(c),\\
\mathscr{A}_{F_2}^c(\tau,T)\sim -\alpha M^c(\tau_2-\tau_1;T){\bf{1}}_{V_0}(c)+\delta_2 \frac{\mathrm{Area}(D_2(c_2))}{T_2(c_2)}T+\mu_2\frac{2\pi  {\rho_2}^{l+1}T\cos(\tau_{4}-c_3\tau_1)}{(1+ {\rho_2}^{2l+1})K(\sqrt{\frac{c_2}{2}})} {\bf{1}}_{V_2}(c),
 \end{align*}
where ${\bf{1}}_{V}(c)$ is the indicator function of a subset $V\subset (0,2)^2\times(0,\infty)$.
It is easy to see that if one of $\alpha,\delta_i,\mu_i \ (i=1,2)$ is not zero, then, for one of $j=0,1,2$, it holds that $\mathscr{A}_{F}^c$ has at least 1 unbounded direction for $c\in V_i$.
When $\alpha\neq 0$ and one of $\delta_i,\mu_i \ (i=1,2)$ is not zero, then $\mathscr{A}_{F}^c$ has 2 unbounded directions on at least one of $V_0\cap  V_1$ or $V_0\cap V_2$, which are key sets.
When $\alpha=0$ and one of $\mu_1\mu_2,\delta_1\mu_2,\delta_2\mu_1$ is not zero, then $\mathscr{A}_{F}^c$ has 2 unbounded directions for $c\in V_1\cap V_2$.
Using Theorem~\ref{thm:main1} and the fact that $I$ is the first integral of \eqref{eqn:pendulums}, we obtain the upper bound of the numbers of first integrals of \eqref{eqn:pendulums}.
\begin{prop}\label{prop:fi_pendla}
	For \eqref{eqn:pendulums}, we have the following:
	\begin{itemize}
		\item[(i)] If one of $\alpha,\delta_i,\mu_i (i=1,2)$ is not zero, then Eq.~ \eqref{eqn:pendulums} has at most $1$
	real-analytic non-constant first integral in $(q_1,p_1,q_2,p_2,\phi,\varepsilon)$  which is functionally independent of the trivial first integral $I$.
		\item[(ii)] If $\alpha\neq 0$ and one of $\delta_i,\mu_i (i=1,2)$ is not zero, then Eq.~ \eqref{eqn:pendulums} has no
	real-analytic non-constant first integral in $(q_1,p_1,q_2,p_2,\phi,\varepsilon)$  which is functionally independent of the trivial first integral $I$.
	\item[(iii)] If $\alpha=0$ and one of $\mu_1\mu_2,\delta_1\mu_2,\delta_2\mu_1$ is not zero, then Eq.~ \eqref{eqn:pendulums} has no
	real-analytic non-constant first integral in $(q_1,p_1,q_2,p_2,\phi,\varepsilon)$  which is functionally independent of the trivial first integral $I$.
	\end{itemize}
\end{prop}

\begin{rmk}\label{rmk:penduli_sharp}
    The case $\alpha\neq 0$ and $\delta_i=\mu_i=0 \ (i=1,2)$ satisfies
    the condition (i) in Proposition~\ref{prop:fi_pendla}.
    In this case, \eqref{eqn:pendulums} becomes a Hamiltonian system with Hamiltonian
    \[
    H_\varepsilon=F_1+F_2+\epsilon\frac{\alpha}{2}(q_1-q_2)^2,
    \]
    which implies that \eqref{eqn:pendulums} has an additional first integral other than $I$.
    Therefore, this case gives the sharpness of Theorem~\ref{thm:main1}.
\end{rmk}
We will next see that the parameter dependence of the upper bound of the numbers of commutative vector fields of \eqref{eqn:pendulums} is different from Proposition~\ref{prop:fi_pendla}. We have
\begin{align*}
d_\tau\mathscr{A}_{F_1}^c(\tau,T)\sim \left(-\alpha \frac{\partial M^c}{\partial t_0}(\tau_2-\tau_1;T){\bf{1}}_{V_0}(c)+\mu_1\frac{2\pi  {\rho_1}^{l+1}T\sin(\tau_{3}- \nu \tau_1)}{(1+ {\rho_1}^{2l+1})K(\sqrt{\frac{c_1}{2}})} {\bf{1}}_{V_1}(c)\right)d\tau_1\\
+\alpha \frac{\partial M^c}{\partial t_0}(\tau_2-\tau_1;T){\bf{1}}_{V_0}(c)d\tau_2-\mu_1\nu\frac{2\pi  {\rho_1}^{l+1}T\sin(\tau_{3}-\nu\tau_1)}{(1+ {\rho_1}^{2l+1})K(\sqrt{\frac{c_1}{2}})} {\bf{1}}_{V_1}(c)d\tau_3,\\
d_\tau\mathscr{A}_{F_2}^c(\tau,T)\sim \left(\alpha \frac{\partial M^c}{\partial t_0}(\tau_2-\tau_1;T){\bf{1}}_{V_0}(c)+\mu_2\frac{2\pi  {\rho_2}^{l+1}T\sin(\tau_{3}-\nu\tau_1)}{(1+ {\rho_2}^{2l+1})K(\sqrt{\frac{c_2}{2}})} {\bf{1}}_{V_2}(c)\right)d\tau_1\\
\ -\alpha \frac{\partial M^c}{\partial t_0}(\tau_2-\tau_1;T){\bf{1}}_{V_0}(c)d\tau_2-\mu_2\nu\frac{2\pi  {\rho_2}^{l+1}T\sin(\tau_{3}-\nu\tau_1)}{(1+ {\rho_2}^{2l+1})K(\sqrt{\frac{c_2}{2}})} {\bf{1}}_{V_2}(c)d\tau_3
\end{align*}
We apply Theorem~\ref{thm:main2} to Eq.~\eqref{eqn:pendulums}.
\begin{prop}\label{prop:cvf_pendla}
	For \eqref{eqn:pendulums}, we have the following:
	\begin{itemize}
		\item[(i)] If one of $\alpha, \mu_i (i=1,2)$ is not zero, then Eq.~ \eqref{eqn:pendulums} has at most $3$
		real-analytic commutative vector fields on $(q_1,p_1,q_2,p_2,\phi,\varepsilon)$ which are linearly independent near $\varepsilon=0$.
		\item[(ii)] If $\alpha\mu_1\neq 0$ or $\alpha\mu_2\neq 0$, then Eq.~ \eqref{eqn:pendulums} has at most 2 
		real-analytic commutative vector fields on $(q_1,p_1,q_2,p_2,\phi,\varepsilon)$  which are linearly independent near $\varepsilon=0$.
		\end{itemize}
\end{prop}

\begin{rmk}
    The above approach also works for other examples in which the unperturbed system has no homoclinic or heteroclinic orbits.
    For \eqref{eqn:pendulums}, using Theorems~\ref{thm:Melnikov1} and ~\ref{thm:Melnikov2}, we can obtain the same statements as in Propositions~\ref{prop:fi_pendla} and ~\ref{prop:cvf_pendla}, in which case we need to compute the homoclinic Melnikov vector.
    Note that the corresponding homoclinic Melnikov vector $M(\tau,\theta,c_3,\phi)=(M_1(\tau,\theta,c_3,\phi),M_2(\tau,\theta,c_3,\phi))$ is
    \begin{align*}
    M_1(\tau,\theta,c_3,\phi)=
    -\alpha\int_0^{\tau_2-\tau_1}\frac{8s}{\sinh s}ds-8\delta_1-\frac{2\pi\mu_1}{\cosh (\pi\nu/2)}\cos(\tau_3-\nu\tau_1),\\
    M_2(\tau,\theta,c_3,\phi)=\alpha\int_0^{\tau_2-\tau_1}\frac{8s}{\sinh s}ds-8\delta_2-\frac{2\pi\mu_2}{\cosh (\pi c_3/2)}\cos(\tau_4-c_3\tau_2).
    \end{align*}
\end{rmk}


\subsection{$N$-coupled pendulums}
Let $N\in\Nset$. Consider a Hamiltonian $H:\Rset^N\times\Rset^N\to \Rset$ with
\begin{align}\label{eqn:Ncoupled}
H_\varepsilon(q_1,\ldots,q_N,p_1,\ldots,p_N)=\sum_{i=1}^N \left\{\frac{p_i^2}{2}+(1-\cos q_i)\right\}+\varepsilon\sum_{j=1}^{N-1} \frac{\alpha_j}{2}(q_j-q_{j+1})^2,
\end{align}
where $\alpha_j\in\Rset\ (j=1,\ldots,N-1)$ are non-zero constants and $\varepsilon\ge 0$ is a perturbation parameter.
The corresponding Hamiltonian system is given by
\begin{equation}\label{eqn:coupled_Duffing}
	\begin{aligned}
		\dot{q}_1&=p_1,\quad
		\dot{p}_1=-\sin q_1-\varepsilon\alpha_1(q_1-q_2),\\
		\dot{q}_i&=p_i,\quad
		\dot{p}_i=-\sin q_i-\varepsilon\alpha_i(q_i-q_{i+1})-\varepsilon\alpha_{i-1}(q_i-q_{i-1}),\\
		\dot{q}_N&=p_N,\quad
		\dot{p}_N=-\sin q_1-\varepsilon\alpha_{N-1}(q_{N}-q_{N-1}),
	\end{aligned}
\end{equation}
with $i=2,\ldots, N-1$.
Note that $F_i=\frac{p_i^2}{2}+(1-\cos q_i)\ (i=1,\ldots,N)$ are pairwise Poisson commutative first integrals of $H_0$.

We take a family of quasi-periodic orbits $\gamma_\tau^c(t)=(q_1(t+\tau_1;c_1),\ldots,q_N(t+\tau_N;c_N),,p_1(t+\tau_1;c_1),\ldots, p_N(t+\tau_N;c_N))$ with 
\begin{align*}
	(q_i(t;c_i), p_i(t;c_i))=\left(2\arcsin\left(\sqrt{\frac{c_i}{2}}\sn\left(t;\sqrt{\frac{c_i}{2}}\right)\right), 2\sqrt{\frac{c_i}{2}}\cn\left(t;\sqrt{\frac{c_i}{2}}\right)\right),
	\end{align*}
	with $i=1,\ldots,N$ for $c=(c_1,\ldots, c_N)\in(0,2)^N$ and $\tau=(\tau_1,\ldots,\tau_N)\in\Tset^N$,
Let
\[
W_i=\left\{c=(c_1,\ldots,c_N)\in(0,2)^N ;\   {\frac{\omega(c_{i+1})}{\omega(c_i)}}=\frac{2m_i+1}{2j_{i+1}+1}\ \text{for some} \ m_i,j_{i+1}\in\Nset_0\right\},
\]
for $i=1,\ldots,{N-1}$, where $\omega(c_i)=\pi/(2K(\sqrt{c_i/2}))$.
Then we have
\begin{align*}
\mathscr{A}_{F_1}^c(\tau,T)\ \sim \ &\alpha_1 M^{(c_1,c_2)}(\tau_2-\tau_1;T){\bf{1}}_{W_1}(c),\\
\mathscr{A}_{F_i}^c(\tau,T)\ \sim \ &\alpha_i M^{(c_i,c_{i+1})}(\tau_{i+1}-\tau_i;T){\bf{1}}_{W_i}(c)+\alpha_i M^{(c_i,c_{i-1})}(\tau_{i-1}-\tau_i;T){\bf{1}}_{W_{i-1}}(c),\\
\mathscr{A}_{F_{N}}^c(\tau,T)\ \sim \ &\alpha_{N-1} M^{(c_{N-1},c_N)}(\tau_{N-1}-\tau_N;T){\bf{1}}_{W_{N-1}}(c),
\end{align*}
with $i=2,\ldots,N-1$. It is easy to verify that $\mathscr{A}_{F}^c(\tau,T)$ has $(N-1)$ unbounded directions if $c\in W:=\bigcap_{i=1}^{N-1} W_i$. Note that $W$ is a key set.
From Theorem~\ref{thm:main1}, we have the following.
\begin{prop}
	Let $\alpha_i\in\Rset\ (i=1,\ldots,N-1)$ be non-zero constant.
	Then the Hamiltonian $H_\varepsilon$ given by \eqref{eqn:Ncoupled} has no real-analytic first integral depending on $(q_1,\ldots,q_N,p_1,\ldots, p_N,\varepsilon)$ which is functionally independent of $H_\varepsilon$.
\end{prop}
It is easy to verify that $\{d_\tau\mathscr{A}_{F_i}^c\}_{i\in\{1,\ldots,N\}}$ has $(N-1)$ unbounded directions if $c\in W:=\bigcap_{i=1}^{N-1} W_i$. Using Theorem~\ref{thm:main2}, we also obtain an obstruction to existence of additional commutative vector fields.
\begin{prop}
	Let $\alpha_i\in\Rset\ (i=1,\ldots,N-1)$ be non-zero constant.
	Then the Hamiltonian $H_\varepsilon$ given by \eqref{eqn:Ncoupled} has no
		real-analytic commutative vector fields in $(q_1,\ldots,q_N,p_1,\ldots, p_N,\varepsilon)$ which is linearly independent of the Hamiltonian vector field of $H_\epsilon$ near $\varepsilon=0$.
		\end{prop}



\setcounter{equation}{0}
\renewcommand{\theequation}{\Alph{section}.\arabic{equation}}

\appendix

\section{An integrable Hamiltonian system with unbounded integrals}\label{app:example}

In this appendix, we provide an example of integrable Hamiltonian system with unbounded integrals.
Consider a Hamiltonian $H_\varepsilon:\Rset^n\times \Tset^n\to\Rset$ given by
\begin{align}\label{eq:Appendix}
	H_\varepsilon(I,\theta)=h(I)+\epsilon \cos(k\cdot \theta),
\end{align}
where $h:\Rset^n\to\Rset$ is an analytic function,  $k\in\Zset^n\setminus\{0\}$ and $\varepsilon\ge 0$.
It is trivial that $H_\varepsilon(I,\theta)$ is Liouville-integrable when $\epsilon=0$.
Let
\[
W:=(\mathrm{span}_{\Rset}\{k\})^\perp =\{w\in\Rset^n; k\cdot w=0\}.
\]
Take a basis $\{w_1,\ldots,w_{n-1}\}$ of $W$. Let $F_i(I,\theta)\ (i=1,\ldots,n-1)$ be a function on $\Rset^n\times \Tset^n$ defined as
\[
	F_i(I,\theta)=w_i\cdot I.
\]
Then it is easy to see that $H_\varepsilon$ is Liouville-integrable for all $\epsilon>0$ with pair-wise Poisson commutative and functionally independent first integrals $H_\epsilon, F_1,F_2,\ldots, F_{n-1}$.
Set
\[
	R_k:=\left\{I\in\Rset^n; k\cdot \frac{\partial h}{\partial I}(I)=0\right\}.
\]
We assume that
\begin{itemize}
\item[(B)] $R_k$ is a non-empty proper subset of $\Rset^n$.
\end{itemize}
Since $h$ is analytic, the function $k\cdot \frac{\partial h}{\partial I}(I)$ is also analytic, and thus $R$ is an analytic variety which is possibly uncountable but not the whole space $\Rset^n$.
For $I^*\in R_k$, we compute our integral
\begin{align*}
	\int_0^T &dI(X_{H^1})_{(I^*, \frac{\partial h}{\partial I}(I^*)t+\tau)}dt=\left(\int_0^T \sin\left(k\cdot \frac{\partial h}{\partial I}(I^*) t +k\cdot \tau\right) dt \right)k\\
	&=\left\{\cos(k\cdot\tau) \int_0^T \sin\left(k\cdot \frac{\partial h}{\partial I}(I^*) t\right) dt +\sin(k\cdot\tau) \int_0^T \cos\left(k\cdot \frac{\partial h}{\partial I}(I^*) t\right) dt \right\}k\\
	&=T\sin(k\cdot\tau) k
\end{align*}
therefore, we have $R_k\subset \mathscr{Q}_1$. Note that $R_k$ is not a key set since if not $R_k=\Rset^n$ by the definition of a key set but this cannot occur.

\section{Conformally symplectic and partially conformally  symplectic systems}
\label{sec:partially_conformal}

We recall conformally symplectic systems \cite{Vaisman85} and partially conformally symplectic systems \cite{CallejaCL13}. 

Let $(\mathscr{M},\Omega)$ be a $2m$-dimensional, symplectic manifold.
Assume that $\Omega$ is exact, i.e., $\Omega=d\alpha$  for some $1$-form $\alpha$ (the Liouville form).
In Darboux local coordinates, we have $\Omega=dp\wedge dq$ and $\alpha=pdq$, with $(p,q)\in\mathbb{R}^m\times\mathbb{R}^m$.

A vector field $X$ on $M$ is conformally symplectic provided
\[\mathcal{L}_X\Omega= -\lambda\Omega\] for some $\lambda\neq 0$, where $\mathcal{L}_X$ denotes the Lie derivative in the direction of $X$.
The case when $\lambda>0$  corresponds to a  dissipative vector field,   $\lambda=0$  to a symplectic vector field,
and $\lambda<0$  to an anti-dissipative (energy-injecting) vector field.

If $X$ is conformally symplectic, also $-X$ is conformally symplectic and
$\mathcal{L}_{-X}\Omega  =  \lambda\Omega$. Hence, up to a time-inversion of the flow, one can always choose the
sign of $\lambda$. Hereafter, we assume the dissipative case $\lambda>0$.

Using Cartan's formula an the closedness of $\Omega$, we have
\[\mathcal{L}_X\Omega= d(\iota_X \Omega) +\iota_X(d\Omega) =d(\iota_X \Omega),\]
where  $\iota_X$ denotes the interior product with $X$. Using the exactness of $\Omega$, it follows that
\[d(\iota_X\Omega+\lambda\alpha)=0,\]
so $\iota_X\Omega+\lambda \alpha$ is a closed $1$-form.
The vector field $X$ is exact conformally symplectic if this closed $1$-form is exact, i.e.,
\[\iota_X\Omega+\lambda \alpha=-dH,\]
for some Hamiltonian function $H:M\to\mathbb{R}$.
In the local coordinates, the exactness condition becomes
\[ \dot p dq-\dot q dp +\lambda pdq=-\frac{\partial H}{\partial p}dp-\frac{\partial H}{\partial q} dq,\]
hence
\[\left\{
    \begin{array}{ll}
      \dot p =&-\frac{\partial H}{\partial q}-\lambda p   \\[1em]
      \dot q  =& \frac{\partial H}{\partial p} .
    \end{array}
  \right.
\]
If we denote $z=(p,q)$, the equations of motion can be written as
\[ \dot z= X_H (z)+ Y(z) \]
where $X_H=JDH$ is the Hamiltonian vector field and
\[Y(z)=-\lambda  p \cdot \frac{\partial}{\partial p}.\]

Now, let us assume that  $(\mathscr{M}_i,\Omega_i)$ are $(2m_i)$-dimensional, symplectic manifolds, for $i=1,\ldots,k$.
Consider the product
\[(\mathscr{M},\Omega):=(\mathscr{M}_1\times\ldots \times \mathscr{M}_k,\Omega_1\oplus \ldots \oplus\Omega_k),\]
where each $\Omega_i$ is exact with $\Omega_i=d\alpha_i$.
A vector field $X$ on $\mathscr{M}$ is quasi-conformally symplectic provided that
\[\mathcal{L}_X\Omega=(-\lambda_1\Omega_1)\oplus\ldots \oplus(-\lambda_k\Omega_k)\] for some $\lambda_1,\ldots,\lambda_k\ge 0$, where we assume that at least one $\lambda_i>0$.
Similarly to above, this condition is equivalent to that
$\iota_X\Omega+( \lambda_1\alpha_1)\oplus\ldots \oplus( \lambda_k\alpha_k)$ is a closed form.
The vector field $X$ is exact partially conformally symplectic if this closed $1$-form is exact, i.e.,
\[\iota_X\Omega+( \lambda_1\alpha_1)\oplus\ldots \oplus( \lambda_k\alpha_k)=-dH,\]
for some $H:\mathscr{M}\to \mathbb{R}$.
In local coordinates, this condition becomes
\[ \bigoplus_{i=1}^{k} \left(\dot p_i dq_i-\dot q_i dp_i +\lambda_i p_idq_i\right)= \bigoplus_{i=1}^{k} \left(-\frac{\partial H}{\partial p_i}dp_i-\frac{\partial H}{\partial q_i} dq_i \right),\] hence
\[\left\{
    \begin{array}{ll}
      \dot p_i =&-\frac{\partial H}{\partial q_i}-\lambda_i p_i \\[1em]
      \dot q_i  =& \frac{\partial H}{\partial p_i} .
    \end{array}
  \right., \quad \textrm{ for } i=1,\ldots, k.
\]
If we denote $z=(p_1,
\ldots,p_k,q_1, \ldots, ,q_k)$, the equations of motion can be written as
\[ \dot z= X_H (z)+ Y(z) \]
where $X_H$ is the Hamiltonian vector field and
\[Y(z)=\sum_{i=1}^{k}-\lambda_ip_i \frac{\partial}{\partial p_i}.\]

Suppose that $H:\mathscr{M}\times \mathbb{R}\to\mathbb{R}$ is a time-dependent Hamiltonian on $M$.
The standard recipe to transform it into an autonomous Hamiltonian is by making the time into an independent variable and introducing a new variable
conjugate with time:
\[\hat H(p,q,A,t)=H(p,q,t)+A.\]
The conformally symplectic case amounts to adding to the Hamiltonian vector field a dissipative vector field of the form
\[Y(p,q,A,t)=-\lambda p \cdot\frac{\partial}{\partial p}-\lambda A\frac{\partial}{\partial A}.\]
The corresponding equations of motions are
\[\left\{
    \begin{array}{ll}
      \dot p =&-\frac{\partial H}{\partial q}-\lambda p   \\[1em]
      \dot q  =& \frac{\partial H}{\partial p}  \\[1em]
      \dot A=& -\frac{\partial H}{\partial t}-\lambda A\\[1em]
      \dot t  =& 1.
    \end{array}
  \right.
\]
The equation for $A$ is independent of the other equations hence plays no role in the dynamics. 
For this reason, we can alternatively add to the Hamiltonian vector field a dissipative vector field of the form
\[Y(p,q,A,t)=-\lambda_1 p \cdot\frac{\partial}{\partial p}-\lambda_2 A\frac{\partial}{\partial A},\]
obtaining a partially conformally symplectic system.
%

\vspace*{33pt}

\noindent
\textbf{Acknowledgement.}~ 
Research of M.G. was partially supported by NSF grants DMS-2154725 and DMS-2307718.
Research of S.M. was partially supported by JSPS KAKENHI Grant Number 25K17288.

\vspace{11pt }
\noindent
\textbf{Data Availability.}~
Data sharing not applicable to this article as no datasets were generated or analyzed during the current study.



\begin{thebibliography}{99}





\bibitem{A89}
V.I.~Arnold,
\textit{Mathematical Methods of Classical Mechanics}, 2nd ed.,
Springer, New York, 1989.




\bibitem{AZ10}
M.~Ayoul and N.T.~Zung,
Galoisian obstructions to non-Hamiltonian integrability,
\textit{C. R. Math. Acad. Sci. Paris}, 
\textbf{348} (2010), 
1323--1326.


\bibitem{B98}
O.I.~Bogoyavlenskij,
Extended integrability and bi-hamiltonian systems,
\textit{Comm. Math. Phys.}, \textbf{196} (1998), 
19--51.


\bibitem{bounemoura2016note}
A.~Bounemoura and V.~Kaloshin,
A note on micro-instability for Hamiltonian systems close to integrable,
\textit{Proc. Amer. Math. Soc.}, \textbf{144} (2016), 
1553--1560.




\bibitem{CallejaCL13}
R.C.~Calleja, A.~Celletti and R.~de la Llave,
A KAM theory for conformally symplectic systems: Efficient algorithms and their validation,
\textit{J. Differential Equations}, \textbf{255} (2013), 
978--1049.


\bibitem{delshams2000splitting}
M. Delshams and P. Guti\'errez,
Splitting potential and the Poincar\'e--Melnikov method for whiskered tori in Hamiltonian systems,
\textit{J. Nonlinear Sci.}, \textbf{10} (2000), 
433--476.


\bibitem{delshams2008geometric}
A.~Delshams, R.~de la Llave and T.M.~Seara,
Geometric properties of the scattering map of a normally hyperbolic invariant manifold,
\textit{Adv. Math.}, \textbf{217} (2008), 
1096--1153.


\bibitem{GL18}
M.~Gidea, R.~de la Llave,
Global Melnikov theory in Hamiltonian systems with general time-dependent perturbations,
\textit{J. Nonlinear Sci.}, \textbf{28} (2018), 1657--1707.


\bibitem{gidea2020general}
M.~Gidea, R.~de la Llave and M.T.~Martinez-Seara,
A general mechanism of instability in Hamiltonian systems: Skipping along a normally hyperbolic invariant manifold,
\textit{Discrete Contin. Dyn. Syst. Ser. A}, \textbf{40} (2020), 
6795--6813.


\bibitem{GH83}
J.~Guckenheimer and P.~Holmes,
\textit{Nonlinear Oscillations, Dynamical Systems, and Bifurcations of Vector Fields},
Springer, New York, 1983.


\bibitem{HPS77}
M.W.~Hirsch, C.C.~Pugh and M.~Shub,
\textit{Invariant Manifolds},
Springer, Berlin, 1977.




\bibitem{K83}
V.V.~Kozlov,
Integrability and non-integarbility in Hamiltonian mechanics,
\textit{Russian Math. Surveys},
\textbf{38} (1983), 1--76.


\bibitem{K96}
V.V.~Kozlov,
\textit{Symmetries, Topology and Resonances in Hamiltonian Mechanics},
Springer, Berlin, 1996.


\bibitem{M63}
V.K.~Melnikov,
On the stability of the center for time periodic perturbations,
\textit{Trans. Moscow Math. Soc.}, \textbf{12} (1963), 
1--56.


\bibitem{MR01a}
J.J.~Morales-Ruiz and J.P.~Ramis,
Galoisian obstructions to integrability of Hamiltonian systems. I,
\textit{Methods and Applications of Analysis}, \textbf{8} (2001), no.~1,
33--96.


\bibitem{MR01b}
J.J.~Morales-Ruiz and J.P.~Ramis,
Galoisian obstructions to integrability of Hamiltonian systems. II,
\textit{Methods and Applications of Analysis}, \textbf{8} (2001), no.~1,
97--112.


\bibitem{MRS07}
J.J.~Morales-Ruiz, J.P.~Ramis and C.~Sim{\'o},
Integrability of Hamiltonian systems and differential Galois groups of higher variational equations,
\textit{Annales Scientifiques de l'{\'E}cole Normale Sup{\'e}rieure}, \textbf{40} (2007), no.~6,
845--884.


\bibitem{M24}
S.~Motonaga,
Subharmonic Melnikov functions and nonintegrability for autonomous and non-autonomous perturbations of single-degree-of-freedom Hamiltonian systems near periodic orbits,
\textit{Phys. D}, \textbf{460} (2024), 134088.


\bibitem{MY21}
S.~Motonaga and K.~Yagasaki,
Persistence of periodic and homoclinic orbits, first integrals and commutative vector fields
in dynamical systems,
\textit{Nonlinearity}, \textbf{34} (2021), 7574--7608.


\bibitem{MY23}
S.~Motonaga and K.~Yagasaki,
Obstructions to integrability of nearly integrable dynamical systems near regular level sets,
\textit{Arch. Ration. Mech. Anal.}, \textbf{247} (2023), 44.


\bibitem{M73}
J.~Moser,
\textit{Stable and Random Motions in Dynamical Systems:
With Special Emphasis on Celestial Mechanics},
Princeton Univ. Press, Princeton, NJ,
1973.


\bibitem{P90}
H.~Poincar\'e,
Sur le probl\'eme des trois corps et les \'equations de la dynamique,
\textit{Acta Math.}, \textbf{13} (1890), 1--270;
English translation:
\textit{The Three-Body Problem and the Equations of Dynamics},
Translated by D.~Popp, Springer, Cham, Switzerland, 2017.


\bibitem{P92}
H.~Poincar\'e,
\textit{New Methods of Celestial Mechanics}, Vol.~1,
AIP Press, New York, 1992 (original 1892).


\bibitem{Vaisman85}
I.~Vaisman,
Locally conformal symplectic manifolds,
\textit{Internat. J. Math. Math. Sci.}, \textbf{8} (1985), 
521--536.


\bibitem{VH85}
P.~Veerman and P.~Holmes,
The existence of arbitrarily many distinct periodic orbits in a two degree of freedom Hamiltonian system.
\textit{Phys. D}, \textbf{14} (1985), 177--192.


\bibitem{W90}
S.~Wiggins,
\textit{Introduction to Applied Nonlinear Dynamical Systems and Chaos}, 
Springer, New York, 
1990.


\bibitem{Y96}
K.~Yagasaki,
The Melnikov theory for subharmonics and their bifurcations in forced oscillations,
\textit{SIAM J. Appl. Math.}, \textbf{56} (1996), 
1720--1765.


\bibitem{Ziglin82}
S.L.~Ziglin,
Branching of solutions and nonexistence of first integrals in Hamiltonian mechanics. I,
\textit{Functional Analysis and Its Applications}, \textbf{16} (1982), no.~3, 181--189.


\bibitem{Ziglin83}
S.L.~Ziglin,
Branching of solutions and nonexistence of first integrals in Hamiltonian mechanics. II,
\textit{Functional Analysis and Its Applications}, \textbf{17} (1983), no.~1, 6--17.


\bibitem{Z06}
N.T.~Zung,
Torus actions and integrable systems,
in \textit{Topological Methods in the Theory of Integrable Systems},
A.V.~Bolsinov, A.T.~Fomenko and A.A.~Oshemkov (eds.),
Cambridge Scientific Publications, Cambridge, 2006, pp.~289--328.


\bibitem{Z18}
N.T.~Zung,
A conceptual approach to the problem of action-angle variables,
\textit{Arch. Ration. Mech. Anal.}, \textbf{229} (2018), 
789--833.

\end{thebibliography}
\end{document}